\documentclass[twoside,11pt]{article}
\usepackage{jmlr2e}
\usepackage[OT1]{fontenc}

\usepackage{amsthm,amsmath,amsfonts,amssymb}
\usepackage[colorlinks,citecolor=blue,linkcolor = red, urlcolor=green]{hyperref}
\usepackage{url}
\usepackage{dsfont}
\usepackage{graphics,graphicx}

\numberwithin{equation}{section}
\newtheorem{defn}{Definition}[section]
\newtheorem{lem}{Lemma}[section]
\newtheorem{prop}{Proposition}[section]
\newtheorem{cor}{Corollary}[section]

\newcommand{\prob}[1]{\mathbb{P}\left(#1 \right)}
\newcommand{\probi}[2]{\mathbb{P}_{#1}\left( #2 \right)}
\newcommand{\probc}[2]{\mathbb{P}\left(#1 \ | \ #2 \right)}

\newcommand{\esp}[1]{\mathbb{E}\left[#1 \right]}

\newcommand{\som}[2]{\sum_{#1}^{#2}}
\newcommand{\f}[1]{ \widehat{f}_k( #1 ) }
\newcommand{\fe}[1]{ \widehat{f}_k^e(#1 ) }
\newcommand{\fep}[1]{ \widehat{f}_k^{e^\prime}(#1 ) }

\newcommand{\cb}[2]{(_{#2}^{#1})}
\newcommand{\gp}[1]{\left(#1\right)}
\newcommand{\ga}[1]{\left\{#1\right\}}
\newcommand{\gc}[1]{\left[#1\right]}

\newcommand{\floor}[1]{\left\lfloor #1 \right\rfloor}

\newcommand{\1}{\mathds{1}}
\newcommand{\ind}[1]{\1_{\ga{#1}}}

\newcommand{\paren}[1]{\left( #1 \right)}
\newcommand{\croch}[1]{\left[\, #1 \,\right]}
\newcommand{\acc}[1]{\left\{ #1 \right\}}

\newcommand{\norm}[1]{\left\| #1 \right\|}

\newcommand{\abs}[1]{\left\lvert #1 \right\rvert} 

\newcommand{\defegal}{:=} 
\newcommand{\iid}{\textit{i.i.d.}\ }

\newcommand{\F}{\mathcal{F}}
\newcommand{\R}{\mathbb{R}}
\newcommand{\N}{\mathbb{N}}
\newcommand{\E}{\mathbb{E}}
\renewcommand{\P}{\mathbb{P}}
\newcommand{\Var}{\mathrm{Var}}

\DeclareMathOperator{\card}{Card} 

\newcommand{\X}{\mathcal{X}}

\newcommand{\Rh}{\widehat{R}}

\newcommand{\Rhp}{\Rh_p}

\newcommand{\D}{\mathcal{D}}
\newcommand{\Dn}{\mathcal{D}_n}
\newcommand{\De}{\mathcal{D}^e}
\newcommand{\Deb}{\mathcal{D}^{\bar e}}

\newcommand{\A}{\mathcal{A}}

\newcommand{\eb}{\bar e}

\newcommand{\Enp}{\mathcal{E}_{n-p}}

\numberwithin{equation}{section}
\theoremstyle{plain}
\newtheorem{thm}{Theorem}[section]

\jmlrheading{1}{2000}{1-48}{4/00}{10/00}{A. Celisse and T. Mary-Huard}

\ShortHeadings{Performance of CV to estimate the risk of $k$NN}{Celisse and Mary-Huard}
\firstpageno{1}

\begin{document}

\title{Theoretical analysis of cross-validation for estimating the risk of the $k$-Nearest Neighbor classifier}
%
%
%

\author{Alain Celisse \\
Laboratoire de Math\'ematiques \\
\textsc{Modal}  Project-Team \\
UMR 8524 CNRS-Universit\'e Lille 1 \\
F-59\,655 Villeneuve d'Ascq Cedex, France\\
\texttt{celisse@math.univ-lille1.fr} \\
\AND
Tristan Mary-Huard\\
INRA, UMR 0320 / UMR 8120 G\'en\'etique V\'eg\'etale et \'Evolution\\
Le Moulon, F-91190 Gif-sur-Yvette, France\\
UMR AgroParisTech INRA MIA 518, Paris, France\\
16 rue Claude Bernard\\
F-75\,231 Paris cedex 05, France \\
\texttt{maryhuar@agroparistech.fr} \\
}


\editor{TBW}

\maketitle

\begin{abstract}
The present work aims at deriving theoretical guaranties on the behavior of some cross-validation procedures applied to the $k$-nearest neighbors ($k$NN) rule in the context of binary classification.
Here we focus on the leave-$p$-out cross-validation (L$p$O) used to assess the performance of the $k$NN classifier. Remarkably this L$p$O estimator can be efficiently computed in this context using closed-form formulas derived by \cite{CelisseMaryHuard11}.

We describe a general strategy to derive moment and exponential concentration inequalities for the L$p$O estimator applied to the $k$NN classifier.
%
%
Such results are obtained first by exploiting the connection between the L$p$O estimator and U-statistics, and second by making an intensive use of the generalized Efron-Stein inequality applied to the L$1$O estimator.
One other important contribution is made by deriving new quantifications of the discrepancy between the L$p$O estimator and the classification error/risk of the $k$NN classifier. The optimality of these bounds is discussed by means of several lower bounds as well as simulation experiments.
%
%

%
%
%

\end{abstract}



\section{Introduction} \label{Section: Introduction}

The $k$-nearest neighbor ($k$NN) algorithm \citep{FixHodges51} in binary classification is a popular prediction algorithm based on the idea that the predicted value at a new point is based on a majority vote from the $k$ nearest labeled neighbors of this point.
Although quite simple, the $k$NN classifier has been successfully applied to many difficult classification tasks \citep{Li04,Simard98,Scheirer03}. Efficient implementations have been also developed to allow dealing with large datasets \citep{Indyk98,Andoni06}.

The theoretical performances of the $k$NN classifier have been already extensively investigated.
In the context of binary classification preliminary theoretical results date back to \cite{Cover_Hart:1967,Cover68,Gyorfi_1981}. More recently, \cite{Psaltis_Snapp_Venkatesh:1994,KulkPosner:1995} derived an asymptotic equivalent to the performance of the 1NN classification rule, further extended to $k$NN by \cite{Snapp_Venkatesh:1998}.
\cite{Hall_Park_Samworth:2008} also derived asymptotic expansions of the risk of the $k$NN classifier assuming either a Poisson or a binomial model for the training points, which relates this risk to the parameter $k$.
By contrast to the aforementioned results, the work by \cite{Chaudhuri_Dasgupta2014} focuses on the finite sample framework. They typically provide upper bounds with high probability on the risk of the $k$NN classifier where the bounds are not distribution-free.
Alternatively in the regression setting, \cite{KulkPosner:1995} provide a finite-sample bound on the performance of 1NN that has been further generalized to the $k$NN rule ($k\geq 1$) by \cite{Biau_Cerou_Guyader:2010}, where a bagged version of the $k$NN rule is also analyzed and then applied to functional data \cite{Biau_Cerou_Guyader:2010_function}.
We refer interested readers to \cite{BiauDevroye_2016} for an almost thorough presentation of known results on the $k$NN algorithm in various contexts.

\medskip

In numerous (if not all) practical applications, computing the cross-validation (CV) estimator \citep{Ston74,Sto:1982} has been among the most popular strategies to evaluate the performance of the $k$NN classifier \citep[][Section~24.3]{DeGyLu_1996}.
%
%
All CV procedures share a common principle which consists in splitting a sample of $n$ points into two disjoint subsets called \emph{training} and \emph{test} sets with respective cardinalities $n-p$ and $p$, for any $1\leq p \leq n-1$. The $n-p$ training set data serve to compute a classifier, while its performance is evaluated from the $p$ \emph{left out} data of the test set.
For a complete and comprehensive review on cross-validation procedures, we refer the interested reader to \cite{ArCe_2010_survey}.

In the present work, we focus on the leave-$p$-out (L$p$O) cross-validation.
Among CV procedures, it belongs to exhaustive strategies since it considers (and averages over) all the ${n\choose p}$ possible such splittings of $\acc{1,\ldots,n}$ into training and test sets.
Usually the induced computation time of the L$p$O is prohibitive, which gives rise to its surrogate called $V-$fold cross-validation (V-FCV) with $V\approx n/p$ \citep{Gei:1975}.
However, \cite{Steele_2009,CelisseMaryHuard11} recently derived closed-form formulas respectively for the bootstrap and the L$p$O procedures applied to the $k$NN classification rule. Such formulas allow one to efficiently compute the L$p$O estimator. Moreover since the V-FCV estimator suffers a larger variance than the L$p$O one \citep{CeRo08,ArCe_2010_survey}, L$p$O (with $p=\floor{n/V}$) strictly improves upon V-FCV in the present context.

\medskip

Although being favored in practice for assessing the risk of the $k$NN classifier,
the use of CV comes with very few theoretical guarantees regarding its performance.
Moreover probably for technical reasons, most existing results apply to Hold-out and leave-one-out (L1O), that is L$p$O with $p=1$ \citep{KeRo99}.
In this paper we rather consider the general L$p$O procedure (for $1\leq p \leq n-1$) used to estimate the risk (alternatively the classification error rate) of the $k$NN classifier. Our main purpose is then to provide distribution-free theoretical guarantees on the behavior of L$p$O with respect to influential parameters such as $p$, $n$, and $k$. For instance we aim at answering questions such as: ``Does it exist any regime of $p=p(n)$ (with $p(n)$ some function of $n$) where the L$p$O estimator is a consistent estimate of the risk of the $k$NN classifier?'', or ``Is it possible to describe the convergence rate of the L$p$O estimator with respect to $p/n$?''
%
%
%

\paragraph{Contributions.}
The main contribution of the present work is two-fold: $(i)$ we describe a new general strategy to derive moment and exponential concentration inequalities for the L$p$O estimator applied to the $k$NN binary classifier, and $(ii)$ these inequalities serve to derive the convergence rate of the L$p$O estimator towards the risk of the $k$NN classifier.

This new strategy relies on several steps.
First exploiting the connection between the L$p$O estimator and U-statistics \citep{Ko_Bo:1994} and the Rosenthal inequality \citep{IbragShar2002}, we prove that upper bounding the polynomial moments of the centered L$p$O estimator reduces to deriving such bounds for the simpler L1O estimator.
Second, we derive new upper bounds on the moments of the L1O estimator using the generalized Efron-Stein inequality \citep[][Theorem~15.5]{Bou_Bou_Lug_Mas:2005,BouLugMas_2013}.
Third, combining the two previous steps provides some insight on the interplay between $p/n$ and $k$ in the concentration rates measured in terms of moments.
This finally results in new exponential concentration inequalities for the L$p$O estimator applying whatever the value of the ratio $p/n \in (0,1) $. In particular while the upper bounds increase with $1\leq p\leq n/2+1$, it is no longer the case if $p>n/2+1$.
We also provide several lower bounds suggesting our upper bounds cannot be improved in some sense in a distribution-free setting. 

\medskip

The remainder of the paper is organized as follows.
The connection between the L$p$O estimator and $U$-statistics is clarified in Section~\ref{Section: Context}, where we also recall the closed-form formula of the L$p$O estimator \citep{CelisseMaryHuard11} applied to the $k$NN classifier. Order-$q$ moments ($q\geq 2$) of the L$p$O estimator are then upper bounded in terms of those of the L1O estimator. This step can be applied to any classification algorithm. 
Section~\ref{Section: Polynomial} then specifies the previous upper bounds in the case of the $k$NN classifier, which leads to the main Theorem~\ref{prop.moment.upper.bounds.Lpo} characterizing the concentration behavior of the L$p$O estimator with respect to $p$, $n$, and $k$ in terms of polynomial moments.
Deriving exponential concentration inequalities for the L$p$O estimator is the main concern of Section~\ref{Section: Deviations} where we highlight the strength of our strategy by comparing our main inequalities with concentration inequalities derived with less sophisticated tools.
Finally Section~\ref{Section:Old} exploits the previous results to bound the gap between the L$p$O estimator and the classification error of the $k$NN classifier. The optimality of these upper bounds is first proved in our distribution-free framework by establishing several new lower bounds matching the upper ones in some specific settings. Second, empirical experiments are also reported which support the above conclusions.

\section{$U$-statistics and L$p$O estimator} \label{Section: Context}


\subsection{Statistical framework}

\paragraph{Classification}
We tackle the binary classification problem where the goal is to predict the unknown label $Y\in \{0, 1\}$ of an observation $X\in \mathcal{X} \subset \R^d$. The random variable $(X,Y)$ has an \emph{unknown} joint distribution $P_{(X,Y)}$ defined by $P_{(X,Y)}(B) = \P\croch{ (X,Y)\in B}$ for any Borelian set $B \in  \X\times \acc{0,1}$, where $\P$ denotes a probability distribution. In what follows no particular distributional assumption is made regarding $X$.
To predict the label, one aims at building a classifier $\hat f:\mathcal{X}\rightarrow\{0,1\}$ on the basis of a set of random variables $ \D_n = \acc{Z_1,\ldots,Z_n}$ called the training sample, where $Z_i=(X_i,Y_i),\ 1\leq i\leq n$ represent $n$ copies of $(X,Y)$ drawn independently from $P_{(X,Y)}$. In settings where no confusion is possible, we will replace $\D_n $ by $\D$.

Any strategy to build such a classifier is called a \emph{classification algorithm} or \emph{classification rule}, and can be formally defined as a function $\A:\ \cup_{n\geq1} \acc{\X\times \acc{0,1}}^n \to \F$ that maps a training sample $\D_n$ onto the corresponding classifier $ \A^{\D_n}\paren{ \cdot}= \hat f \in \F$, where $\F$ is the set of all measurable functions from $\X$ to $\acc{0,1}$.
%
%
Numerous classification rules have been considered in the literature and it is out of the scope of the present paper to review all of them (see \cite{DeGyLu_1996} for many instances).
Here we focus on the $k$-nearest neighbor rule ($k$NN) initially proposed by \cite{FixHodges51} and further studied for instance by \cite{Dev_Wag:1977,RogersWagner78}.

\paragraph{The $k$NN algorithm}
For $1\leq k\leq n$, the $k$NN rule, denoted by $\A_k$, consists in classifying any new observation $x$ using a \emph{majority vote} decision rule based on the label of the $k$  points $X_{(1)}(x),\ldots,X_{(k)}(x)$ closest to $x$ among the training sample $X_{1},\ldots,X_n$. 
In what follows these $k$ \emph{nearest neighbors} are chosen according to the distance associated with the usual Euclidean norm in $\R^d$.
Note that other \emph{adaptive metrics} have been also considered in the literature \citep[see for instance][Chap.~14 ]{Has_Tib_Fri:2001}. But such examples are out of the scope of the present work that is, our reference distance does not depend on the training sample at hand.
Let us also emphasize that possible ties are broken by using the \emph{smallest index} among ties, which is one possible choice for the Stone lemma to hold true \citep[][Lemma~10.6, p.125]{BiauDevroye_2016}.

Formally, given $V_k(x) = \acc{1\leq i\leq n,\ X_i \in \acc{X_{(1)}(x),\ldots,X_{(k)}(x)}}$ the set of indices of the $k$ nearest neighbors of $x$ among $X_1,\ldots,X_n$, the kNN classification rule is defined by
\begin{eqnarray} \label{def.knn.classifier}
\A_k(\D_n;x) = \widehat{f}_k(\D_n;x) \defegal
\left\{
\begin{tabular}{cl}
1 &, if $\frac{1}{k}\sum_{i \in V_k(x)} Y_i =\frac{1}{k}\sum_{i =1}^k Y_{(i)}(x) > 0.5$ \\
0 &, if $\frac{1}{k}\sum_{i =1}^k Y_{(i)}(x) < 0.5$\\
\mbox{$\mathcal{B}$}(0.5) &, \mbox{otherwise}
\end{tabular}
\right. \ \ ,
\end{eqnarray}
where $Y_{(i)}(x)$ is the label of the $i$-th nearest neighbor of $x$ for $1\leq i \leq k$, and $\mathcal{B}(0.5)$ denotes a Bernoulli random variable with parameter 1/2.

\paragraph{Leave-$p$-out cross-validation}
For a given sample $\D_n$, the performance of any classifier $\hat f=\A^{\Dn}(\cdot)$ (respectively of any classification algorithm $\AA$) is assessed by the classification error $L(\hat f)$ (respectively the risk $R(\hat f)$) defined by
\begin{eqnarray*}
L(\hat{f}) = \probc{\hat{f}(X) \neq Y}{\D_n}\ , \quad \mbox{and}\qquad R(\hat{f}) = \E\croch{\probc{\hat{f}(X) \neq Y}{\D_n}}  .
\end{eqnarray*}
In this paper we focus on the estimation of $L(\hat f)$ (and its expectation $R(\hat f)$) by use of the \emph{Leave-$p$-Out} (L$p$O) cross-validation for $1\leq p\leq n-1$ \citep{Zha:1993,CeRo08}.
L$p$O successively considers all possible splits of $\D_n$ into a training set of cardinality $n-p$ and a test set of cardinality $p$.
Denoting by $\Enp$
the set of all possible subsets of $\acc{1,\ldots,n}$ with cardinality $n-p$, any $e\in\Enp$ defines a split of $\D_n$ into a training sample $\De=\acc{Z_i \mid i\in e}$ and a test sample $\Deb$, where $\bar e = \acc{1,\ldots,n}\setminus e$.
For a given classification algorithm $\A$, the final L$p$O estimator of the performance of $\A^{\D_n}(\cdot) = \widehat f$ is the average (over all possible splits) of the classification error estimated on each test set, that is
\begin{eqnarray}\label{def.Lpo.estimator}
\widehat{R}_{p}( \A,\D_n) = {n\choose p}^{-1}\som{e\in\mathcal{E}_{n-p}}{} \paren{\frac{1}{p}\som{i\in\bar{e}}{}\ind{ \A^{\De}\paren{X_i}\neq Y_i} } ,
\end{eqnarray}
where $\A^{\De}\paren{ \cdot}$ is the classifier built from $\De$.
We refer the reader to \cite{ArCe_2010_survey} for a detailed description of L$p$O and other cross-validation procedures.
In the sequel, the lengthy notation $\widehat{R}_{p}( \A,\D_n) $ is replaced by $\widehat{R}_{p,n}$ in settings where no confusion can arise about the algorithm $\A$ or the training sample $\Dn$, and by  $\widehat{R}_{p}(\D_n) $ if the training sample has to be kept in mind.

\paragraph{Exact L$p$O for the $k$NN classification algorithm}
Usually due to its seemingly prohibitive computational cost, L$p$O is not
applied except with $p = 1$ where it reduces to the well known leave-one-out.
However unlike this widespread idea \cite{CeRo08,Cel:2008:phd,Celi_2014} proved that the L$p$O estimator can be efficiently computed by deriving closed-form formulas in several statistical frameworks.  
The $k$NN classification rule is another instance for which efficiently computing the L$p$O estimator is possible with a time complexity linear in $p$ as previously established by \cite{CelisseMaryHuard11}.
Let us briefly recall the main steps leading to the closed-form formula.
\begin{enumerate}
\item From Eq.~\eqref{def.Lpo.estimator} the L$p$O estimator can be expressed as a sum (over the $n$ observations of the complete sample) of probabilities:
\begin{eqnarray*}
{n\choose p}^{-1}\sum_{e\in \mathcal{E}_{n-p}}\frac{1}{p} \paren{ \sum_{i \notin e} \ind{ \A^{\De}\gp{X_i}\neq Y_i} } &=& \frac{1}{p}\sum_{i=1}^n \croch{ {n\choose p}^{-1}\sum_{e\in\mathcal{E}_{n-p}} \ind{ \A^{\De}\gp{X_i}\neq Y_i} \ind{ i \notin e} } \\
&=& \frac{1}{p}\sum_{i=1}^n \P_e(\A^{\De}\gp{X_i} \neq Y_i \mid i \notin e)\P_e( i \notin e) .
\end{eqnarray*}
Here $\P_e$ means that the integration is computed with respect to the  random variable $e \in \mathcal{E}_{n-p}$, which follows the uniform distribution over the ${n\choose p}$ possible subsets with cardinality $n-p$ in $\mathcal{E}_{n-p}$.
For instance $\P_e( i \notin e) = p/n$ since it is the proportion of subsamples with cardinality $n-p$ which do not contain a given prescribed index $i$, which equals ${n-1\choose n-p}/{n\choose p}$. (See also Lemma~\ref{BasicResultForResamplingKnn} for further examples of such calculations.)

\item For any $X_i$, let $X_{(1)},...,X_{(k+p-1)},X_{(k+p)},...,X_{(n-1)}$ be the ordered sequence of neighbors of $X_i$. 
This list depends on $X_i$, i.e. $X_{(1)}$ should be noted $X_{(i,1)}$, but this dependency is skipped here for the sake of readability. 

The key in the derivation is to condition with respect to the random variable $R_{k}^i$ which denotes the rank (in the whole sample $\D_n$) of the $k-$th neighbor of $X_i$ in the $\De$, that is $R_k^i=j$ means that $X_{(j)}$ is the $k$-th neighbor of $X_i$ in $\De$.
%
%
Then
\begin{eqnarray*}
\P_e(\A^{\De}\gp{X_i} \neq Y_i|i \notin e) = \sum_{j=k}^{k+p-1} \P_e(\A^{\De}\gp{X_i}\neq Y_i \mid R_k^i=j, \ i \notin e)\P_e(R_k^i=j \mid i \notin e) ,
\end{eqnarray*}
where the sum involves $p$ terms since only $X_{(k)},\dots,X_{(k+p-1)}$ are candidates for being the $k$-th neighbor of $X_i$ in at least one training subset $e$.

\item Observe that the resulting probabilities can be easily computed (see Lemma~\ref{BasicResultForResamplingKnn}):
\begin{equation*}
\begin{array}{l}
\star \ \P_e(i \notin e) = \frac{p}{n}\\
\star \ \P_e(R_k^i=j| i \notin e) = \frac{k}{j}P\left(U=j-k\right)   \\
\star \ \P_e(\A^{\De}\gp{X_i} \neq Y_i| V_k^i=j, \ i \notin e) =
(1-Y_j) \croch{1-F_H\left(\frac{k+1}{2}\right) } +
Y_j \croch{ 1-F_{H'}\paren{ \frac{k-1}{2} } } ,
\end{array}
\end{equation*}
with $U \sim  \mathcal{H}(j,n-j-1,p-1)$, $H \sim \mathcal{H}(N_i^j,j-N_i^j-1,k-1)$, and $H' \sim \mathcal{H}(N_i^j-1,j-N_i^j,k-1)$, where $\mathcal{H}$ denotes the hypergeometric distribution and $N_i^j$ is the number of 1's among the $j$ nearest neighbors of $X_i$ in $\D_n$.
\end{enumerate}
The computational cost of L$p$O for the $k$NN classifier is the same as that of L$1$O for the $(k+p-1)$NN classifier whatever $p$, that is $O(p\, n)$. 
This contrasts with the usual ${n\choose p}$ prohibitive computational complexity seemingly suffered by L$p$O.


\subsection{$U$-statistics: General bounds on L$p$O moments}

The purpose of the present section is to describe a general strategy allowing to derive new upper bounds on the polynomial moments of the L$p$O estimator.
As a first step of this strategy, we establish the connection between the L$p$O risk estimator and U-statistics.
Second, we exploit this connection to derive new upper bounds on the order-$q$ moments of the L$p$O estimator for $q\geq2$. Note that these upper bounds, which relate moments of the L$p$O estimator to those of the L1O estimator, hold true with any classifier.

Let us start by introducing $U$-statistics and recalling some of their basic properties that will serve our purposes. For a thorough presentation, we refer to the books by \cite{Serf:1980,Ko_Bo:1994}.
The first step is the definition of a $U$-statistic of order $m\in\N^*$ as an average over all $m$-tuples of distinct indices in $\acc{1,\ldots,n}$.
\begin{defn}[\cite{Ko_Bo:1994}] \label{def.Ustatistic}
Let $h:\ \X^m \longrightarrow \R $ (or $\R^k$) denote any Borelian function where $m\geq 1$ is an integer.
Let us further assume $h$ is a symmetric function of its arguments.
Then any function $U_n:\ \X^n \longrightarrow \R$ such that
\begin{align*}
  U_n(x_1, \ldots,x_n) = U_n(h)(x_1, \ldots,x_n) = { n\choose m}^{-1} \sum_{1\leq i_1<\ldots<i_m\leq n} h\paren{ x_{i_1},\ldots, x_{i_m} }
\end{align*}
where $m\leq n$, is a $U$-statistic of order $m$ and kernel $h$.
\end{defn}
Before clarifying the connection between L$p$O and $U$-statistics, let us introduce the main property of $U$-statistics our strategy relies on.
It consists in representing any U-statistic as an average, over all permutations, of sums of independent variables.
\begin{prop}[Eq.~(5.5) in~\cite{Hoef:1963}]\label{Prop : HoeffdingDecomposition}
  With the notation of Definition~\ref{def.Ustatistic}, let us define $W:\ \X^n \longrightarrow \R$ by
\begin{align}\label{eq.Ustatistics.sum.independent}
  W(x_1,\ldots,x_n) = \frac{1}{r} \sum_{j=1}^r h\paren{ x_{(j-1)m+1},\ldots,x_{jm} } ,
\end{align}
where $r = \floor{n/m}$ denotes the integer part of $n/m$.
Then
\begin{align*}
  U_n(x_1, \ldots,x_n) = \frac{1}{n!} \sum_{\sigma}
W\paren{ x_{\sigma(1)},\ldots,x_{\sigma(n)} } ,
\end{align*}
where $\sum_{\sigma}$ denotes the summation over all permutations $\sigma$ of $\acc{1,\ldots,n}$.
\end{prop}

\medskip

We are now in position to state the key remark of the paper. All the developments further exposed in the following result from this connection between the L$p$O estimator defined by Eq.~\eqref{def.Lpo.estimator} and $U$-statistics.
\begin{thm}\label{theorem.lpo.ustat}
For any classification rule $\A$ and any $1\leq p \leq n-1$ such that the following quantities are well defined, the L$p$O estimator $\Rh_{p,n}$ is a U-statistic of order $m=n-p+1$ with kernel $h_m: \ \X^m \longrightarrow \R$ defined by
\begin{align*}
  h_m( Z_1,\ldots,Z_m ) = \frac{1}{m} \sum_{i=1}^m \1_{ \acc{\A^{ \D_m^{(i)} }( X_i) \neq Y_i } } ,
\end{align*}
where $\D_m^{(i)}$ denotes the sample $\D_m=(Z_1,\ldots,Z_m)$ with $Z_i$ withdrawn.
\end{thm}
\begin{proof}[Proof of Theorem~\ref{theorem.lpo.ustat}]
  ~\\

From Eq.~\eqref{def.Lpo.estimator}, the L$p$O estimator of the performance of any classification algorithm $\A$ computed from $\D_n$ satisfies
\begin{align*}
\Rh_p(\A,\D_n) = \Rh_{p,n} & = \frac{1}{{n\choose p}}\sum_{e\in \Enp} \frac{1}{p} \sum_{i\in\bar e} \1_{\acc{ \A^{\De}\paren{X_i} \neq Y_i }} \\
& = \frac{1}{{n\choose p}}\sum_{e\in\mathcal{E}_{n-p}} \frac{1}{p} \sum_{i\in\bar e} \paren{\sum_{v\in \mathcal{E}_{n-p+1}}\1_{\acc{ v = e\cup \acc{i}} }} \1_{\acc{ \A^{\De}\paren{X_i} \neq Y_i }} ,
\end{align*}
since there is a unique set of indices $v$ with cardinality $n-p+1$ such that $v= e\cup \acc{i}$.
Then
\begin{align*}
\Rh_{p,n} 
& = \frac{1}{{n\choose p}}\sum_{v\in \mathcal{E}_{n-p+1}}  \frac{1}{p} \sum_{i=1}^n  \paren{ \sum_{e\in\mathcal{E}_{n-p}} \1_{\acc{ v = e\cup \acc{i}} }  \1_{\acc{i\in\bar e}} } \1_{\acc{ \A^{\D^{v\setminus \acc{i}}}\paren{ X_i} \neq Y_i }}.
\end{align*}
Furthermore for $v$ and $i$ fixed, $\sum_{e\in\mathcal{E}_{n-p}} \1_{\acc{ v = e\cup \acc{i}} }  \1_{\acc{i\in\bar e}} = \1_{\acc{ i\in v}}$ since there is a unique set of indices $e$ such that $e = v\setminus i$.
One gets
\begin{align*}
\Rh_{p,n} 
& = \frac{1}{p}\frac{1}{{n\choose p}}\sum_{v\in \mathcal{E}_{n-p+1}}   \sum_{i=1}^n  \1_{\acc{ i\in v}} \1_{\acc{ \A^{\D^{v\setminus \acc{i}}}\paren{ X_i} \neq Y_i }} \\
& =  \frac{1}{{n\choose n-p+1}} \sum_{v\in \mathcal{E}_{n-p+1}}   \frac{1}{n-p+1} \sum_{i \in v} \1_{\acc{ \A^{\D^{v\setminus \acc{i}}}\paren{ X_i} \neq Y_i }},
\end{align*}
by noticing $p {n\choose p} = \frac{ p n!}{ p! n-p!} = \frac{ n!}{ p-1! n-p!} = (n-p+1) {n\choose n-p+1} .$

\end{proof}
The kernel $h_m$ is a deterministic and symmetric function of its arguments that does only depend on $m$.
Let us also notice that $h_m\paren{Z_{1},\ldots, Z_{m} }$ reduces to the L1O estimator of the risk of the classification rule $\A$ computed from $Z_1,\ldots,Z_m$, that is
\begin{align}\label{eq.notation.LpO.Ustatistic}
  h_m\paren{Z_{1},\ldots, Z_{m} } = \Rh_1\paren{ \A,\D_m } = \Rh_{1,n-p+1} .
\end{align}
In the context of testing whether two binary classifiers have different error rates, this fact has already been pointed out by \cite{FHDBB:2013}.

\medskip

We now derive a general upper bound on the $q$-th moment ($q\geq 1$) of the L$p$O estimator that holds true for any classification rule as long as the following quantities remain meaningful.
\begin{thm}\label{Prop : upper.bound.lpo.loo}
For any classification rule $\A$, let $\A^{\D_n}(\cdot)$ and $\A^{\D_m}(\cdot)$ be the corresponding classifiers built from respectively $\Dn $ and $\D_{m}$, where $ m = n-p+1$.
Then for every $1\leq p\leq n-1$ such that the following quantities are well defined, and any $q\geq 1$,
\begin{eqnarray} \label{ineq.Lpo.Loo.moments.direct.quad}
& \esp{ \left| \Rh_{p,n} - \esp{ \Rh_{p,n} } \right|^q } \leq  \E\croch{ \abs{  \Rh_{1,m} - \E\croch{ \Rh_{1,m} } }^q } . \hfill
\end{eqnarray}
Furthermore as long as $p> n/2+1$, one also gets
\begin{itemize}
    \item for $q=2$
    \begin{align} \label{ineq.Lpo.Loo.moments.quad}
    \esp{ \left| \Rh_{p,n} - \esp{ \Rh_{p,n} } \right|^2 } \leq   \frac{  \esp{ \left| \Rh_{1,m} - \esp{ \Rh_{1,m} } \right|^2 }  }{\floor{ \frac{n}{m}}} \enspace\cdot \hfill
    \end{align}
    \item for every $q>2$
\begin{align} \label{ineq.Lpo.Loo.moments.larger.quad}
& \esp{ \left| \Rh_{p,n} - \esp{ \Rh_{p,n} } \right|^q } \leq  B(q,\gamma) \times  \hfill \nonumber \\
&\max\acc{ 2^{q}\gamma \floor{ \frac{n}{m}} \E\croch{ \abs{ \frac{ \Rh_{1,m} - \E\croch{ \Rh_{1,m} } }{\floor{ \frac{n}{m}}} }^q } , \paren{ \sqrt{ \frac{2
\Var\paren{  \Rh_{1,m} }}{ \floor{ \frac{n}{m}} } } }^q } , \hfill
\end{align}
where $\gamma>0$ is a numeric constant and $B(q,\gamma)$ denotes the optimal constant defined in the Rosenthal inequality (Proposition~\ref{prop.rosenthal.inequality}).
\end{itemize}
\end{thm}
The proof is given in Appendix~\ref{Appendix: UpperBoundLpO.L1O}.
Eq.~\eqref{ineq.Lpo.Loo.moments.direct.quad} and Eq.~\eqref{ineq.Lpo.Loo.moments.quad} straightforwardly result from the Jensen inequality applied to the average over all permutations provided in Proposition~\ref{Prop : HoeffdingDecomposition}.
If $p>n/2+1$, the integer part $\floor{n/m}$ becomes larger than 1 and Eq.~\eqref{ineq.Lpo.Loo.moments.quad} becomes better than Eq.~\eqref{ineq.Lpo.Loo.moments.direct.quad} for $q=2$.
As a consequence of our strategy of proof, the right-hand side of Eq.~\eqref{ineq.Lpo.Loo.moments.quad} is equal to the classical upper bound on the variance of U-statistics which suggests it cannot be improved without adding further assumptions.

Unlike the above ones, Eq.~\eqref{ineq.Lpo.Loo.moments.larger.quad} is derived from the Rosenthal inequality, which enables us to upper bound a sum $\norm{\sum_{i=1}^r \xi_i }_q$ of independent and identically centered random variables in terms of $\sum_{i=1}^r \norm{\xi_i }_q$ and $\sum_{i=1}^r\Var(\xi_i)$.
Let us remark that, for $q=2$, both terms of the right-hand side of Eq.~\eqref{ineq.Lpo.Loo.moments.larger.quad} are of the same order as Eq.~\eqref{ineq.Lpo.Loo.moments.quad} up to constants.
Furthermore using the Rosenthal inequality allows us to take advantage of the integer part $\floor{n/m}$ when $p>n/2+1$, unlike what we get by using Eq.\eqref{ineq.Lpo.Loo.moments.direct.quad} for $q>2$.
In particular it provides a new understanding of the behavior of the L$p$O estimator when $p/n \to 1$ as highlighted later by Proposition~\ref{prop.exponential.inequality.p.large.explicit.deviations}.
%



\section{New bounds on L$p$O moments for the $k$NN classifier} \label{Section: Polynomial}

Our goal is now to specify the general upper bounds provided by Theorem~\ref{Prop : upper.bound.lpo.loo} in the case of the $k$NN classification rule $\A_k$ ($1\leq k\leq n$) introduced by~\eqref{def.knn.classifier}.

Since Theorem~\ref{Prop : upper.bound.lpo.loo} expresses the moments of the L$p$O estimator in terms of those of the L1O estimator computed from $\D_m$ (with $m=n-p+1$), the next step consists in focusing on the L1O moments.
Deriving upper bounds on the moments of the L1O is achieved using a generalization of the well-known Efron-Stein inequality (see Theorem~\ref{thm.efron.stein} for Efron-Stein's inequality and Theorem~15.5 in \cite{BouLugMas_2013} for its generalization).
For the sake of completeness, we first recall a corollary of this generalization that is proved in Section~\ref{subsec.generalized.efron.stein} (see Corollary~\ref{cor.generalized.Efron.Stein.appendix}).
\begin{prop}\label{cor.generalized.Efron.Stein} \sloppy
Let $Z_1,\ldots,Z_n$ denote $n$ independent random variables and $\zeta=f(Z_1,\ldots,Z_n)$, where $f:\ \R^n\rightarrow \R$ is any Borelian function.
With $Z'_1,\ldots,Z'_n$ independent copies of the $Z_i$s, there exists a universal constant $\kappa\leq 1.271$ such that for any $q\geq 2$,
\begin{align*}
	\norm{ \zeta-\E \zeta }_q & \leq \sqrt{2\kappa q} \sqrt{\norm{\sum_{i=1}^n \paren{ f(Z_1,\ldots,Z_n)-f(Z_1,\ldots,Z^\prime_i,\ldots,Z_n) }^2}_{q/2}} .
\end{align*}
\end{prop}
Then applying Proposition~\ref{cor.generalized.Efron.Stein} with $ \zeta = \Rh_1(\A_k,\D_m) = \Rh_{1,m}$ (L$1$O estimator computed from $\D_m$ with $m=n-p+1$) leads to the following Theorem~\ref{prop.moment.upper.bounds.Loo}, which finally allows us to control the order-$q$ moments of the L1O estimator applied to the $k$NN classifier.
\begin{thm}\label{prop.moment.upper.bounds.Loo} For every $1\leq k\leq n-1$, let $A_k^{\D_m}$ ($m=n-p+1$) denote the $k$NN classifier learnt from $\D_m$ and $\Rh_{1,m} $ be the corresponding L$1$O estimator given by Eq.~\eqref{def.Lpo.estimator}.
Then 
\begin{itemize}
    \item for $q=2$,
    \begin{eqnarray}\label{ineq.variance.generalized.efron.stein.loo.q2}
    \E\croch{ \paren{\Rh_{1,m} - \E\croch{\Rh_{1,m}} }^2}  \leq C_1   \frac{ k^{3/2}}{m} \enspace ;
    \end{eqnarray}
    \item for every $q>2$,
    \begin{eqnarray}\label{ineq.variance.generalized.efron.stein.loo.qlarger2}
    \E\croch{ \abs{ \Rh_{1,m} - \E\croch{ \Rh_{1,m} } }^{q} }
    & \leq   (C_2 \cdot k )^{q} \paren{ \frac{q}{m} }^{q/2}   ,
    \end{eqnarray}
\end{itemize}
with $C_1 =  2+16\gamma_d $ and $C_2 = 4\gamma_d\sqrt{2\kappa}$, where $\gamma_d$ is a constant (arising from Stone's lemma, see Lemma~\ref{Stone}) that grows exponentially with dimension $d$, and $\kappa$ is defined in Proposition~\ref{cor.generalized.Efron.Stein}.
\end{thm}
Its proof (detailed in Section~\ref{Appendix: BorneMomentLpO.kNN}) involves the use of Stone's lemma (Lemma~\ref{Stone}), which upper bounds, for a given $X_i$, the number of points in $\D_n^{ (i) }$ having $X_i$ among their $k$ nearest neighbors by $k\gamma_d$.
The dependence of our upper bounds with respect to $\gamma_d$ (see explicit constants $C_1$ and $C_2$) induces their strong deterioration as the dimension $d$ grows since $\gamma_d \approx 4.8^d-1$. Therefore the larger the dimension $d$, the larger the required sample size $n$ for the upper bound to be small (at least smaller than 1).
Note also that the tie breaking strategy (based on the smallest index) is chosen so that it ensures Stone's lemma to hold true.

In Eq.~\eqref{ineq.variance.generalized.efron.stein.loo.q2}, the easier case $q=2$  enables to exploit exact calculations (rather than upper bounds) of the variance of the L1O.
Further noticing $\E\croch{ \Rh_{1,m} }= R\paren{ A_k^{\D_{n-p}}}$ (risk of the $k$NN classifier learnt from $\D_{n-p}$), the resulting $k^{3/2}/m$ rate is a strict improvement upon the usual upper bound in $k^2/m$ which is derived from using the sub-Gaussian exponential concentration inequality provided by Theorem~24.4 in \cite{DeGyLu_1996}.

By contrast the larger $k^q$ in Eq.~\eqref{ineq.variance.generalized.efron.stein.loo.qlarger2} comes from the difficulty to derive a tight upper bound with $q>2$ for the expectation of $ (\sum_{i =1}^n \ind{ A_k^{\D_m^{(i)}}\paren{X_i} \neq A_k^{\D_m^{(i,j)}}\paren{X_i} }  )^q$, where $\D_m^{(i)}$ (resp. $\D_m^{(i,j)}$) denotes the sample $\D_m$ where $Z_i$ has been (resp. $Z_i$ and $Z_j$ have been) removed.

\bigskip

We are now in position to state the main result of this section.
It follows from the combination of Theorem~\ref{Prop : upper.bound.lpo.loo} (connecting moments of the L$p$O estimator to those of the L1O) and Theorem~\ref{prop.moment.upper.bounds.Loo} (providing an upper bound on the order-$q$ moments of the L1O).
\begin{thm}\label{prop.moment.upper.bounds.Lpo} For every $p,k\geq 1$  such that $p+k \leq n$, let $\Rh_{p,n}$ denote the L$p$O risk estimator (see \eqref{def.Lpo.estimator}) of the $k$NN classifier $ \A_k^{\Dn}\paren{\cdot}$ defined by \eqref{def.knn.classifier}.
Then there exist (known) constants $C_1,C_2>0$ such that for every $1\leq p\leq n-k$,
\begin{itemize}
    \item for $q=2$,
    \begin{eqnarray}\label{ineq.variance.generalized.efron.stein}
    \E\croch{ \paren{\Rh_{p,n} - \E\croch{\Rh_{p,n}} }^2}  \leq   C_1  \frac{ k^{3/2}}{ (n-p+1) }   \enspace ;
    \end{eqnarray}
    \item for every $q>2$,
        \begin{eqnarray}\label{ineq.p.small.generalized.efron.stein}
        & \E\croch{ \abs{ \Rh_{p,n} - \E\croch{\Rh_{p,n}} }^{q} }
        \leq   \paren{C_2 k}^q \paren{ \frac{q}{n-p+1}  }^{q/2}   ,
        \end{eqnarray}
\end{itemize}
with $C_1 =  \frac{128\kappa\gamma_d}{\sqrt{2\pi}}$ and $C_2 = 4\gamma_d\sqrt{2\kappa}$, where $\gamma_d$ denotes the constant arising from Stone's lemma (Lemma~\ref{Stone}).
Furthermore in the particular setting where $n/2+1 < p \leq n-k$, then
    \begin{itemize}
        \item  for $q=2$,
    \begin{eqnarray}\label{ineq.p.large.variance.generalized.efron.stein}
    \E\croch{ \paren{\Rh_{p,n} - \E\croch{\Rh_{p,n}} }^2}  \leq   C_1  \frac{ k^{3/2} }{ (n-p+1)\floor{ \frac{n}{n-p+1} } }  \enspace ,
    \end{eqnarray}
    \item for every $q>2$,

\begin{eqnarray}\label{ineq.p.large.generalized.efron.stein}
&& \E\croch{ \abs{ \Rh_{p,n} - \E\croch{\Rh_{p,n}} }^{q} } \hspace*{8.5cm} \nonumber \\
&\leq & \floor{\frac{n}{n-p+1} } \Gamma^q  \max\paren{  \frac{k^{3/2} }{\paren{n-p+1}\floor{\frac{n}{n-p+1}} } \, q  ,    \frac{k^2}{ (n-p+1) \floor{\frac{n}{n-p+1}}^2} \, q^{3} }^{q/2} ,
\end{eqnarray}
where $\Gamma =2\sqrt{2e} \max\paren{    \sqrt{ 2 C_1  } ,  2 C_2 }$.
    \end{itemize}
\end{thm}
The straightforward proof is detailed in Section~\ref{sec.proof.upper.bound.moments.Lpo}.
Let us start by noticing that both upper bounds in Eq.~\eqref{ineq.variance.generalized.efron.stein} and~\eqref{ineq.p.small.generalized.efron.stein} deteriorate as $p$ grows.
This is no longer the case for Eq.~\eqref{ineq.p.large.variance.generalized.efron.stein} and~\eqref{ineq.p.large.generalized.efron.stein}, which are specifically designed to cover the setup where $p>n/2+1$, that is where $\floor{n/m}$ is no longer equal to 1.
%
%
Therefore unlike Eq.~\eqref{ineq.variance.generalized.efron.stein} and~\eqref{ineq.p.small.generalized.efron.stein}, these last two inequalities are particularly relevant in the setup where $p/n\to 1$, as $n\to +\infty$, which has been investigated in different frameworks by \cite{Sha:1993,Yan:2006,Yang07,Celi_2014}.
Eq.~\eqref{ineq.p.large.variance.generalized.efron.stein} and~\eqref{ineq.p.large.generalized.efron.stein} lead to respective convergence rates at worse $k^{3/2}/n$ (for $q=2$) and $k^q/n^{q-1}$ (for $q>2$). In particular this last rate becomes approximately equal to $(k/n)^q$ as $q$ gets large.

One can also emphasize that, as a U-statistic of fixed order $m=n-p+1$, the L$p$O estimator has a known Gaussian limiting distribution, that is \citep[see Theorem~A, Section~5.5.1][]{Serf:1980} 
\begin{align*}
\frac{\sqrt{n}}{m} \paren{ \Rh_{p,n} - \E\croch{\Rh_{p,n}} } \xrightarrow[n\to +\infty]{\mathcal{L}} \mathcal{N}\paren{0, \sigma_1^2} ,
\end{align*}
where $\sigma_1^2 = \Var\croch{ g(Z_1) } $, with $g(z) = E\croch{ h_m(z,Z_2,\ldots,Z_{m}) }$.
Therefore the upper bound given by Eq.~\eqref{ineq.p.large.variance.generalized.efron.stein} is non-improvable in some sense with respect to the interplay between $n$ and $p$ since one recovers the right magnitude for the variance term as long as $m=n-p+1$ is assumed to be constant.

Finally Eq.~\eqref{ineq.p.large.generalized.efron.stein} has been derived using a specific version of the Rosenthal inequality \citep{IbragShar2002} stated with the optimal constant and involving a ``balancing factor''.
In particular this balancing factor has allowed us to optimize the relative weight of the two terms between brackets in Eq.~\eqref{ineq.p.large.generalized.efron.stein}.
This leads us to claim that the dependence of the upper bound with respect to $q$ cannot be improved with this line of proof.
However we cannot conclude that the term in $q^3$ cannot be improved using other technical arguments.

\section{Exponential concentration inequalities} \label{Section: Deviations}

This section provides exponential concentration inequalities for the L$p$O estimator applied to the $k$NN classifier.
Our main results heavily rely on the moments inequalities previously derived in Section~\ref{Section: Polynomial}, that is Theorem~\ref{prop.moment.upper.bounds.Lpo}.
In order to emphasize the gain allowed by this strategy of proof, we start this section by successively proving two exponential inequalities obtained with less sophisticated tools.
We then discuss the strength and weakness of each of them to justify the additional refinements we introduce step by step along the section.

A first exponential concentration inequality for $\Rhp(\A_k,\D_n) = \Rh_{p,n}$ can be derived by use of the bounded difference inequality following the line of proof of \citet[][Theorem~24.4]{DeGyLu_1996} originally developed for the L1O estimator.
\begin{prop}\label{prop.concentration.Lpo.DGL.straightforward}
For any integers $p,k\geq 1$ such that $ p+k\leq n$, let $\Rh_{p,n}$ denote the L$p$O estimator \eqref{def.Lpo.estimator} of the classification error of the $k$NN classifier $\A_k^{\Dn}(\cdot)$ defined by \eqref{def.knn.classifier}.
Then for every $t>0$,
\begin{align} \label{naive.upper.bound.difference}
  \prob{ \abs{  \Rh_{p,n} - \E\paren{ \Rh_{p,n} } } > t  } \leq 2 e^{- n \frac{t^2}{  8(k+p-1)^2 \gamma_d^2 } }.
\end{align}
where $\gamma_d$ denotes the constant introduced in Stone's lemma (Lemma~\ref{Stone}).
\end{prop}
The proof is given in Appendix~\ref{subsec.bounded.differences.proof}.

The upper bound of Eq.~\eqref{naive.upper.bound.difference} strongly exploits the facts that: (i) for $X_j$ to be one of the $k$ nearest neighbors of $X_i$ in at least one subsample $X^e$, it requires $X_j$ to be one of the $k+p-1$ nearest neighbors of $X_i$ in the complete sample, and (ii) the number of points for which $X_j$ may be one of the $k+p-1$ nearest neighbors cannot be larger than $(k+p-1)\gamma_d$ by Stone's Lemma (see Lemma~\ref{Stone}).

This reasoning results in a rough upper bound since the denominator in the exponent exhibits a $(k+p-1)^2 $ factor where $k$ and $p$ play the same role.
The reason is that we do not distinguish between points for which $X_j$ is among or above the $k$ nearest neighbors of $X_i$ in the whole sample,
although these two setups lead to strongly different probabilities of being among the $k$ nearest neighbors in the training resample.
Consequently the dependence of the convergence rate on $k$ and $p$ in Proposition~\ref{prop.concentration.Lpo.DGL.straightforward} can be improved, as confirmed by forthcoming Theorems~\ref{Prop : ConcentrationLpO} and~\ref{Prop : ConcentrationLpO:bis}.

\medskip

Based on the previous comments, a sharper quantification of the influence of each neighbor among the $k+p-1$ ones leads to the next result.
\begin{thm}\label{Prop : ConcentrationLpO}
For every $p,k\geq 1$ such that $p + k \leq n$, let $\Rh_{p,n}$ denote the L$p$O estimator \eqref{def.Lpo.estimator} of the classification error of the $k$NN classifier $\A_k^{\Dn}(\cdot)$ defined by \eqref{def.knn.classifier}.
Then there exists a numeric constant $\square>0$ such that for every $t>0$,
{\small
\begin{align*}
 & \max\paren{\prob{  \Rh_{p,n} - \E\paren{ \Rh_{p,n} }  > t  } ,  \prob{ \E\paren{ \Rh_{p,n} } - \Rh_{p,n} > t  } }   \leq  \exp \paren{ -  n \frac{  t^2 }{  \square k^2 \croch{1 + (k+p)\frac{p-1}{n-1}} } }  ,
\end{align*}}
with $\square = 1024 e\kappa (1+\gamma_d )$, where $\gamma_d$ is introduced in Lemma~\ref{Stone} and $\kappa\leq 1.271$ is a universal constant.
\end{thm}
The proof is given in Section~\ref{Appendix:ConcentrationLpO}.

Let us remark that unlike Proposition~\ref{prop.concentration.Lpo.DGL.straightforward}, taking into account the rank of each neighbor in the whole sample enables to considerably reduce the weight of $p$ (compared to that of $k$) in the denominator of the exponent.
In particular, one observes that letting $p/n\to 0$ as $n\to +\infty$ (with $k$ assumed to be fixed for instance) makes the influence of the $k+p$ factor asymptotically negligible. This would allow to recover (up to numeric constants) a similar upper bound to that of \citet[][Theorem~24.4]{DeGyLu_1996}, achieved with $p=1$.

\medskip

However the upper bound of Theorem~\ref{Prop : ConcentrationLpO} does not reflect the right dependencies with respect to $k$ and $p$ compared with what has been proved for polynomial moments in Theorem~\ref{prop.moment.upper.bounds.Lpo}.
The upper bound seems to strictly deteriorate as $p$ increases, which contrasts with the upper bounds derived for $p>n/2+1$ in Theorem~\ref{prop.moment.upper.bounds.Lpo}.
This drawback is overcome by the following result, which is our main contribution in the present section.
\begin{thm}\label{Prop : ConcentrationLpO:bis}
For every $p,k\geq 1$ such that  $p+k \leq n$, let $\Rh_{p,n}$ denote the L$p$O estimator of the classification error of the $k$NN classifier $\hat f_k=\A_k^{\Dn}(\cdot)$ defined by \eqref{def.knn.classifier}.
Then for every $t>0$,
\begin{align}\label{ineq.exponential.concentration.small.p}
\max\paren{\prob{ \Rh_{p,n} - \E\croch{\Rh_{p,n}} >t }   , \prob{ \E\croch{\Rh_{p,n}} - \Rh_{p,n} >t } }  \leq   \exp\paren{- (n-p+1) \frac{t^2}{\Delta^2 k^2} } ,
\end{align}
where $\Delta = 4\sqrt{e}\max\paren{ C_2,\sqrt{C_1} }$ with $C_1,C_2>0$ defined in Theorem~\ref{prop.moment.upper.bounds.Loo}.

Furthermore in the particular setting where $p>n/2+1$, it comes
\begin{align}
& \max\paren{ \prob{ \Rh_{p,n} - \E\croch{\Rh_{p,n}} >t } ,  \prob{ \E\croch{\Rh_{p,n}} - \Rh_{p,n} >t }} \leq e \floor{\frac{n}{n-p+1}} \times \nonumber \\
&{\small  \exp\gc{- \frac{1}{2e} \min\acc{  \paren{n-p+1}\floor{\frac{n}{n-p+1}}\frac{t^2 }{4\Gamma^2 k^{3/2}}  , \  \gp{ (n-p+1) \floor{\frac{n}{n-p+1}}^2\frac{ t^2}{4\Gamma^2k^2} }^{1/3} }  }  } , \hfill \label{ineq.exponential.concentration.large.p}
\end{align}
where $\Gamma$ arises in Eq.~\eqref{ineq.p.large.generalized.efron.stein} and $\gamma_d$ denotes the constant introduced in Stone's lemma (Lemma~\ref{Stone}).
\end{thm}
The proof has been postponed to Appendix~\ref{Appendix:ConcentrationLpObis}. It involves different arguments for the two inequalities~\eqref{ineq.exponential.concentration.small.p} and~\eqref{ineq.exponential.concentration.large.p} depending on the range of values of $p$.
Firstly for $p\leq n/2+1$, a simple argument is applied to derive Ineq.~\eqref{ineq.exponential.concentration.small.p} from the two corresponding moment inequalities of Theorem~\ref{prop.moment.upper.bounds.Lpo} characterizing the sub-Gaussian behavior of the L$p$O estimator in terms of its even moments (see Lemma~\ref{lem.subgaussian.moment.to.exp}).
Secondly for $p>n/2+1$, we rather exploit: $(i)$ the appropriate upper bounds on the moments of the L$p$O estimator given by Theorem~\ref{prop.moment.upper.bounds.Lpo}, and $(ii)$ a dedicated Proposition~\ref{prop.moment.exponential} which provides exponential concentration inequalities from general moment upper bounds.

In accordance with the conclusions drawn about Theorem~\ref{prop.moment.upper.bounds.Lpo}, the upper bound of Eq.~\eqref{ineq.exponential.concentration.small.p} increases as $p$ grows, unlike that of Eq.~\eqref{ineq.exponential.concentration.large.p} which improves as $p$ increases.
In particular the best concentration rate in Eq.~\eqref{ineq.exponential.concentration.large.p} is achieved as $p/n\to 1$, whereas Eq.~\eqref{ineq.exponential.concentration.small.p} turns out to be useless in that setting.
Let us also notice that	Eq.~\eqref{ineq.exponential.concentration.small.p} remains strictly better than Theorem~\ref{Prop : ConcentrationLpO} as long as $p/n \to \delta \in [0,1[ $, as $n\to +\infty$.
Note also that the constants $\Gamma$ and $\gamma_d$ are the same as in Theorem~\ref{prop.moment.upper.bounds.Loo}. Therefore the same comments regarding their dependence with respect to the dimension $d$ apply here.

\medskip

In order to facilitate the interpretation of the last Ineq.~\eqref{ineq.exponential.concentration.large.p}, we also derive the following proposition (proved in Appendix~\ref{Appendix:ConcentrationLpObis}) which focuses on the description of each deviation term in the particular case where $p> n/2+1$.
\begin{prop}\label{prop.exponential.inequality.p.large.explicit.deviations}
With the same notation as Theorem~\ref{Prop : ConcentrationLpO:bis}, for any $p,k\geq 1$ such that $p+k \leq n$, $p> n/2+1$, and for every $t>0$
\begin{align*}
 \P\croch{ \abs{\Rh_{p,n} - \E\croch{\Rh_{p,n}}} >  \frac{\sqrt{2e}\Gamma}{\sqrt{n-p+1}} \paren{\sqrt{\frac{k^{3/2}}{\floor{\frac{n}{n-p+1} } } t} +   2e  \frac{k}{\floor{\frac{n}{n-p+1} } } t^{3/2}   } } \leq \floor{\frac{n}{n-p+1}} e \cdot e^{-t} , \hfill
\end{align*}
where $\Gamma>0$ is the constant arising from \eqref{ineq.p.large.generalized.efron.stein}.
\end{prop}
The present inequality is very similar to the well-known Bernstein inequality \citep[][Theorem~2.10]{BouLugMas_2013} except the second deviation term of order $t^{3/2}$ instead of $t$ (for the Bernstein inequality).

With respect to $n$, the first deviation term is of order $\approx k^{3/2}/\sqrt{n}$, which is the same as with the Bernstein inequality.
The second deviation term is of a somewhat different order, that is $\approx k\sqrt{n-p+1}/n$, as compared with the usual $1/n$ in the Bernstein inequality.
Note that we almost recover this $k/n$ rate by choosing for instance $p \approx n (1- \log n/n)$, which leads to $k \sqrt{\log n}/n$.
Therefore varying $p$ allows to interpolate between the $k/\sqrt{n}$ and the $k/n$ rates.

Note also that the dependence of the first (sub-Gaussian) deviation term with respect to $k$ is only $k^{3/2}$, which improves upon the usual $k^2$ resulting from Ineq.~\eqref{ineq.exponential.concentration.small.p} in Theorem~\ref{Prop : ConcentrationLpO:bis} for instance. However this $k^{3/2}$ remains certainly too large for being optimal even if this question remains widely open at this stage in the literature. 

More generally one strength of our approach is its versatility. Indeed the two above deviation terms directly result from the two upper bounds on the moments of the L1O stated in Theorem~\ref{prop.moment.upper.bounds.Loo}. Therefore any improvement of the latter upper bounds would immediately lead to enhance the present concentration inequality (without changing the proof).

%
%


\section{Assessing the gap between L$p$O and classification error} \label{Section:Old}
\subsection{Upper bounds}
First, we derive new upper bounds on different measures of the discrepancy between $\Rh_{p,n} = \Rhp\paren{\A_k, \Dn }$ and the classification error $L(\hat f_k)$ or the risk $R(\hat f_k) = \E\croch{L(\hat f_k)}$.
These bounds on the L$p$O estimator are completely new for $p>1$, some of them being extensions of former ones specifically derived for the L1O estimator applied to the $k$NN classifier.
%

%
%
\begin{thm}\label{Prop : Consistency}
For every $ p,k\geq 1$ such that $p\leq \sqrt{k}$ and $ \sqrt{k}+k \leq n$, let $\Rh_{p,n}$ denote the L$p$O risk estimator (see \eqref{def.Lpo.estimator}) of the $k$NN classifier $\hat f_k=\A_k^{\Dn}(\cdot)$ defined by \eqref{def.knn.classifier}.
Then,
\begin{align}\label{Equ : BiasLpO}
\abs{\mathbb{E}\croch{ \Rh_{p,n} } -R(\hat{f}_k)  }  \leq  \frac{4}{\sqrt{2\pi}}\frac{p\sqrt{k}}{n}  \enspace ,
\end{align}
and
\begin{align}\label{eq.EQM.LpO}
\mathbb{E}\croch{ \paren{ \Rh_{p,n} - R(\hat{f}_k) }^2 }   \leq  \frac{128\kappa\gamma_d}{\sqrt{2\pi}} \frac{k^{3/2}}{n-p+1}+  \frac{16}{2\pi}\frac{p^2 k }{n^2} \enspace \cdot
\end{align}
Moreover,
\begin{eqnarray}\label{ineq.upper.bound.squared.deviation.moment}
\esp{\left(\Rh_{p,n}-L(\hat f_k)\right)^2} \leq \frac{2\sqrt{2}}{\sqrt{\pi}}\frac{(2p+3)\sqrt{k}}{n} + \frac{1}{n} \enspace\cdot
\end{eqnarray}
\end{thm}
By contrast with the results in the previous sections, a new restriction on $p$ arises in Theorem~\ref{Prop : Consistency}, that is $p\leq \sqrt{k}$.
It is the consequence of using Lemma~\ref{Lemma : HOresult} in the above proof to quantify how different two classifiers respectively computed from the same $n$ and $n-p$ points can be.
Indeed this lemma, which provides an upper bound on the $L^1$ stability of the $k$NN classifier previously proved by \cite{DeWa79}, only remains meaningful as long as $1\leq p\leq \sqrt{k}$.
\begin{proof}[Proof of Theorem~\ref{Prop : Consistency}]
	~\\
	\textbf{Proof of \eqref{Equ : BiasLpO}:}
	With $\hat{f}^e_k = \A_k^{\D^e}$, Lemma~\ref{Lemma : HOresult} immediately provides
	\begin{eqnarray*}
		\left|\mathbb{E}\left[\Rh_{p,n}-L(\hat{f}_k)\right]\right| &=& \left|\mathbb{E}\left[L(\hat{f}^e_k)\right]-\mathbb{E}\left[L(\hat{f}_k)\right]\right|\\
		& \leq &  \mathbb{E}\left[ \abs{\1_{\{\A_k^{\D^e}(X)\neq Y\}}-\1_{\{\A_k^{\Dn}(X)\neq Y\}} } \right] \\
		&=& \prob{\A_k^{\D^e}(X)\neq \A_k^{\Dn}(X)} \leq  \frac{4}{\sqrt{2\pi}}\frac{p\sqrt{k}}{n} \enspace\cdot
	\end{eqnarray*}

	\noindent \textbf{Proof of \eqref{eq.EQM.LpO}:}
	The proof combines the previous upper bound with the one established for the variance of the L$p$O estimator, that is Eq.~\eqref{ineq.variance.generalized.efron.stein}.
	\begin{align*}
	\mathbb{E}\croch{ \paren{ \Rh_{p,n}-\E\croch{L(\hat{f}_k)} }^2 }
	& = \mathbb{E}\croch{ \paren{ \Rh_{p,n}-\E\croch{ \Rh_{p,n} } }^2 } +  \paren{ \mathbb{E}\croch{ \Rh_{p,n} } -\E\croch{L(\hat{f}_k)} }^2  \\
	& \leq  \frac{128\kappa\gamma_d}{\sqrt{2\pi}} \frac{k^{3/2}}{n-p+1}+ \paren{ \frac{4}{\sqrt{2\pi}}\frac{p\sqrt{k}}{n} }^2 \enspace ,
	\end{align*}
	which concludes the proof.

	The proof of Ineq.~\eqref{ineq.upper.bound.squared.deviation.moment} is more intricate and has been postponed to Appendix~\ref{appendix.proof.consistency}.
\end{proof}

Keeping in mind that $\E\croch{ \Rh_{p,n}} = R( \A_k^{\D_{n-p}})$, the right-hand side of Ineq.~\eqref{Equ : BiasLpO} is an upper bound on the bias of the L$p$O estimator, that is on the difference between the risks of the classifiers built from respectively $n-p$ and $n$ points. Therefore, the fact that this upper bound increases with $p$ is reliable since the classifiers $\A_k^{\D_{n-p+1}}(\cdot)$ and $\A_k^{\Dn}(\cdot)$ can become more and more different from one another as $p$ increases.
More precisely, the upper bound in Ineq.~\eqref{Equ : BiasLpO} goes to 0 provided $p\sqrt{k}/n$ does. With the additional restriction $p\leq \sqrt{k}$, this reduces to the usual condition $k/n\to 0$ as $n\to+\infty$  \citep[see][Chap.~6.6 for instance]{DeGyLu_1996}.
The monotonicity of this upper upper bound with respect to $k$ can seem somewhat unexpected. One could think that the two classifiers would become more and more ``similar'' to each other as $k$ increases.
However it can be proved that, in some sense, this dependence cannot be improved in the present distribution-free framework (see Proposition~\ref{res.counter.example.bias} and Figure~\ref{Fig: EvolBias}).
%

%

%
Note that an upper bound similar to that of Ineq.~\eqref{eq.EQM.LpO} can be easily derived for any order-$q$ moment ($q\geq 2$) at the price of increasing the constants by using $(a+b)^q \leq 2^{q-1} (a^q + b^q)$, for every $a,b\geq 0$.
We also emphasize that Ineq.~\eqref{eq.EQM.LpO} allows us to control the discrepancy between the L$p$O estimator and the risk of the $k$NN classifier, that is the expectation of its classification error.
Ideally we would have liked to replace the risk $R(\hat f_k)$ by the prediction error $L(\hat f_k)$. But with our strategy of proof, this would require an additional distribution-free concentration inequality on the prediction error of the $k$NN classifier. To the best of our knowledge, such a concentration inequality is not available up to now.

Upper bounding the squared difference between the L$p$O estimator and the prediction error is precisely the purpose of  Ineq.~\eqref{ineq.upper.bound.squared.deviation.moment}. Proving the latter inequality requires a completely different strategy which can be traced back to an earlier proof by \citet[][see the proof of Theorem~2.1]{RogersWagner78} applying to the L1O estimator.
Let us mention that Ineq.~\eqref{ineq.upper.bound.squared.deviation.moment} combined with the Jensen inequality lead to a less accurate upper bound than Ineq.~\eqref{Equ : BiasLpO}.

Finally the apparent difference between the upper bounds in Ineq.~\eqref{eq.EQM.LpO} and~\eqref{ineq.upper.bound.squared.deviation.moment} results from the completely different schemes of proof.
The first one allows us to derive general upper bounds for all centered moments of the L$p$O estimator, but exhibits a worse dependence with respect to $k$. By contrast the second one is exclusively dedicated to upper bounding the mean squared difference between the prediction error and the L$p$O estimator and leads to a smaller $\sqrt{k}$.
However (even if probably not optimal), the upper bound used in Ineq.~\eqref{eq.EQM.LpO} still enables to achieve minimax rates over some H\"older balls as proved by Proposition~\ref{res.optimal.minimax.rate}.

\subsection{Lower bounds}

\subsubsection{Bias of the L$1$O estimator}

The purpose of the next result is to provide a counter-example highlighting that the upper bound of Eq.~\eqref{Equ : BiasLpO} cannot be improved in some sense. We consider the following discrete setting where $\mathcal{X}=\acc{0,1}$ with $\pi_0 = P\croch{X=0}$, and we define $\eta_0 = \P\croch{ Y=1 \mid X=0}$ and $\eta_1 = \P\croch{ Y=1 \mid X=1}$. In what follows this two-class generative model will be referred to as the discrete setting \textbf{DS}.

Note that $(i)$ the 3 parameters $\pi_0, \eta_0$ and $\eta_1$ fully describe the joint distribution $P_{(X,Y)}$, and $(ii)$ the distribution of \textbf{DS} satisfies the strong margin assumption of \cite{MassNedelec_2006} if both $\eta_0$ and $\eta_1$ are chosen away from 1/2. However this favourable setting has no particular effect on the forthcoming lower bound except a few simplifications along the calculations.

\begin{prop}\label{res.counter.example.bias}
Let us consider the \textbf{DS} setting with $\pi_0 = 1/2$, $\eta_0 = 0$ and $\eta_1 = 1$, and assume that $k$ is odd.
Then there exists a numeric constant $C>1$ independent of $n$ and $k$ such that, for all $n/2 \leq k \leq n-1$, the $k$NN classifiers $\A_k^{\Dn}$ and $\A_k^{\D_{n-1}}$ satisfy
\begin{align*}
\E\croch{ L\paren{\A_k^{\Dn}} - L\paren{\A_k^{\D_{n-1}}}  } \geq C \frac{\sqrt{k}}{n} \enspace\cdot
\end{align*}
\end{prop}
The proof of Proposition~\ref{res.counter.example.bias} is provided in Appendix~\ref{sec.proof.couter.example}. 
The rate $\sqrt{k}/n$ in the right-hand side of Eq.~\eqref{Equ : BiasLpO} is then achieved under the generative model \textbf{DS} for any $k\geq n/2$.
As a consequence this rate cannot be improved without any additional assumption, for instance on the distribution of the $X_i$s.
See also Figure \ref{Fig: EvolBias} below and related comments.

\paragraph{Empirical illustration}

To further illustrate the result of Proposition~\ref{res.counter.example.bias}, we simulated data according to the \textbf{DS} setting, for different values of $n$ ranging from 100 to 500 and different values of $k$ ranging from 5 to $n-1$.

Figure~\ref{Fig: EvolBias} (a) displays the evolution of the absolute bias $\left|\E\croch{ L\paren{\A_k^{\Dn}} - L\paren{\A_k^{\D_{n-1}}}  }\right|$ as a function of $k$, for several values of $n$ (plain curves).
The absolute bias is a nondecreasing function of $k$, as suggested by the upper bound provided in Eq.~\eqref{Equ : BiasLpO} which is also plotted (dashed lines) to ease the comparison.
Importantly, the non-decreasing behavior of the absolute bias is not always restricted to high values of $k$ (w.r.t. $n$), as illustrated in Figure~\ref{Fig: EvolBias} (b) which corresponds to the same \textbf{DS} setting but with parameter values $(\pi_0,\eta_0,\eta_1)=(0.2,0.2,0.9)$. In particular the non-decreasing behavior now appears for a range of values of $k$ that are lower than $n/2$. 

Note that a rough idea about the location of the peak, denoted by $k_{peak}$, can be deduced as follows in the simple case where $\eta_0=0$ and $\eta_1=1$.
\begin{itemize}
	\item For the peak to arise, the two classifiers (based on $n$ and respectively $n-1$ observations) have to disagree the most strongly. 
	
	\item This requires one of the two classifiers -- say the first one -- to have ties among the $k$ nearest neighbors of each label in at least one of the two cases $X=0$ or $X=1$. 
	
	\item With $\pi_0<0.5$, then ties will most likely occur for the case $X=0$. Therefore the discrepancy between the two classifiers will be the highest at any new observation $x_0=0$. 
		
	\item For the tie situation to arise at $x_0$, half of its neighbors have to be 1.
	This only occurs if $(i)$ $k>n_0$ (with $n_0$ the number of observations such that $X=0$ in the training set), and $(ii)$ $k_0\eta_0 + k_1\eta_1 = k/2$, where $k_0$ (resp. $k_1$) is the number of neighbors of $x_0$ such that $X=0$ (resp. $X=1$).
	
	\item Since $k>n_0$, one has $k_0=n_0$ and the last expression boils down to $\displaystyle{k=\frac{n_0(\eta_1-\eta_0)}{\eta_1-1/2}} \enspace \cdot$
	
	\item For large values of $n$, one should have $n_0\approx n\pi_0$, that is the peak should appear at $\displaystyle{k_{peak}\approx\frac{ n\pi_0(\eta_1-\eta_0)}{\eta_1-1/2}} \enspace \cdot$

\end{itemize}
In the setting of Proposition~\ref{res.counter.example.bias}, this reasoning remarkably yields $k_{peak}\approx n$, while it leads to $k_{peak}\approx 0.4n$ in the setting of Figure~\ref{Fig: EvolBias} (b), which is close to the location of the observed peaks.
This also suggests that even smaller values of $k_{peak}$ can arise by tuning the parameter $\pi_0$ close to 0.
Let us mention that very similar curves have been obtained for a Gaussian mixture model with two disjoint classes (not reported here). 
On the one hand this empirically illustrates that the $\sqrt{k}/n$ rate is not limited to the discrete setting \textbf{DS}. On the other hand, all of this confirms that this rate cannot be improved in the present distribution-free framework.

Let us finally consider Figure~\ref{Fig: EvolBias} (c), which displays the absolute bias as a function of $n$ where $k = \floor{\mbox{Coef} \times n}$ for different values of $\mbox{Coef}$, where $\floor{\cdot}$ denotes the integer part.
With this choice of $k$, Proposition~\ref{res.counter.example.bias} implies that the absolute bias should decrease at a $1/\sqrt{n}$ rate, which is supported by the plotted curves. By contrast,  panel (d) of Figure~\ref{Fig: EvolBias} illustrates that choosing smaller values of $k$, that is $k = \floor{\mbox{Coef} \times \sqrt{n}}$, leads to a faster decreasing rate.

\begin{figure}
	\hspace*{-1cm}
\begin{tabular}{cc}
\includegraphics[width=8cm]{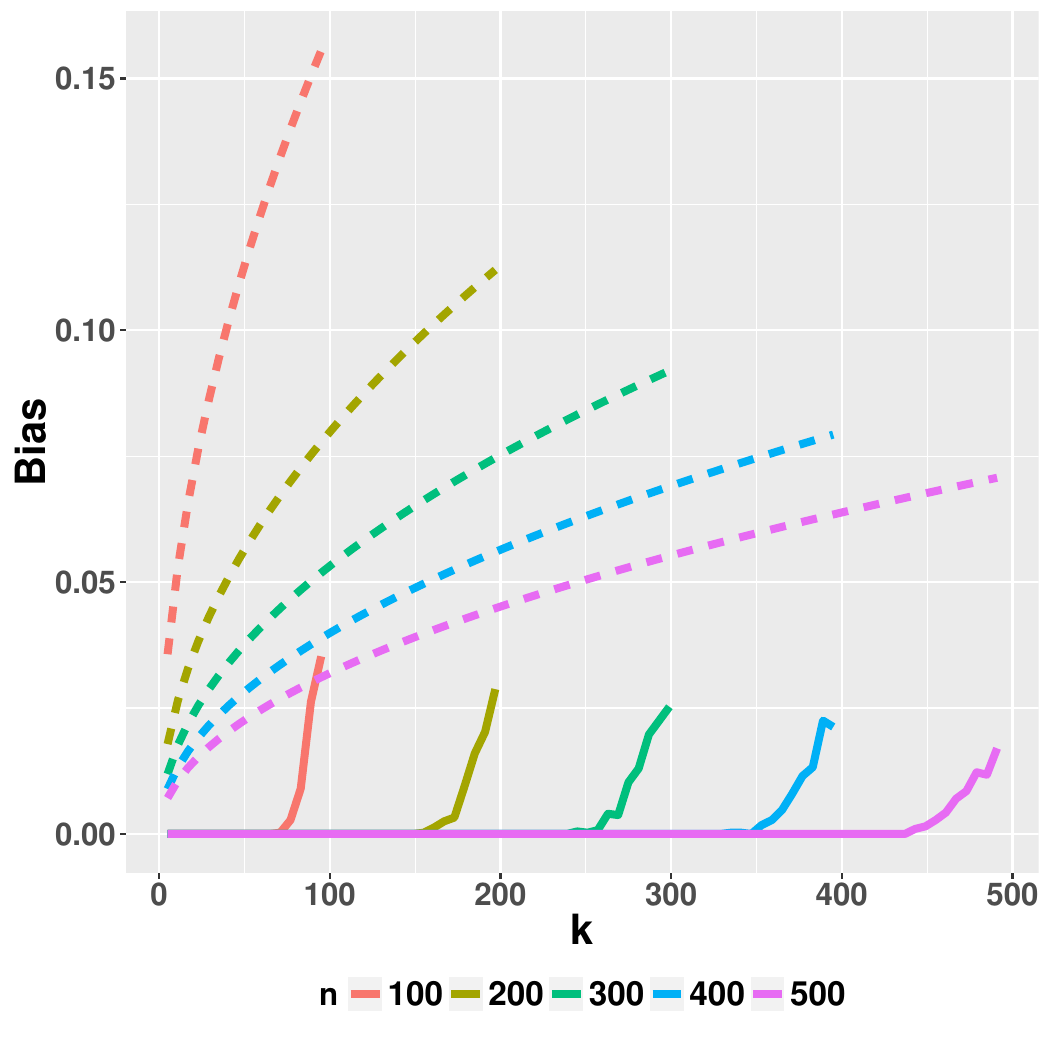} &
\includegraphics[width=8cm]{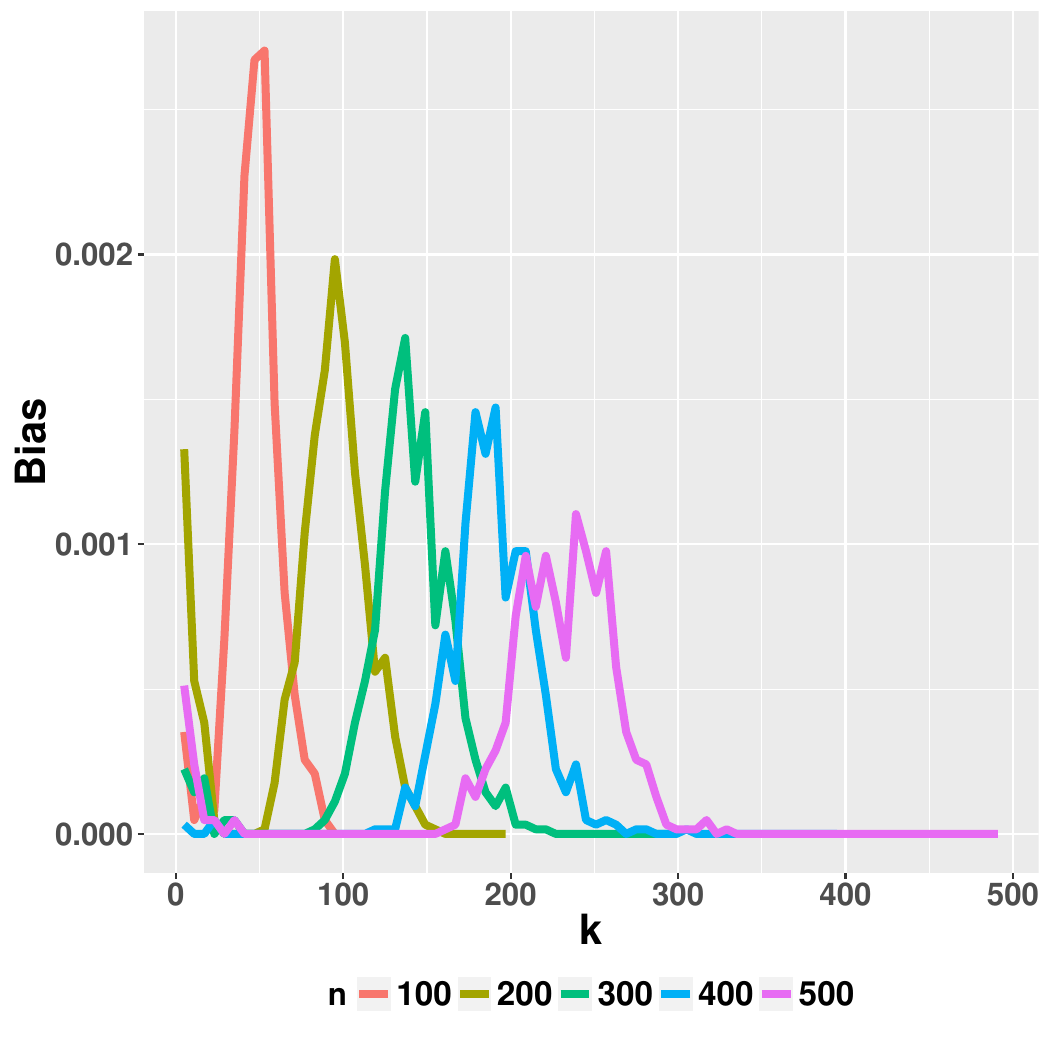} \\
(a) & (b)\\
\includegraphics[width=8cm]{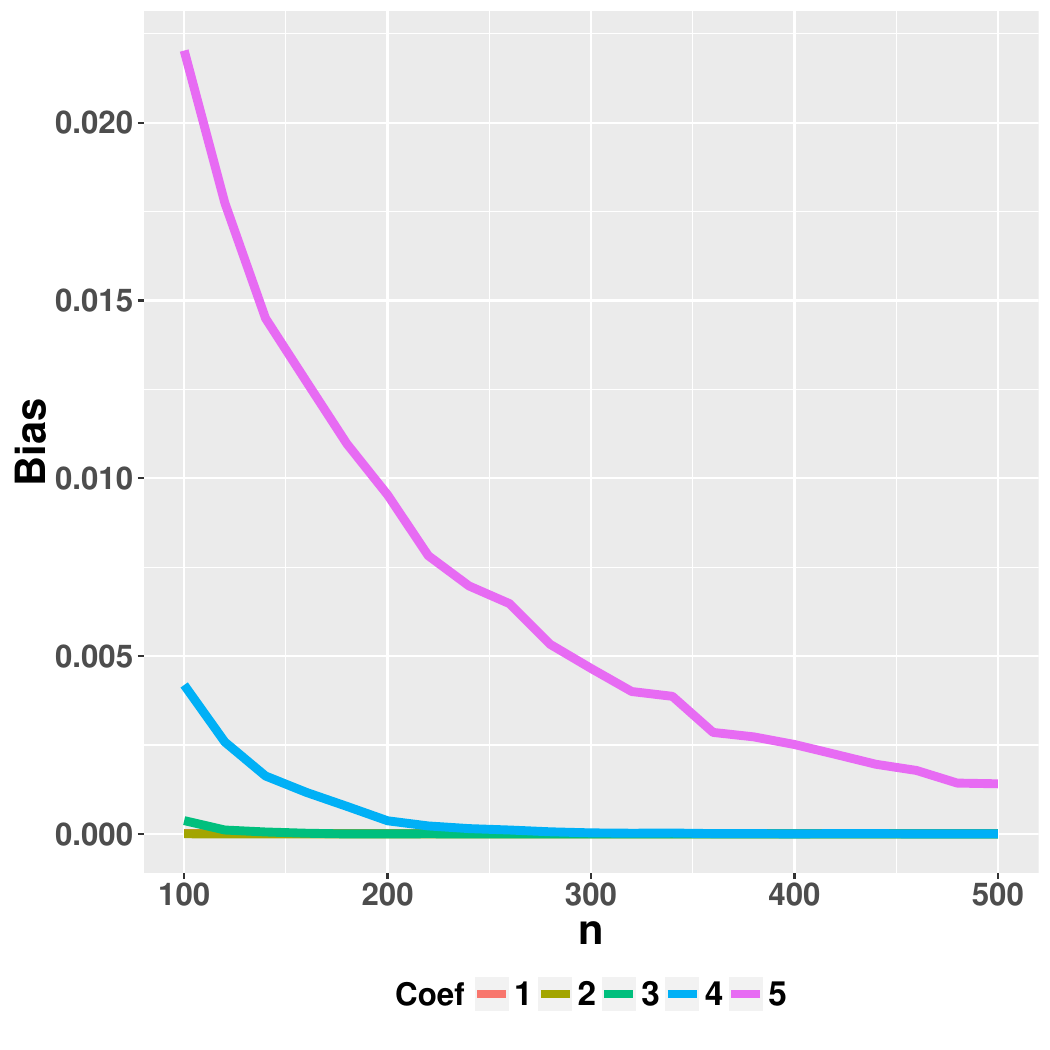} &
\includegraphics[width=8cm]{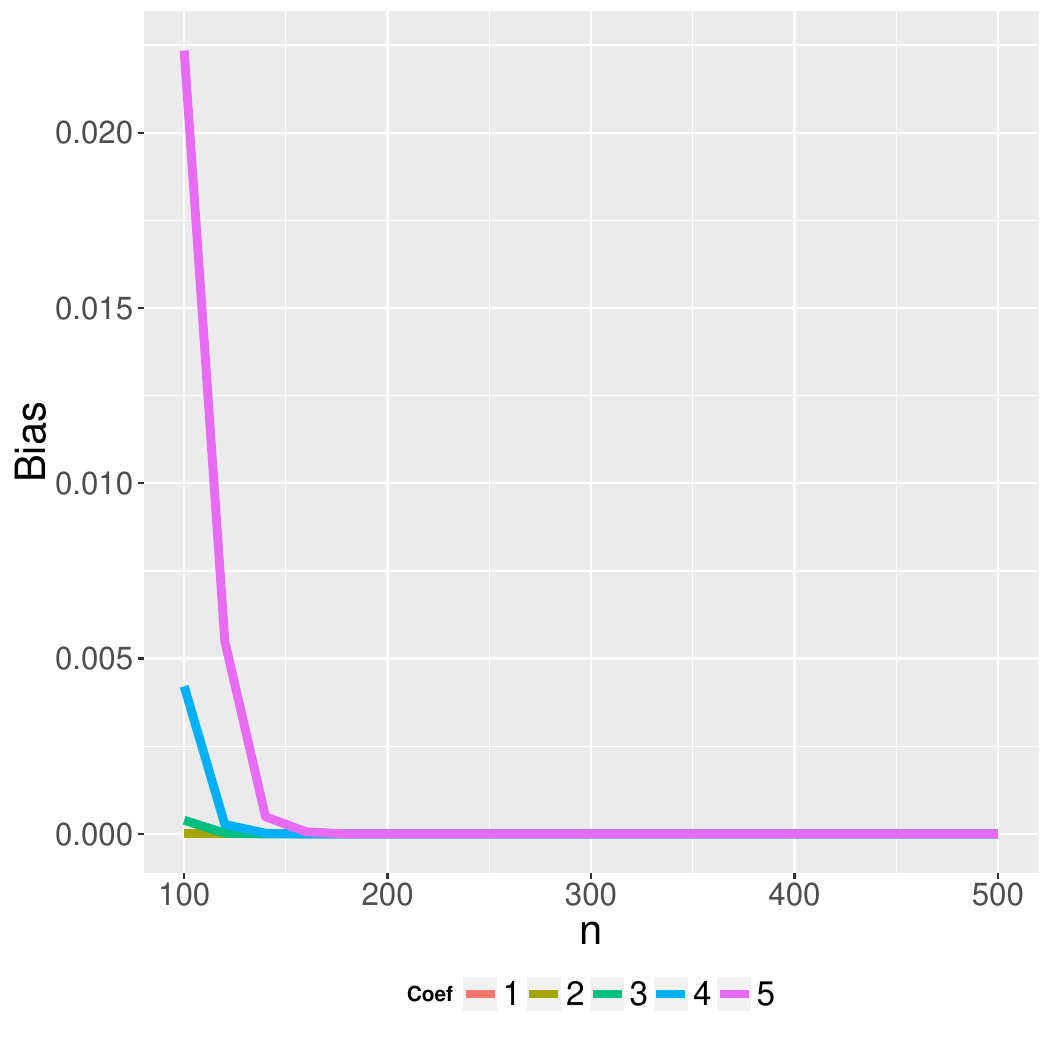}\\
(c) & (d)
\end{tabular}
\caption{{\bf (a)} Evolution of the absolute value of the bias as a function of $k$, for different values of $n$ (plain lines). The dashed lines correspond to the upper bound obtained in \eqref{Equ : BiasLpO}. 
	{\bf (b)} Same as previous, except that data were generated according to the \textbf{DS} setting with parameters $(\pi_0,\eta_0,\eta_1)=(0.2,0.2,0.9)$. Upper bounds are not displayed in order to fit the scale of the absolute bias.
{\bf (c)} Evolution of the absolute value of the bias with respect to $n$, when $k$ is chosen such that $k= \floor{\mbox{Coef} \times n }$ ($\floor{\cdot}$ denotes the integer part). The different colors correspond to different values of $\mbox{Coef}$. 
{\bf (d)} Same as previous, except that $k$ is chosen such that $k=\floor{\mbox{Coef} \times \sqrt{n}}$.}
\label{Fig: EvolBias}
\end{figure}

\subsubsection{Mean squared error}

Following an example described by \cite{DevWag_79potential}, we now provide a lower bound on the minimal convergence rate of the mean squared error \cite[see also][Chap.~24.4, p.415  for a similar argument]{DeGyLu_1996}.
\begin{prop}\label{res.lower.bound.EQM}
	Let us assume $n$ is even, and that $P( Y=1\mid X )  =P(Y=1 )  = 1/2$ is independent of $X$.
Then for $k=n-1$ ($k$ odd), it results
\begin{align*}
\E\croch{ \paren{ \Rh_{1,n}-L(\hat f_k) }^2 } = \int_{0}^1 2 t \cdot \P\croch { \abs{ \Rh_{1,n}-L(\hat f_k) }> t }\, dt \geq \frac{1}{8\sqrt{\pi}} \cdot \frac{1}{\sqrt{n}} \enspace \cdot
\end{align*}
\end{prop}
From the upper bound of order $\sqrt{k}/n$ (with $p=1$) provided by Ineq.~\eqref{ineq.upper.bound.squared.deviation.moment}, choosing $k=n-1$ leads to the same $1/\sqrt{n}$ rate as that of Proposition~\ref{res.lower.bound.EQM}.
This suggests that, at least for very large values of $k$, the  $\sqrt{k}/n$ rate is of the right order and cannot be improved in the distribution-free framework.

\subsection{Minimax rates}

Let us conclude this section with a corollary, which provides a finite-sample bound on the gap between $\Rh_{p,n}$ and $R(\hat f_k) = \E\croch{L(\hat f_k)}$ with high probability. It is stated under the same restriction on $p$ as the previous Theorem~\ref{Prop : Consistency} it is based on, that is for $p\leq \sqrt{k}$.
\begin{cor}\label{Prop : UperBound}
With the notation of Theorems~\ref{Prop : ConcentrationLpO:bis} and~\ref{Prop : Consistency}, let us assume $p,k \geq 1$ with $p \leq \sqrt{k}$, $\sqrt{k}+k\leq n$, and $ p \leq n/2+1$.
Then for every $x>0$, there exists an event with probability at least $1 - 2e^{-x}$ such that
\begin{align}\label{ineq.confidence.bound}
\abs{ R(\hat{f}_k)) - \Rh_{p,n}}  & \leq \sqrt{ \frac{\Delta^2 k^2 }{n \paren{1-\frac{p-1}{n}} } x } + \frac{4}{\sqrt{2\pi}}\frac{p\sqrt{k}}{n} \enspace,
\end{align}
where $\hat f_k=\A_k^{\Dn}(\cdot)$.
\end{cor}
\begin{proof}[Proof of Corollary~\ref{Prop : UperBound}]
	Ineq.~\eqref{ineq.confidence.bound} results from combining
	the exponential concentration result derived for $\Rh_{p,n}$, namely Ineq.~\eqref{ineq.exponential.concentration.small.p} (from Theorem~\ref{Prop : ConcentrationLpO:bis}) and the upper bound on the bias, that is Ineq.~\eqref{Equ : BiasLpO}.
	\begin{align*}
	\abs{  R(\hat{f}_k)) - \Rh_{p,n} } & \leq  \abs{  R(\hat{f}_k)) - \E\croch{ \Rh_{p,n} } } + \abs{ \E\croch{ \Rh_{p,n} } -  \Rh_{p,n} }  \\
	& \leq  \frac{4}{\sqrt{2\pi}}\frac{p\sqrt{k}}{n}  +\sqrt{ \frac{\Delta^2 k^2 }{n-p+1} x }  \enspace\cdot
	\end{align*}
\end{proof}
Note that the right-hand side of Ineq.~\eqref{ineq.confidence.bound} could be used to derive bounds on $R(\hat f_k)$ that seem similar to confidence bounds.
However we do not recommend doing this in practice for several reasons.
On the one hand, Ineq.~\eqref{ineq.confidence.bound} results from the repeated use of
concentration inequalities where numeric constants are not optimized at all, which leads to require a large sample size $n$ for the deviation terms to be small in practice.
On the other hand, explicit numeric constants such as $\Delta^2$ in Corollary~\ref{Prop : UperBound} exhibit a dependence on $\gamma_d \approx 4.8^d-1$, which becomes exponentially large as $d$ increases. Proving that this dependence can be weakened or not remains a completely open question at this stage. Nevertheless one can highlight that, for a given $n$, increasing $d$ will quickly make the deviation term larger than 1, whereas both $R(\hat f_k)$ and $\Rh_{p,n}$ belong to $[0,1]$.

The right-most term of order $\sqrt{k}/n$ in Ineq.~\eqref{ineq.confidence.bound} results from the bias. This is a necessary price to pay which cannot be improved in the present distribution-free framework according to Proposition~\ref{res.counter.example.bias}.
Besides combining the restriction $p\leq \sqrt{k}$ with the usual consistency constraint $k/n = o(1)$ leads to the conclusion that small values of $p$ (w.r.t. $n$) have almost no effect on the convergence rate of the L$p$O estimator. Weakening the key restriction $p\leq \sqrt{k}$ would be necessary to potentially nuance this conclusion.

\medskip

In order to highlight the interest of the above deviation inequality, let us deduce an optimality result in terms of minimax rate.
In the following statement, Corollary~\ref{Prop : UperBound} is used to prove that, uniformly with respect to $k$, the L$p$O estimator $\Rh_{p,n}$ and the risk $R(\hat f_k)$ of the $k$NN classifier remain close to each other with high probability.
\begin{prop}\label{res.optimal.minimax.rate}
	With the same notation as Corollary~\ref{Prop : UperBound}, for every $C>1$ and $\theta>0$, there exists an event of probability at least $1-2\cdot n^{-(C-1)}$ on which, for any $p,k \geq 1$ such that $p\leq \sqrt{k}$, $k+\sqrt{k}\leq n$, and $p\leq n/2+1$, the L$p$O estimator of the $k$NN classifier satisfies
		 \begin{align}
		 & \paren{1-\theta} \croch{ R(\hat{f}_k) - L^\star  } - \frac{\theta^{-1}\Delta^2 C }{4}  \frac{k^2\log(n)}{n\paren{R(\hat{f}_k) - L^\star}} - \frac{4}{\sqrt{2\pi} } \frac{p\sqrt{k}}{n}  \nonumber \\
		 &\leq \Rhp(\A_k,\Dn) - L^\star    \leq  \paren{1+\theta} \croch{ R(\hat{f}_k) - L^\star  }  + \frac{\theta^{-1}\Delta^2 C }{4}  \frac{k^2\log(n)}{n\paren{R(\hat{f}_k) - L^\star}} + \frac{4}{\sqrt{2\pi} } \frac{p\sqrt{k}}{n} ,\label{ineq.uniform.control.LpO.Holder}
		 \end{align}
	where $L^\star$ denotes the classification error of the Bayes classifier.

Furthermore if one assumes the regression function $\eta$ belongs to a H\"older ball $\mathcal{H}(\tau,\alpha)$ for some $\alpha \in ]0,d/4[$ (recall that $X_i\in \R^d$) and $\tau>0$, then choosing $k = k^\star = k_0 \cdot n^{\frac{2\alpha}{2\alpha+d}}$ leads to
 		 \begin{align}\label{eq.uniform.control.LpO.Holder.equiv}
 \Rhp(\A_{k^\star},\Dn) - L^\star   \sim_{n\to +\infty}  R(\hat{f}_{k^\star}) - L^\star  .
 		 \end{align}
\end{prop}
Ineq.~\eqref{ineq.uniform.control.LpO.Holder} gives a uniform control (over $k$) of the gap between the excess risk	$R(\hat{f}_{k}) - L^\star $ and the corresponding L$p$O estimator $ \Rhp(\hat{f}_{k}) - L^\star$ with high probability.
The decreasing rate (in $n^{-(C-1)}$) of this probability is directly related to the $\log(n)$ factor in the lower and upper bounds. This decreasing rate could be made faster at the price of increasing the exponent of the $\log(n)$ factor.
In a similar way the numeric constant $\theta$ has no precise meaning and can be chosen as close to 0 as we want, leading to increase one of the other deviation terms by a numeric factor $\theta^{-1}$. For instance one could choose $\theta=1/\log(n)$, which would replace the $\log(n)$ by a $\paren{\log n}^2$.

The equivalence stated by \eqref{eq.uniform.control.LpO.Holder.equiv} results from knowing that this choice $k=k^\star$ makes the $k$NN classifier achieve the minimax rate $n^{-\frac{\alpha}{2\alpha+d}}$ over H\"older balls with smoothness parameter $\alpha \in ]0,d/4[$ \citep[see Theorems~3.3 and~3.5 in][]{Aud_Tsyb_2007}. Therefore it is not difficult to check that the other deviation terms are negligible with respect to the excess risk 	 $R(\hat{f}_{k}) - L^\star $ for $k=k^\star$.

\begin{proof}[Proof of Proposition~\ref{res.optimal.minimax.rate}]
Let us define $K\leq n$ as the maximum value of $k$ and assume $x_k = C \cdot \log(n)$ (for some constant $C>1$) for any $1\leq k\leq K$.
Let us also introduce the event
\begin{align*}
\Omega_n = \acc{ \forall 1\leq k\leq K,\  \abs{ R(\hat{f}_k) - \Rhp(\A_k,\Dn) } \leq \sqrt{ \Delta^2 \frac{k^2}{n} x_k } + \frac{4}{\sqrt{2\pi} } \frac{p\sqrt{k}}{n} }.
\end{align*}
Then $\P\croch{ \Omega_n^c  } \leq \frac{1}{n^{C-1}} \to 0$, as $n\to +\infty$, since a union bound leads to
	\begin{align*}
	\sum_{k=1}^K e^{-x_k} & = \sum_{k=1}^K e^{-C \cdot \log(n) }  = K \cdot e^{-C \cdot \log(n) } \leq  e^{- (C-1) \cdot \log(n) }  = \frac{1}{n^{C-1}} \enspace \cdot
	\end{align*}

Furthermore combining (for $a,b>0$) the inequality $ab \leq a^2 \theta^2  + b^2 \theta^{-2}/4 $ for every $\theta>0$ with $\sqrt{a+b}\leq \sqrt{a}+\sqrt{b}$, it results that
\begin{align*}
\sqrt{ \Delta^2 \frac{k^2}{n} x_k } & \leq \theta \paren{R(\hat{f}_k) - L^\star} + \frac{\theta^{-1}}{4} \Delta^2 \frac{k^2}{n\paren{R(\hat{f}_k) - L^\star}} x_k \\
& \leq \theta \paren{R(\hat{f}_k) - L^\star} + \frac{\theta^{-1}}{4} \Delta^2 \frac{k^2}{n\paren{R(\hat{f}_k) - L^\star}} C \cdot \log(n) ,
\end{align*}
hence Ineq.~\eqref{ineq.uniform.control.LpO.Holder}.

\medskip

Let us now choose $k=k^\star$. Then Theorems~3.3 and~3.5 in \cite{Aud_Tsyb_2007} combined with Theorem~7 in \cite{Chaudhuri_Dasgupta2014} provide that the minimax rate of the $k$NN classifier is
\begin{align*}
\paren{R(\hat{f}_{k^\star}) - L^\star} \asymp n^{-\frac{\alpha}{2\alpha+d}} ,
\end{align*}
where $a\asymp b$ means there exist numeric constants $l,u>0$ such that $l\cdot b \leq a \leq u\cdot b$.

It is then easy to check that
\begin{itemize}
	\item $\frac{\theta^{-1}}{4} \Delta^2 \frac{{k^\star}^2}{n\paren{R(\hat{f}_{k^\star}) - L^\star}} C \cdot \log(n) = \frac{ C \theta^{-1}\Delta^2 k_0}{4}  \cdot \paren{ n^{ -\frac{ d-3\alpha }{2\alpha+d}}  \log(n) }  = o_{n\to =\infty}\paren{R(\hat{f}_{k^\star}) - L^\star}$,
	
	\item 	 $\frac{p\sqrt{k^\star}}{n} \leq \frac{k^\star}{n} = k_0 \cdot n^{-\frac{d}{2\alpha+d}}= o_{n\to +\infty}\paren{ R(\hat{f}_{k^\star}) - L^\star } $.
\end{itemize}
The desired conclusion \eqref{eq.uniform.control.LpO.Holder.equiv} finally results from choosing $\theta=1/(\log n)$.
\end{proof}


\section{Discussion} \label{Section: Discussion}


The present work provides several new results quantifying the performance of the L$p$O estimator applied to the $k$NN classifier. 
By exploiting the connexion between L$p$O and U-statistics (Section~\ref{Section: Context}), the polynomial and exponential inequalities derived in Sections~\ref{Section: Polynomial} and~\ref{Section: Deviations} give some new insight on the concentration of the L$p$O estimator around its expectation for different regimes of $p/n$.
In Section~\ref{Section:Old}, these results serve for instance to conclude to the consistency of the L$p$O estimator towards the risk (or the classification error rate) of the $k$NN classifier (Theorem~5.1). They also allow us to establish the asymptotic equivalence between the L$p$O estimator (shifted by the Bayes risk $L^\star$) and the excess risk over some H\"older class of regression functions (Proposition~\ref{res.optimal.minimax.rate}).

\medskip

It is worth mentioning that the upper-bounds derived in Sections 4 and 5 --- see for instance Theorem~\ref{Prop : Consistency} --- can be minimized by choosing $p=1$, suggesting that the L1O estimator is optimal in terms of risk estimation when applied to the $k$NN classification algorithm.
This observation corroborates the results of the simulation study presented in \citet{CelisseMaryHuard11}, where it is empirically shown that small values of $p$ (and in particular $p=1$) lead to the best estimation of the risk, whatever the value of parameter $k$ or the level of noise in the data.
The suggested optimality of L1O (for risk estimation) is also consistent with results by \citet{Burm89} and \citet{Celi_2014}, where it is proved that L1O is asymptotically the best cross-validation procedure to perform risk estimation in the context of low-dimensional regression and density estimation respectively.
%

\medskip

Alternatively, the L$p$O estimator can also be used as a data-dependent calibration procedure to choose $k$: the value $\hat{k}_p$ leading to the minimum L$p$O estimate is selected.
Although the focus of the present paper is different, it is worth mentioning that the concentration results established in Section~\ref{Section: Deviations} are a significant early step towards deriving theoretical guarantees on L$p$O as a model selection procedure.
Indeed, exponential concentration inequalities have been a key ingredient to assess model selection consistency or model selection efficiency in various contexts (see for instance \citet{Celi_2014} or \citet{ArlotLerasle2012} in the density estimation framework).
%
%
%
%
%
Still theoretically investigating the behavior of $\hat{k}_p$ requires some further dedicated developments.
One first step towards such results is to derive a tighter upper bound on the bias between the L$p$O estimator and the risk.
The best known upper bound currently available is derived from \citet[][see Lemma~\ref{Lemma : HOresult} in the present paper]{DevroyeWagner80}. Unfortunately it does not fully capture the true behavior of the L$p$O estimator with respect to $p$ (at least as $p$ becomes large) and could be improved in particular for $p> \sqrt{k}$ as emphasized in the comments following Theorem~\ref{Prop : Consistency}.
Another important direction for studying the model selection behavior of the L$p$O procedure is to prove a concentration inequality for the classification error rate of the $k$NN classifier around its expectation.
While such concentration results have been established for the $k$NN algorithm in the (fixed-design) regression framework \citep{ArlotBach09}, deriving similar results in the classification context remains a challenging problem to the best of our knowledge.

%
%
%
%
%
%
%
%
%


\bibliographystyle{alpha}
\bibliography{bibliography}

\begin{thebibliography}{51}
\providecommand{\natexlab}[1]{#1}
\providecommand{\url}[1]{\texttt{#1}}
\expandafter\ifx\csname urlstyle\endcsname\relax
  \providecommand{\doi}[1]{doi: #1}\else
  \providecommand{\doi}{doi: \begingroup \urlstyle{rm}\Url}\fi

\bibitem[Andoni and Indyk(2006)]{Andoni06}
A.~Andoni and P.~Indyk.
\newblock Near-optimal hashing algorithms for approximate nearest neighbor in
  high dimensions.
\newblock In \emph{Foundations of Computer Science, 2006. FOCS'06. 47th Annual
  IEEE Symposium on}, pages 459--468. IEEE, 2006.

\bibitem[Arlot(2007)]{Arl:2007:phd}
S.~Arlot.
\newblock \emph{Resampling and Model Selection}.
\newblock PhD thesis, University Paris-Sud 11, December 2007.
\newblock URL \url{http://tel.archives-ouvertes.fr/tel-00198803/en/}.
\newblock oai:tel.archives-ouvertes.fr:tel-00198803\_v1.

\bibitem[Arlot and Bach(2009)]{ArlotBach09}
S.~Arlot and F.~Bach.
\newblock Data-driven calibration of linear estimators with minimal penalties.
\newblock \emph{Advances in Neural Information Processing Systems (NIPS)},
  2:\penalty0 46--54, 2009.

\bibitem[Arlot and Celisse(2010)]{ArCe_2010_survey}
S.~Arlot and A.~Celisse.
\newblock A survey of cross-validation procedures for model selection.
\newblock \emph{Statistics Surveys}, 4:\penalty0 40--79, 2010.

\bibitem[Arlot and Lerasle(2012)]{ArlotLerasle2012}
S.~Arlot and M.~Lerasle.
\newblock Why v= 5 is enough in v-fold cross-validation.
\newblock \emph{arXiv preprint arXiv:1210.5830}, 2012.

\bibitem[Audibert and Tsybakov(2007)]{Aud_Tsyb_2007}
J.-Y. Audibert and A.~Tsybakov.
\newblock Fast learning rates for plug-in estimators under the margin
  condition.
\newblock \emph{The Annals of Statistics}, 35\penalty0 (2), 2007.

\bibitem[Biau and Devroye(2016)]{BiauDevroye_2016}
G.~Biau and L.~Devroye.
\newblock \emph{Lectures on the nearest neighbor method}.
\newblock Springer, 2016.

\bibitem[Biau et~al.(2010{\natexlab{a}})Biau, C{\'e}rou, and
  Guyader]{Biau_Cerou_Guyader:2010}
G.~Biau, F.~C{\'e}rou, and A.~Guyader.
\newblock On the rate of convergence of the bagged nearest neighbor estimate.
\newblock \emph{The Journal of Machine Learning Research}, 11:\penalty0
  687--712, 2010{\natexlab{a}}.

\bibitem[Biau et~al.(2010{\natexlab{b}})Biau, C{\'e}rou, and
  Guyader]{Biau_Cerou_Guyader:2010_function}
G.~Biau, F.~C{\'e}rou, and A.~Guyader.
\newblock Rates of convergence of the functional-nearest neighbor estimate.
\newblock \emph{Information Theory, IEEE Transactions on}, 56\penalty0
  (4):\penalty0 2034--2040, 2010{\natexlab{b}}.

\bibitem[Boucheron et~al.(2005)Boucheron, Bousquet, Lugosi, and
  Massart]{Bou_Bou_Lug_Mas:2005}
S.~Boucheron, O.~Bousquet, G.~Lugosi, and P.~Massart.
\newblock Moment inequalities for functions of independent random variables.
\newblock \emph{Ann. Probab.}, 33\penalty0 (2):\penalty0 514--560, 2005.
\newblock ISSN 0091-1798.

\bibitem[Boucheron et~al.(2013)Boucheron, Lugosi, and Massart]{BouLugMas_2013}
S.~Boucheron, G.~Lugosi, and P.~Massart.
\newblock \emph{Concentration Inequalities: A Nonasymptotic Theory of
  Independence}.
\newblock Oxford University Press, 2013.

\bibitem[Burman(1989)]{Burm89}
P.~Burman.
\newblock {C}omparative study of {O}rdinary {C}ross-{V}alidation, v-{F}old
  {C}ross-{V}alidation and the repeated {L}earning-{T}esting {M}ethods.
\newblock \emph{Biometrika}, 76\penalty0 (3):\penalty0 503--514, 1989.

\bibitem[Celisse(2008)]{Cel:2008:phd}
A.~Celisse.
\newblock \emph{{M}odel selection via cross-validation in density estimation,
  regression and change-points detection. (In English)}.
\newblock PhD thesis, University Paris-Sud 11.
  \texttt{http://tel.archives-ouvertes.fr/tel-00346320/en/}., December 2008.
\newblock URL \url{http://tel.archives-ouvertes.fr/tel-00346320/en/}.

\bibitem[Celisse(2014)]{Celi_2014}
A.~Celisse.
\newblock Optimal cross-validation in density estimation with the $l^2$-loss.
\newblock \emph{The Annals of Statistics}, 42\penalty0 (5):\penalty0
  1879--1910, 2014.

\bibitem[Celisse and Mary-Huard(2011)]{CelisseMaryHuard11}
A.~Celisse and T.~Mary-Huard.
\newblock Exact cross-validation for knn: applications to passive and active
  learning in classification.
\newblock \emph{JSFdS}, 152\penalty0 (3), 2011.

\bibitem[Celisse and Robin(2008)]{CeRo08}
A.~Celisse and S.~Robin.
\newblock {N}onparametric density estimation by exact leave-p-out
  cross-validation.
\newblock \emph{Computational Statistics and Data Analysis}, 52\penalty0
  (5):\penalty0 2350--2368, 2008.

\bibitem[Chaudhuri and Dasgupta(2014)]{Chaudhuri_Dasgupta2014}
K.~Chaudhuri and S.~Dasgupta.
\newblock Rates of convergence for nearest neighbor classification.
\newblock In \emph{Advances in Neural Information Processing Systems}, pages
  3437--3445, 2014.

\bibitem[Cover(1968)]{Cover68}
T.~M. Cover.
\newblock Rates of convergence for nearest neighbor procedures.
\newblock In \emph{Proceedings of the Hawaii International Conference on
  Systems Sciences}, pages 413--415, 1968.

\bibitem[Cover and Hart(1967)]{Cover_Hart:1967}
T.~M. Cover and P.~E. Hart.
\newblock Nearest neighbor pattern classification.
\newblock \emph{Information Theory, IEEE Transactions on}, 13\penalty0
  (1):\penalty0 21--27, 1967.

\bibitem[Devroye and Wagner(1979{\natexlab{a}})]{DevWag_79potential}
L.~Devroye and T.~Wagner.
\newblock Distribution-free performance bounds for potential function rules.
\newblock \emph{IEEE Transactions on Information Theory}, 25:\penalty0
  601--604, 1979{\natexlab{a}}.

\bibitem[Devroye et~al.(1996)Devroye, Gy{\"o}rfi, and Lugosi]{DeGyLu_1996}
L.~Devroye, L.~Gy{\"o}rfi, and G.~Lugosi.
\newblock \emph{A Probilistic Theory of Pattern Recognition}.
\newblock Springer Verlag, 1996.

\bibitem[Devroye and Wagner(1977)]{Dev_Wag:1977}
L.~P. Devroye and T.~J. Wagner.
\newblock The strong uniform consistency of nearest neighbor density estimates.
\newblock \emph{Ann. Statist.}, 5\penalty0 (3):\penalty0 536--540, 1977.
\newblock ISSN 0090-5364.

\bibitem[Devroye and Wagner(1979{\natexlab{b}})]{DeWa79}
L.~P. Devroye and T.~J. Wagner.
\newblock Distribution-free inequalities for the deleted and holdout error
  estimates.
\newblock \emph{Information Theory, IEEE Transactions on}, 25\penalty0
  (2):\penalty0 202--207, 1979{\natexlab{b}}.

\bibitem[Devroye and Wagner(1980)]{DevroyeWagner80}
L.P. Devroye and T.J. Wagner.
\newblock Distribution-free consistency results in nonparametric discrimination
  and regression function estimation.
\newblock \emph{Ann. Statist.}, 8\penalty0 (2):\penalty0 231--239, 1980.

\bibitem[Fix and Hodges(1951)]{FixHodges51}
E.~Fix and J.~Hodges.
\newblock \emph{Nearest Neighbor (NN) Norms: NN Pattern Classification
  Techniques}, chapter Discriminatory analysis- nonparametric discrimination:
  Consistency principles.
\newblock {IEEE} Computer Society Press, Los Alamitos, {CA}, 1951.
\newblock Reprint of original work from 1952.

\bibitem[Fuchs et~al.(2013)Fuchs, Hornung, De~Bin, and Boulesteix]{FHDBB:2013}
M.~Fuchs, R.~Hornung, R.~De~Bin, and A.-L. Boulesteix.
\newblock A u-statistic estimator for the variance of resampling-based error
  estimators.
\newblock Technical report, arXiv, 2013.

\bibitem[Geisser(1975)]{Gei:1975}
S.~Geisser.
\newblock {T}he predictive sample reuse method with applications.
\newblock \emph{J. Amer. Statist. Assoc.}, 70:\penalty0 320--328, 1975.

\bibitem[Gy\"{o}rfi(1981)]{Gyorfi_1981}
L.~Gy\"{o}rfi.
\newblock The rate of convergence of $k_n$-nn regression estimates and
  classification rules.
\newblock \emph{IEEE Trans. Commun}, 27\penalty0 (3):\penalty0 362--364, 1981.

\bibitem[Hall et~al.(2008)Hall, Park, and Samworth]{Hall_Park_Samworth:2008}
P.~Hall, B.~U. Park, and R.~J. Samworth.
\newblock Choice of neighbor order in nearest-neighbor classification.
\newblock \emph{The Annals of Statistics}, pages 2135--2152, 2008.

\bibitem[Hastie et~al.(2001)Hastie, Tibshirani, and Friedman]{Has_Tib_Fri:2001}
T.~Hastie, R.~Tibshirani, and J.~Friedman.
\newblock \emph{The elements of statistical learning}.
\newblock Springer Series in Statistics. Springer-Verlag, New York, 2001.
\newblock ISBN 0-387-95284-5.
\newblock Data mining, inference, and prediction.

\bibitem[Hoeffding(1963)]{Hoef:1963}
W.~Hoeffding.
\newblock Probability inequalities for sums of bounded random variables.
\newblock \emph{Journ. of the American Statistical Association}, 58\penalty0
  (301):\penalty0 13--30, 1963.

\bibitem[Ibragimov and Sharakhmetov(2002)]{IbragShar2002}
R.~Ibragimov and S.~Sharakhmetov.
\newblock On extremal problems and best constants in moment inequalities.
\newblock \emph{Sankhy{\=a}: The Indian Journal of Statistics, Series A}, pages
  42--56, 2002.

\bibitem[Indyk and Motwani(1998)]{Indyk98}
P.~Indyk and R.~Motwani.
\newblock Approximate nearest neighbors: towards removing the curse of
  dimensionality.
\newblock In \emph{Proceedings of the thirtieth annual ACM symposium on Theory
  of computing}, pages 604--613. ACM, 1998.

\bibitem[Kearns and Ron(1999)]{KeRo99}
M.~Kearns and D.~Ron.
\newblock {A}lgorithmic {S}tability and {S}anity-{C}heck {B}ounds for
  {L}eave-{O}ne-{O}ut {C}ross-{V}alidation.
\newblock \emph{Neural Computation}, 11:\penalty0 1427--1453, 1999.

\bibitem[Koroljuk and Borovskich(1994)]{Ko_Bo:1994}
V.~S. Koroljuk and Y.~V. Borovskich.
\newblock \emph{Theory of U-statistics}.
\newblock Springer, 1994.

\bibitem[Kulkarni and Posner(1995)]{KulkPosner:1995}
S.~R. Kulkarni and S.~E. Posner.
\newblock Rates of convergence of nearest neighbor estimation under arbitrary
  sampling.
\newblock \emph{Information Theory, IEEE Transactions on}, 41\penalty0
  (4):\penalty0 1028--1039, 1995.

\bibitem[Li et~al.(2004)Li, Umbach, Terry, and Taylor]{Li04}
L.~Li, D.~M. Umbach, P.~Terry, and J.~A. Taylor.
\newblock Application of the ga/knn method to seldi proteomics data.
\newblock \emph{Bioinformatics}, 20\penalty0 (10):\penalty0 1638--1640, 2004.

\bibitem[Massart and N{\'e}d{\'e}lec(2006)]{MassNedelec_2006}
Pascal Massart and {\'E}lodie N{\'e}d{\'e}lec.
\newblock Risk bounds for statistical learning.
\newblock \emph{The Annals of Statistics}, pages 2326--2366, 2006.

\bibitem[Psaltis et~al.(1994)Psaltis, Snapp, and
  Venkatesh]{Psaltis_Snapp_Venkatesh:1994}
D.~Psaltis, R.~R. Snapp, and S.~S. Venkatesh.
\newblock On the finite sample performance of the nearest neighbor classifier.
\newblock \emph{Information Theory, IEEE Transactions on}, 40\penalty0
  (3):\penalty0 820--837, 1994.

\bibitem[Rogers and Wagner(1978)]{RogersWagner78}
W.~H. Rogers and T.~J. Wagner.
\newblock A finite sample distribution-free performance bound for local
  discrimination rules.
\newblock \emph{Annals of Statistics}, 6\penalty0 (3):\penalty0 506--514, 1978.

\bibitem[Scheirer and Slaney(2003)]{Scheirer03}
E.~D. Scheirer and M.~Slaney.
\newblock Multi-feature speech/music discrimination system, May~27 2003.
\newblock US Patent 6,570,991.

\bibitem[Serfling(1980)]{Serf:1980}
R.~J. Serfling.
\newblock \emph{Approximation Theorems of Mathematical Statistics}.
\newblock John Wiley \& Sons Inc., 1980.

\bibitem[Shao(1993)]{Sha:1993}
J.~Shao.
\newblock Linear model selection by cross-validation.
\newblock \emph{J. Amer. Statist. Assoc.}, 88\penalty0 (422):\penalty0
  486--494, 1993.
\newblock ISSN 0162-1459.

\bibitem[Simard et~al.(1998)Simard, LeCun, Denker, and Victorri]{Simard98}
P.~Y. Simard, Y.~A. LeCun, J.~S. Denker, and B.~Victorri.
\newblock Transformation invariance in pattern recognition tangent distance and
  tangent propagation.
\newblock In \emph{Neural networks: tricks of the trade}, pages 239--274.
  Springer, 1998.

\bibitem[Snapp and Venkatesh(1998)]{Snapp_Venkatesh:1998}
R.~R Snapp and S.~S. Venkatesh.
\newblock Asymptotic expansions of the $ k $ nearest neighbor risk.
\newblock \emph{The Annals of Statistics}, 26\penalty0 (3):\penalty0 850--878,
  1998.

\bibitem[Steele(2009)]{Steele_2009}
B.~M. Steele.
\newblock Exact bootstrap k-nearest neighbor learners.
\newblock \emph{Machine Learning}, 74\penalty0 (3):\penalty0 235--255, 2009.

\bibitem[Stone(1982)]{Sto:1982}
C.~J. Stone.
\newblock Optimal global rates of convergence for nonparametric regression.
\newblock \emph{Ann. Statist.}, 10\penalty0 (4):\penalty0 1040--1053, 1982.
\newblock ISSN 0090-5364.

\bibitem[Stone(1974)]{Ston74}
M.~Stone.
\newblock {C}ross-validatory choice and assessment of statistical predictions.
\newblock \emph{J. Roy. Statist. Soc. Ser. B}, 36:\penalty0 111--147, 1974.
\newblock ISSN 0035-9246.
\newblock With discussion by G. A. Barnard, A. C. Atkinson, L. K. Chan, A. P.
  Dawid, F. Downton, J. Dickey, A. G. Baker, O. Barndorff-Nielsen, D. R. Cox,
  S. Giesser, D. Hinkley, R. R. Hocking, and A. S. Young, and with a reply by
  the authors.

\bibitem[Yang(2006)]{Yan:2006}
Y.~Yang.
\newblock Comparing learning methods for classification.
\newblock \emph{Statist. Sinica}, 16\penalty0 (2):\penalty0 635--657, 2006.
\newblock ISSN 1017-0405.

\bibitem[Yang(2007)]{Yang07}
Y.~Yang.
\newblock {C}onsistency of cross-validation for comparing regression
  procedures.
\newblock \emph{The Annals of Statistics}, 35\penalty0 (6):\penalty0
  2450--2473, 2007.

\bibitem[Zhang(1993)]{Zha:1993}
P.~Zhang.
\newblock Model selection via multifold cross validation.
\newblock \emph{Ann. Statist.}, 21\penalty0 (1):\penalty0 299--313, 1993.
\newblock ISSN 0090-5364.

\end{thebibliography}


\newpage
\appendix \label{Section: appendix}
\section{Proofs of polynomial moment upper bounds}

\subsection{Proof of Theorem~\ref{Prop : upper.bound.lpo.loo}}
\label{Appendix: UpperBoundLpO.L1O}
The proof relies on Proposition~\ref{Prop : HoeffdingDecomposition} that allows to relate the L$p$O estimator to a sum of independent random variables.
In the following, we distinguish between the two settings $q=2$ (where exact calculations can be carried out), and $q>2$ where only upper bounds can be derived.

When $q>2$, our proof deals separately with the cases $p\leq n/2+1$ and $p>n/2+1$. 
In the first one, a straightforward use of Jensen's inequality leads to the result.
In the second setting, one has to be more cautious when deriving upper bounds. This is done by using the more sophisticated Rosenthal's inequality, namely Proposition~\ref{prop.rosenthal.inequality}.

\subsubsection{Exploiting Proposition~\ref{Prop : HoeffdingDecomposition}}
According to the proof of Proposition~\ref{Prop : HoeffdingDecomposition}, it arises that the L$p$O estimator can be expressed as a $U$-statistic since
\begin{eqnarray*}
\Rh_{p,n} = \frac{1}{n!} \sum_{\sigma} W\gp{Z_{\sigma(1)},\ldots,Z_{\sigma(n)}} \enspace,
\end{eqnarray*}
with
\begin{eqnarray*}
W\paren{Z_1,\ldots,Z_n} &=& \floor{ \frac{n}{m}}^{-1} \sum_{a = 1} ^{\floor{ \frac{n}{m}}} h_m\paren{ Z_{(a-1)m+1},\ldots,Z_{am}}\qquad \mbox{(with $m=n-p+1$)} \\
\text{and }\ \ \  h_m\paren{ Z_{1},\ldots,Z_{m}} &=& \frac{1}{m} \sum_{i=1}^m \1_{\ga{ \A^{\D_m^{(i)}}( X_i) \neq Y_i }} = \Rh_{1,n-p+1} \enspace,
\end{eqnarray*}
where $\A^{\D_m^{(i)}}( .)$ denotes the classifier based on sample $\D_m^{(i)}=(Z_{1},\ldots,Z_{i-1},Z_{i+1},\ldots,Z_{m})$.
Further centering the L$p$O estimator, it comes
\begin{eqnarray*}
\Rh_{p,n} -\esp{ \Rh_{p,n} } = \frac{1}{n!} \sum_{\sigma} \bar W\gp{Z_{\sigma(1)},\ldots,Z_{\sigma(n)}},
\end{eqnarray*}
where $\bar W(Z_1,\ldots,Z_n) = W(Z_1,\ldots,Z_n) - \E\croch{W(Z_1,\ldots,Z_n)}$.

Then with $\bar h_m(Z_1,\ldots,Z_m) = h_m(Z_1,\ldots,Z_m) - \E\croch{ h_m(Z_1,\ldots,Z_m) }$, one gets
\begin{align}
\esp{ \abs{ \Rh_{p,n} - \esp{\Rh_{p,n}} }^{q}  }
 & \leq  \esp{ \abs{  \bar W\paren{Z_1,\ldots,Z_n} }^{q}  } \quad \mathrm{(Jensen's\ inequality)} \nonumber \\
 & =  \esp{  \abs{ \floor{ \frac{n}{m}}^{-1}  \sum_{i = 1} ^{\floor{ \frac{n}{m}}}  \bar h_m\paren{ Z_{(i-1)m+1},\ldots,Z_{im}} }^{q}  }  \label{ineq.Lpo.Ustat.Jensen} \\
 & =  \floor{ \frac{n}{m}}^{-q} \esp{ \abs{ \sum_{i = 1} ^{\floor{ \frac{n}{m}}}  \bar h_m\paren{ Z_{(i-1)m+1},\ldots,Z_{im}} }^{q}  } \nonumber .
\end{align}

\subsubsection{The setting $q=2$}

If $q=2$, then by independence it comes
\begin{align*}
\esp{ \abs{ \Rh_{p,n} - \esp{\Rh_{p,n}} }^{q}  }
& \leq \floor{ \frac{n}{m}}^{-2} \Var\paren{  \sum_{i = 1} ^{\floor{ \frac{n}{m}}}  h_m\paren{ Z_{(i-1)m+1},\ldots,Z_{im}}  } \\
& = \floor{ \frac{n}{m}}^{-2} \sum_{i = 1} ^{\floor{ \frac{n}{m}}} \Var\croch{  h_m\paren{ Z_{(i-1)m+1},\ldots,Z_{im}}  } \\
& =    \floor{ \frac{n}{m}}^{-1}  \Var\paren{  \Rh_1(\A,Z_{1,n-p+1}) } ,
\end{align*}
which leads to the result.

\subsubsection{The setting $q>2$}

\paragraph{If $p\leq n/2+1$:}

A straightforward use of Jensen's inequality from \eqref{ineq.Lpo.Ustat.Jensen} provides
\begin{align*}
\esp{ \abs{ \Rh_{p,n} - \esp{\Rh_{p,n}} }^{q}  }
 & \leq \floor{ \frac{n}{m}}^{-1} \sum_{i = 1} ^{\floor{ \frac{n}{m}}}  \esp{ \abs{    \bar h_m\paren{ Z_{(i-1)m+1},\ldots,Z_{im}} }^{q}  } \\
 & = \esp{ \abs{  \Rh_{1,n-p+1} -\E\croch{\Rh_{1,n-p+1} } }^{q}  } . \hfill
\end{align*}

\paragraph{If $p> n/2+1$:}
Let us now use Rosenthal's inequality (Proposition~\ref{prop.rosenthal.inequality}) by introducing symmetric random variables $\zeta_1,\ldots,\zeta_{\floor{n/m}}$ such that
\begin{align*}
\forall 1\leq i\leq \floor{n/m}, \quad \zeta_i =  h_m\paren{ Z_{(i-1)m+1},\ldots,Z_{im}} - h_m\paren{ Z^\prime_{(i-1)m+1},\ldots,Z^\prime_{im}},
\end{align*}
where $Z^\prime_1,\ldots,Z^\prime_n$ are \iid copies of $Z_1,\ldots,Z_n$.
Then it comes for every $\gamma>0$
\begin{align*}
\esp{ \abs{ \sum_{i = 1} ^{\floor{ \frac{n}{m}}}  \bar h_m\paren{ Z_{(i-1)m+1},\ldots,Z_{im}} }^{q}  }  \leq \esp{ \abs{ \sum_{i = 1} ^{\floor{ \frac{n}{m}}}  \zeta_i }^{q}  } ,
\end{align*}
which implies
\begin{align*}
\esp{ \abs{ \sum_{i = 1} ^{\floor{ \frac{n}{m}}}  \bar h_m\paren{ Z_{(i-1)m+1},\ldots,Z_{im}} }^{q}  }
%
& \leq B(q,\gamma) \max\acc{ \gamma \sum_{i = 1} ^{\floor{ \frac{n}{m}}}  \E\croch{ \abs{ \zeta_i }^q } , \paren{ \sqrt{ \sum_{i = 1} ^{\floor{ \frac{n}{m}}} \E\croch{ \zeta_i^2 } } }^q } .
\end{align*}
Then using for every $i$ that
\begin{align*}
\E\croch{\abs{\zeta_i}^q }\leq 2^{q} \E\croch{ \abs{ \bar h_m\paren{ Z_{(i-1)m+1},\ldots,Z_{im}}}^q} ,
\end{align*}
it comes
\begin{align*}
& \esp{ \abs{ \sum_{i = 1} ^{\floor{ \frac{n}{m}}}  \bar h_m\paren{ Z_{(i-1)m+1},\ldots,Z_{im}} }^{q}  } \\
& \leq B(q,\gamma) \max\left(  2^q \gamma \floor{ \frac{n}{m}}  \E\croch{ \abs{ \Rh_{1,m} - \E\croch{ \Rh_{1,m} } }^q } , \paren{ \sqrt{ \floor{ \frac{n}{m}} 2 \Var\paren{  \Rh_{1,m} } } }^q \right) .
\end{align*}
Hence, it results for every $q>2$
\begin{align*}
&\esp{ \abs{ \Rh_{p,n} - \esp{\Rh_{p,n}} }^{q} } \\
 & \leq  B(q,\gamma) \max \left(  2^q \gamma \floor{ \frac{n}{m}}^{-q+1}  \E\croch{ \abs{ \Rh_{1,m} - \E\croch{ \Rh_{1,m} } }^q } , \floor{ \frac{n}{m}}^{-q/2} \paren{ \sqrt{ 2 \Var\paren{  \Rh_{1,m} } } }^q \right) ,
\end{align*}
which concludes the proof.

\subsection{Proof of Theorem~\ref{prop.moment.upper.bounds.Loo}}
\label{Appendix: BorneMomentLpO.kNN}
Our strategy of proof follows several ideas.
The first one consists in using Proposition~\ref{cor.generalized.Efron.Stein} which says that, for every $q\geq 2$,
\begin{align*}
\norm{ \bar h_m(Z_1,\ldots,Z_m) }_q \leq \sqrt{2\kappa q} \sqrt{ \norm{ \sum_{j=1}^m \paren{ h_m(Z_1,\ldots,Z_m)-h_m(Z_1,\ldots,Z_j',\ldots,Z_m) }^2  }_{q/2}  },
\end{align*}
where $h_m(Z_1,\ldots,Z_m) = \Rh_{1,m} $ by Eq.~\eqref{eq.notation.LpO.Ustatistic}, and $ \bar h_m(Z_1,\ldots,Z_m) = h_m(Z_1,\ldots,Z_m) - \E\croch{h_m(Z_1,\ldots,Z_m)}$.
The second idea consists in deriving upper bounds of 
$$\Delta^j h_m = h_m(Z_1,\ldots,Z_m)-h_m(Z_1,\ldots,Z_j',\ldots,Z_m)$$ 
by repeated uses of Stone's lemma, that is Lemma~\ref{Stone} which upper bounds by $k\gamma_d$ the maximum number of $X_i$s that can have a given $X_j$ among their $k$ nearest neighbors.
Finally, for technical reasons we have to distinguish the case $q=2$ where we get tighter bounds, and $q>2$.

\subsubsection{Upper bounding $\Delta^j h_m$}
For the sake of readability let us now use the notation $\D^{(i)}= \D_m^{(i)}$ (see Theorem~\ref{theorem.lpo.ustat}), and let $\D_j^{(i)}$ denote the set$\paren{Z_1,\ldots,Z_j^\prime,\ldots,Z_n}$ where the $i$-th coordinate has been removed.
Then, $\Delta^j h_m = h_m(Z_1,\ldots,Z_m)-h_m(Z_1,\ldots,Z_j',\ldots,Z_m)$ is now upper bounded by
\begin{align}
\abs{ \Delta^j h_m }
& \leq \frac{1}{m} + \frac{1}{m} \sum_{i\neq j} \abs{ \ind{ \A_k^{\D^{(i)}}\paren{ X_i }\neq Y_i  } - \ind{  \A_k^{\D_j^{(i)}}\paren{ X_i } \neq Y_i }  } \nonumber\\
& \leq \frac{1}{m} + \frac{1}{m} \sum_{i\neq j} \abs{ \ind{ \A_k^{\D^{(i)}}\paren{X_i } \neq \A_k^{\D_j^{(i)}}\paren{X_i} }  } . 
\end{align}

Furthermore, let us introduce for every $1\leq j\leq n$,
\begin{align*}
  A_j = \acc{ 1\leq i\leq m,\ i\neq j,\ j \in V_k(X_i)} \text{ and } A_j' = \acc{ 1\leq i\leq m,\ i\neq j,\ j \in V_k'(X_i)}
\end{align*}
where $V_k(X_i)$ and $V_k'(X_i)$ denote the indices of the $k$ nearest neighbors of $X_i$ respectively among $X_1,\ldots,X_{j-1},X_{j},X_{j+1},\ldots,X_m$ and $X_1,...,X_{j-1},X_{j}',X_{j+1},\ldots,X_m$.
Setting $  B_{j} = A_j\cup A'_j $, one obtains
\begin{align} \label{maj.diff.hj_hjp}
\abs{  \Delta^j h_m }
& \leq \frac{1}{m} + \frac{1}{m} \sum_{i \in B_j} \abs{ \ind{ \A_k^{\D^{(i)}}\paren{X_i } \neq \A_k^{\D_j^{(i)}}\paren{X_i} }  } .
\end{align}

From now on, we distinguish between $q=2$ and  $q>2$ because we will be able to derive a tighter bound for $q=2$ than for $q>2$.

\subsubsection{Case $q>2$}
From \eqref{maj.diff.hj_hjp}, Stone's lemma (Lemma~\ref{Stone}) provides
\begin{align*}
\abs{  \Delta^j h_m }
& \leq \frac{1}{m} + \frac{1}{m} \sum_{i\in B_j} \ind{ \A_k^{\D^{(i)}}\paren{X_i } \neq \A_k^{\D_j^{(i)}}\paren{X_i} }  
 \leq \frac{1}{m} + \frac{2k\gamma_d}{m} \enspace\cdot
\end{align*}

Summing over $1\leq j \leq n$ and applying $(a+b)^q \leq 2^{q-1} \paren{a^q + b^q}$ ($a,b\geq 0$ and $q\geq 1$), it comes
\begin{align*}
\sum_j\paren{  \Delta^j h_m }^2 \leq  \frac{2}{m}\gp{ 1 + (2k\gamma_d)^2 }\leq  &  \frac{4}{m}(2k\gamma_d)^2 \enspace,
\end{align*}
hence
\begin{align*}
\norm{ \sum_{j=1}^m \paren{ h_m(Z_1,\ldots,Z_m) - h_m(Z_1,\ldots,Z_j',\ldots,Z_m) }^2 }_{q/2} \leq  \frac{4}{m}(2k\gamma_d)^2  .
\end{align*}
This leads for every $q > 2$ to
\begin{align*}
  \norm{ \bar h_m(Z_1,\ldots,Z_m) }_q 
& \leq q^{1/2} \sqrt{2\kappa}  \frac{4k\gamma_d}{\sqrt{m}} \enspace ,
\end{align*}
which enables to conclude.

\subsubsection{Case $q=2$}
It is possible to obtain a slightly better upper bound in the case $q=2$ with the following reasoning.
With the same notation as above and from \eqref{maj.diff.hj_hjp}, one has
\begin{align*}
\esp{ \gp{    \Delta^j h_m }^2 }
& = \frac{2}{m^2} + \frac{2}{m^2} \esp{ \gp{\sum_{i\in B_j}  \ind{ \A_k^{\D^{(i)}}\paren{X_i } \neq \A_k^{\D_j^{(i)}}\paren{X_i} }  }^2} \qquad \mbox{(using $\1_{\acc{\cdot}} \leq 1$)}\\
& \leq \frac{2}{m^2} + \frac{2}{m^2} \esp{ |B_j| \sum_{i\in B_j}  \ind{ \A_k^{\D^{(i)}}\paren{X_i } \neq \A_k^{\D_j^{(i)}}\paren{X_i} }  }.
\end{align*}
Lemma~\ref{Stone} implies $\abs{B_j} \leq 2 k \gamma_d$, which allows to conclude
\begin{align*}
\esp{ \gp{  \Delta^j h_m }^2 }
& \leq \frac{2}{m^2} + \frac{4k\gamma_d}{m^2} \esp{ \sum_{i\in B_j} \ind{ \A_k^{\D^{(i)}}\paren{X_i } \neq \A_k^{\D_j^{(i)}}\paren{X_i} }  }.
\end{align*}
Summing over $j$ and introducing an independent copy of $Z_1$ denoted by $Z_0$, one derives
\begin{align}
&\sum_{j=1}^m\esp{ \gp{   h_m(Z_1,\ldots,Z_m)-h_m(Z_1,\ldots,Z_j',\ldots,Z_m) }^2 } \nonumber \\
%
%
& \leq \frac{2}{m} + \frac{4k\gamma_d}{m} \sum_{i=1}^m\esp{ \ind{ \A_k^{\D^{(i)}}\paren{X_i } \neq  \A_k^{\D^{(i)}\cup Z_0}\paren{ X_i} } +\ind{ \A_k^{\D^{(i)}\cup Z_0}\paren{X_i} \neq  \A_k^{\D_j^{(i)}}\paren{X_i } } }  \nonumber \\
& \leq \frac{2}{m} + 4k\gamma_d \times 2 \frac{4\sqrt{k}}{\sqrt{2\pi}m} = \frac{2}{m} + \frac{32\gamma_d}{\sqrt{2\pi}}\frac{k\sqrt{k}}{m} \leq (2+16\gamma_d) \frac{k\sqrt{k}}{m} \enspace , \label{ineq.technical.proof.improved}
\end{align}
where the last but one inequality results from Lemma~\ref{Lemma : HOresult}.

\subsection{Proof of Theorem~\ref{prop.moment.upper.bounds.Lpo}}
\label{sec.proof.upper.bound.moments.Lpo}

The idea is to plug the upper bounds previously derived for the L1O estimator, namely Ineq.~\eqref{ineq.Lpo.Loo.moments.direct.quad} and~\eqref{ineq.Lpo.Loo.moments.quad} from Theorem~\ref{Prop : upper.bound.lpo.loo}, in the inequalities proved for the moments of the L$p$O estimator in Theorem~\ref{Prop : upper.bound.lpo.loo}.

\bigskip

\paragraph{Proof of  Ineq.~\eqref{ineq.variance.generalized.efron.stein}, \eqref{ineq.p.small.generalized.efron.stein}, and~\eqref{ineq.p.large.variance.generalized.efron.stein}:}
These inequalities straightforwardly result from the combination of Theorem~\ref{Prop : upper.bound.lpo.loo} and Ineq.~\eqref{ineq.Lpo.Loo.moments.direct.quad} and~\eqref{ineq.Lpo.Loo.moments.quad} from Theorem~\ref{prop.moment.upper.bounds.Loo}.

\bigskip

\paragraph{Proof of Ineq.~\eqref{ineq.p.large.generalized.efron.stein}:}
It results from the upper bounds proved in Theorem~\ref{prop.moment.upper.bounds.Loo} and plugged in Ineq.~\eqref{ineq.Lpo.Loo.moments.larger.quad} (derived from Rosenthal's inequality with optimized constant $\gamma$, namely Proposition~\ref{prop.rosenthal.inequality.upper.bound}).

Then it comes
{\small
\begin{align*}
&\E\croch{ \abs{ \Rh_{p,n} - \E\croch{\Rh_{p,n}} }^{q} }
 \leq  \paren{2\sqrt{2e}}^q \times \\
 & \max\acc{ \paren{\sqrt{q}}^q \paren{ \sqrt{ \floor{\frac{n}{n-p+1}}^{-1} 2C_1  \sqrt{k} \paren{\frac{ \sqrt{k}}{\sqrt{n-p+1}}}^2 }}^q ,  q^q \floor{\frac{n}{n-p+1}}^{-q+1} (2C_2\sqrt{q})^{q} \paren{ \frac{k}{\sqrt{n-p+1}}}^q  } \\
& =  \paren{2\sqrt{2e}}^q \times \\
& \max \acc{ \paren{\sqrt{q}}^q \paren{ \sqrt{  2C_1  \sqrt{k} } \sqrt{ \frac{k}{\paren{n-p+1}\floor{\frac{n}{n-p+1}} } } }^q ,  \paren{q^{3/2} }^q \floor{\frac{n}{n-p+1}}  \paren{ 2C_2\frac{k}{\floor{\frac{n}{n-p+1}}\sqrt{n-p+1}} }^q  } \\
& \leq \floor{\frac{n}{n-p+1}} \max\acc{ \paren{ \lambda_1 q^{1/2} }^q , \paren{ \lambda_2 q^{3/2} }^q } ,
\end{align*}}
with
\begin{align*}
  \lambda_1 & =  2\sqrt{2e}   \sqrt{  2C_1  \sqrt{k} } \sqrt{ \frac{k}{\paren{n-p+1}\floor{\frac{n}{n-p+1}} } } , \quad
\lambda_2  =   2\sqrt{2e } 2C_2\frac{k}{\floor{\frac{n}{n-p+1}}\sqrt{n-p+1}} \enspace\cdot
\end{align*}
Finally introducing $\Gamma = 2\sqrt{2e } \max\paren{2C_2,\sqrt{2C1}}$ provides the result.

\newpage

\section{Proofs of exponential concentration inequalities}

\subsection{Proof of Proposition~\ref{prop.concentration.Lpo.DGL.straightforward}}
\label{subsec.bounded.differences.proof}
The proof relies on two successive ingredients: McDiarmid's inequality (Theorem~\ref{thm.bounded.differences.ineq}), and Stone's lemma (Lemma~\ref{Stone}).

First with $\Dn=\D$ and $\D_j = (Z_1,\ldots,Z_{j-1},Z_j',Z_{j+1},\ldots,Z_n)$, let us start by upper bounding $ \abs{  \Rhp\paren{\Dn} -  \Rhp\paren{\D_j}  } $ for every $1\leq j\leq n$.

Using Eq.~\eqref{def.Lpo.estimator}, one has
\begin{align*}
& \abs{  \Rhp\paren{\D} - \Rhp\paren{\D_j}   } \\
 &\leq \frac{1}{p}\sum_{i=1}^n\gp{^n_p}^{-1} \sum_e \abs{\ind{ \A_k^{\D^e}\paren{X_i} \neq Y_i} - \ind{ \A_k^{\D_j^e}\paren{X_i}\neq Y_i} } \ind{i\not\in e}\\
 &\leq  \frac{1}{p}\sum_{i=1}^n\gp{^n_p}^{-1} \sum_e \ind{ \A_k^{\D^e}\paren{X_i} \neq \A_k^{\D_j^e}\paren{X_i} }\ind{i\not\in e} \\
 &\leq \frac{1}{p}\sum_{i\neq j}^n\gp{^n_p}^{-1} \sum_e \gc{\ind{j \in V_k^{\D^e}(X_i)} + \ind{j \in V_k^{\D_j^e}(X_i)}}\ind{i\not\in e} + \frac{1}{p}\gp{^n_p}^{-1} \sum_e \ind{j\not\in e} ,
\end{align*}
where $\D_j^e$ denotes the set of random variables among $\D_j$ having indices in $e$, and $V_k^{\D^e}(X_i)$ (resp. $V_k^{\D_j^e}(X_i)$) denotes the set of indices of the $k$ nearest neighbors of $X_i$ among $\D^e$ (resp. $\D_j^e$).

Second, let us now introduce
\begin{align*}
B_j^{\Enp} = \underset{e\in \Enp}{\cup} \acc{ 1\leq i\leq n,\ i\not\in e\cup\acc{j},\  V_k^{\D_j^e}(X_i)\ni j\ \mbox{or}\ V_k^{\D^e}(X_i)\ni j}.
\end{align*}
Then Lemma~\ref{Stone} implies $\card(B_j^{\Enp}) \leq 2(k+p-1)\gamma_d$, hence
\begin{align*}
\abs{  \Rhp\paren{\Dn} - \Rhp\paren{\D_j}   }  &\leq \frac{1}{p}\sum_{i\in B_j^{\Enp} }\gp{^n_p}^{-1} \sum_e 2\cdot \ind{i\not\in e} + \frac{1}{n}  \leq  \frac{4(k+p-1)\gamma_d}{n} + \frac{1}{n} \enspace\cdot
\end{align*}
The conclusion results from McDiarmid's inequality (Section~\ref{subsec.mc.diarmid}).
%

\subsection{Proof of Theorem~\ref{Prop : ConcentrationLpO}}
\label{Appendix:ConcentrationLpO}
In this proof, we use the same notation as in that of Proposition~\ref{prop.concentration.Lpo.DGL.straightforward}.

The goal of the proof is to provide a refined version of previous Proposition~\ref{prop.concentration.Lpo.DGL.straightforward} by taking into account the 
status of each $X_j$ as one of the $k$ nearest neighbors of a given $X_i$ (or not).

To do so, our strategy is to prove a sub-Gaussian concentration inequality by use of Lemma~\ref{lem.subgaussian.moment.to.exp}, which requires the control of the even moments of the L$p$O estimator $\Rhp$.

Such upper bounds are derived 
\begin{itemize}
	\item First, by using Ineq.~\eqref{ineg.moment.removed.obs} (generalized Efron-Stein inequality), which amounts to control the $q$-th moments of the differences
$$\Rh_p( \D )- \Rhp\paren{\D_j} .$$

	\item Second, by precisely evaluating the contribution of each neighbor $X_i$ of a given $X_j$, that is by computing quantities such as $\P_e\croch{j\in e,\ i\in \bar e,\ j \in V_k^{\D^e}(X_i)  }$, where $\P_e\croch{\cdot}$ denotes the probability measure with respect to the uniform random variable $e$ over $\Enp$, and $V_k^{\D^e}(X_i)$ denotes the indices of the $k$ nearest neighbors of $X_i$ among $X^e = \acc{X_{\ell}, \ell\in e}$.

\end{itemize}

\subsubsection{Upper bounding $\Rh_p( \D )- \Rhp\paren{\D_j}$}

For every $1\leq j \leq n$, one gets 
\begin{align*}
\Rh_p( \D )- \Rhp\paren{\D_j} &= {n\choose p}^{-1}\sum_e \left\{ \ind{j\in \bar e} \frac{1}{p}  \gp{\ind{\A_k^{\D^e}(X_j)\neq Y_j}-\ind{\A_k^{\D^e}(X_j^{\prime})\neq Y_j^{\prime}}}\right.\\
& \hspace{2.5cm}+ \left. \ind{j\in e} \frac{1}{p} \sum_{i\in \bar e} \gp{\ind{\A_k^{\D^e}(X_i)\neq Y_i}-\ind{\A_k^{\D_j^e}(X_i)\neq Y_i}} \right\} .
\end{align*}
Absolute values and Jensen's inequality then provide
\begin{align*}
\abs{  \Rh_p( \D )- \Rhp\paren{\D_j} } 
& \leq  {n\choose p}^{-1}\sum_e \acc{ \ind{j\in \bar e} \frac{1}{p} + \ind{j\in e} \frac{1}{p} \sum_{i\in \bar e} \ind{\A_k^{\D^e}(X_i)\neq \A_k^{\D_j^e}(X_i)} } \\
& \leq  \frac{1}{n}+ {n\choose p}^{-1}\sum_e  \ind{j\in e} \frac{1}{p} \sum_{i\in \bar e} \ind{\A_k^{\D^e}(X_i)\neq \A_k^{\D_j^e}(X_i)}  \\
& =  \frac{1}{n}+ \frac{1}{p} \sum_{i=1}^n \P_e\croch{j\in e,\ i\in \bar e,\ \A_k^{\D^e}(X_i)\neq \A_k^{\D_j^e}(X_i)} .
\end{align*}
where the notation $\P_e$ means the integration is carried out with respect to the random variable $e \in \Enp$, which follows a discrete uniform distribution over the set $\Enp$ of all $n-p$ distinct indices among $\acc{1,\ldots,n}$.

Let us further notice that $\acc{ \A_k^{\D^e}(X_i)\neq \A_k^{\D_j^e}(X_i)} \subset \acc{j \in V_k^{\D^e}(X_i)\cup V_k^{\D_j^e}(X_i) } $, where $V_k^{\D_j^e}(X_i)$ denotes the set of indices of the $k$ nearest neighbors of $X_i$ among $\D_j^e$ with the notation of the proof of Proposition~\ref{prop.concentration.Lpo.DGL.straightforward}.
Then it results
\begin{align*}
& \sum_{i=1}^n \P_e\croch{j\in e,\ i\in \bar e,\ \A_k^{\D^e}(X_i)\neq \A_k^{\D_j^e}(X_i)} \\
& \leq \sum_{i=1}^n \P_e\croch{j\in e,\ i\in \bar e,\ j \in V_k^{\D^e}(X_i)\cup V_k^{\D_j^e}(X_i) } \\
& \leq \sum_{i=1}^n \paren{ \P_e\croch{j\in e,\ i\in \bar e,\ j \in  V_k^{\D^e}(X_i) } + \P_e\croch{j\in e,\ i\in \bar e,\ j \in V_k^{\D^e}(X_i)\cup V_k^{\D_j^e}(X_i) } } \\
& \leq 2 \sum_{i=1}^n  \P_e\croch{j\in e,\ i\in \bar e,\ j \in  V_k^{\D^e}(X_i) } ,
\end{align*}
which leads to
\begin{align*}
\abs{  \Rh_p( \D )- \Rhp\paren{\D_j} } 
 \leq  \frac{1}{n}+ \frac{2}{p} \sum_{i=1}^n  \P_e\croch{j\in e,\ i\in \bar e,\ j \in  V_k^{\D^e}(X_i) }  .
\end{align*}

Summing over $1\leq j\leq n$ the square of the above quantity, it results
\begin{align*}
 \sum_{j=1}^{n}\gp{  \Rh_p( \D )- \Rhp\paren{\D_j}  }^2 
&\leq  \sum_{j=1}^{n}\ga{ \frac{1}{n}+\frac{2}{p} \sum_{i=1}^n \P_e\croch{j\in e,\ i\in \bar e,\ j \in V_k^{\D^e}(X_i) } }^2 \\
&\leq 2 \sum_{j=1}^{n} \frac{1}{n^2} + 2\ga{ \frac{2}{p} \sum_{i=1}^n \P_e\croch{j\in e,\ i\in \bar e,\ j \in V_k^{\D^e}(X_i) } }^2\nonumber\\
&\leq \frac{2}{n} + 8 \sum_{j=1}^{n}\ga{ \frac{1}{p} \sum_{i=1}^n \P_e\croch{j\in e,\ i\in \bar e,\ j \in V_k^{\D^e}(X_i) } }^2 \nonumber .
\end{align*}

\subsubsection{Evaluating the influence of each neighbor}

Further using that
\begin{align*}
& \sum_{j=1}^n \paren{ \frac{1}{p} \sum_{i=1}^n \P_e\croch{j\in e,\ i\in \bar e,\ j \in V_k^{\D^e}(X_i) } }^2 \\
& = \sum_{j=1}^n \frac{1}{p^2}\sum_{i=1}^n \paren{ \P_e\croch{j\in e,\ i\in \bar e,\ j \in V_k^{\D^e}(X_i) } }^2 + \\
& \qquad \sum_{j=1}^n \frac{1}{p^2} \sum_{1\leq i\neq \ell \leq n} \P_e\croch{j\in e,\ i\in \bar e,\ j \in V_k^{\D^e}(X_i) }\P_e\croch{j\in e,\ i\in \bar e,\ j \in V_k^{\D^e}(X_{\ell}) } \\
& = \quad T1 \qquad + \qquad T2\enspace ,
\end{align*}
let us now successively deal with each of these two terms.

\paragraph{Upper bound on $T1$}

First, we start by partitioning the sum over $j$ depending on the rank of $X_j$ as a neighbor of $X_i$ in the whole sample $(X_1,\ldots,X_n)$.
It comes {\small
\begin{align*}
 & =   \sum_{j=1}^{n}  \sum_{i=1}^n \ga{\P_e\croch{j\in e,\ i\in \bar e,\ j \in V_k^{\D^e}(X_i)  } }^2 \\
  & =\sum_{i=1}^n \paren{  \sum_{j \in V_k(X_i)}  \ga{\P_e\croch{j\in e,\ i\in \bar e,\ j \in V_k^{\D^e}(X_i)  } }^2 + \sum_{j \in V_{k+p}(X_i)\setminus V_{k}(X_i)}  \ga{\P_e\croch{j\in e,\ i\in \bar e,\ j \in V_k^{\D^e}(X_i)  } }^2 }.
\end{align*}}
Then Lemma~\ref{BasicResultForResamplingKnn} leads to
\begin{align*}
& \sum_{j \in V_k(X_i)}  \ga{\P_e\croch{j\in e,\ i\in \bar e,\ j \in V_k^{\D^e}(X_i)  } }^2 + \sum_{j \in V_{k+p}(X_i)\setminus V_{k}(X_i)}  \ga{\P_e\croch{j\in e,\ i\in \bar e,\ j \in V_k^{\D^e}(X_i)  } }^2  \\
\leq & \sum_{j \in V_k(X_i)}  \paren{ \frac{p}{n} \frac{n-p}{n-1} }^2  +
\sum_{ j \in V_{k+p}(X_i)\setminus V_{k}(X_i) }  \P_e\croch{j\in e,\ i\in \bar e,\ j \in V_k^{\D^e}(X_i)  } \frac{p}{n} \frac{n-p}{n-1} \\
= &\ k  \paren{ \frac{p}{n} \frac{n-p}{n-1} }^2  + \frac{kp}{n} \frac{p-1}{n-1} \frac{p}{n} \frac{n-p}{n-1} =  k  \paren{ \frac{p}{n}  }^2 \frac{n-p}{n-1} \enspace,
\end{align*}
where the upper bound results from $\sum_j a_j^2 \leq \paren{\max_j a_j} \sum_j a_j$, for $a_j\geq 0$.
It results
\begin{align*}
 T1 = \frac{1}{p^2}  \sum_{j=1}^{n}  \sum_{i=1}^n \ga{\P_e\croch{j\in e,\ i\in \bar e,\ j \in V_k^{\D^e}(X_i)  } }^2 &  \leq  \frac{1}{p^2} n \croch{k  \paren{ \frac{p}{n}  }^2 \frac{n-p}{n-1} }  =  \frac{k}{n}   \frac{n-p}{n-1} \enspace\cdot
\end{align*}

\paragraph{Upper bound on $T2$}

Let us now apply the same idea to the second sum, partitioning the sum over $j$ depending on the rank of $j$ as a neighbor of $\ell$ in the whole sample.
Then,
\begin{align*}
T2& =  \frac{1}{p^2}  \sum_{j=1}^{n}   \sum_{1\leq i\neq \ell \leq n}  \P_e\croch{j\in e,\ i\in \bar e,\ j \in V_k^{\D^e}(X_i)  }\P_e\croch{j\in e,\ \ell\in \bar e,\ j \in V_k^{\D^e}(X_{\ell}) } \\
& \leq   \frac{1}{p^2}  \sum_{i=1}^n  \sum_{\ell \neq i}   \sum_{j \in V_k(X_{\ell}) } \P_e\croch{j\in e,\ i\in \bar e,\ j \in V_k^{\D^e}(X_i)  }  \frac{p}{n}\frac{n-p}{n-1} \\
& \qquad +  \frac{1}{p^2}  \sum_{i=1}^n  \sum_{\ell \neq i}   \sum_{j \in V_{k+p}(X_{\ell})\setminus V_{k}(X_{\ell} } \P_e\croch{j\in e,\ i\in \bar e,\ j \in V_k^{\D^e}(X_i)  }   \frac{kp}{n}\frac{p-1}{n-1} \enspace\cdot
\end{align*}
We then apply Stone's lemma (Lemma~\ref{Stone}) to get
{\small 
\begin{align*}
& T2 \\
& = \frac{1}{p^2}  \sum_{i=1}^n  \sum_{j=1}^n \P_e\croch{j\in e,\ i\in \bar e,\ j \in V_k^{\D^e}(X_i)  } \paren{ \sum_{\ell \neq i}   \1_{j \in V_k(X_{\ell}) }   \frac{p}{n}\frac{n-p}{n-1} + \sum_{\ell \neq i}   \1_{j \in V_{k+p}(X_{\ell})\setminus V_{k}(X_{\ell} }  \frac{kp}{n}\frac{p-1}{n-1} } \\
& \leq \frac{1}{p^2}  \sum_{i=1}^n  \frac{kp}{n} \paren{ k\gamma_d   \frac{p}{n}\frac{n-p}{n-1} + (k+p)\gamma_d \frac{kp}{n}\frac{p-1}{n-1} } =   \gamma_d \frac{k^2}{n} \paren{ \frac{n-p}{n-1} + (k+p)\frac{p-1}{n-1} }  \\
& =  \gamma_d \frac{k^2}{n} \paren{ 1 + (k+p-1)\frac{p-1}{n-1} } .
\end{align*}}

\paragraph{Gathering the upper bounds}

The two previous bounds provide
\begin{align*}
\sum_{j=1}^{n}\ga{ \frac{1}{p} \sum_{i=1}^n \P_e\croch{j\in e,\ i\in \bar e,\ j \in V_k^{\D^e}(X_i)  } }^2 & = T1 + T2 \\
&\leq \frac{k}{n}   \frac{n-p}{n-1} + \gamma_d \frac{k^2}{n} \paren{ 1 + (k+p-1)\frac{p-1}{n-1} },
\end{align*}
which enables to conclude
\begin{align*}
& \sum_{j=1}^{n}\gp{  \Rh_p( \D )- \Rhp\paren{\D_j} }^2 \\
& \leq \frac{2}{n}\paren{1 + 4k + 4k^2\gamma_d \croch{1 + (k+p)\frac{p-1}{n-1}}} \leq \frac{8k^2 (1+\gamma_d )}{n}\croch{1 + (k+p)\frac{p-1}{n-1}} .
\end{align*}

\subsubsection{Generalized Efron-Stein inequality}
Then \eqref{ineg.moment.removed.obs} provides for every $q\geq 1$
\begin{align*}
  \norm{ \Rh_{p,n} - \E\croch{\Rh_{p,n}} }_{2q} \leq 4 \sqrt{\kappa q } \sqrt{ \frac{8(1+\gamma_d )k^2}{n} \croch{1 + (k+p)\frac{p-1}{n-1}} }.
\end{align*}
Hence combined with $q! \geq q^q e^{-q} \sqrt{2\pi q}$, it comes
\begin{align*}
  \E\croch{ \paren{ \Rh_{p,n} - \E\croch{\Rh_{p,n}} }^{2q} } & \leq \paren{ 16 \kappa q }^q \paren{ \frac{8(1+\gamma_d )k^2 }{n} \croch{1 + (k+p)\frac{p-1}{n-1}} }^q \\
& \leq q! \paren{ 16 e\kappa  \frac{8(1+\gamma_d )k^2}{n} \croch{1 + (k+p)\frac{p-1}{n-1}} }^q .
\end{align*}
The conclusion follows from Lemma~\ref{lem.subgaussian.moment.to.exp} with $C = 16 e\kappa  \frac{8(1+\gamma_d )k^2 }{n} \croch{1 + (k+p)\frac{p-1}{n-1}}$.
Then for every $t>0$,
\begin{align*}
 \prob{  \Rh_{p,n} - \E\paren{ \Rh_{p,n} }  > t  } \vee \prob{ \E\paren{ \Rh_{p,n} } - \Rh_{p,n} > t  } 
& \leq  \exp \paren{ -  \frac{  nt^2 }{ 1024 e\kappa  k^2(1+\gamma_d)  \croch{1 + (k+p)\frac{p-1}{n-1}} } } \cdot
\end{align*}


\subsection{Proof of Theorem~\ref{Prop : ConcentrationLpO:bis} and Proposition~\ref{prop.exponential.inequality.p.large.explicit.deviations}}
\label{Appendix:ConcentrationLpObis}

%
%

\subsubsection{Proof of Theorem~\ref{Prop : ConcentrationLpO:bis}}

\paragraph{If $p< n/2+1$:}
~\\
In what follows, we exploit a characterization of sub-Gaussian random variables by their $2q$-th moments (Lemma~\ref{lem.subgaussian.moment.to.exp}).

\medskip

\sloppy
From \eqref{ineq.variance.generalized.efron.stein} and \eqref{ineq.p.small.generalized.efron.stein} applied with $2q$, and further introducing a constant $\Delta = 4 \sqrt{e}\max\paren{ \sqrt{C_1/2} , C_2 }>0$, it comes for every $q\geq 1$
\begin{align}\label{ineq.moments.2q}
\E\croch{ \abs{ \Rh_{p,n} - \E\croch{\Rh_{p,n}} }^{2q} } \leq \paren{ \frac{\Delta^2}{16 e}  \frac{k^2}{n-p+1} }^q  (2q)^{q} \leq \paren{  \frac{\Delta^2}{8}  \frac{k^2}{n-p+1} }^q   q! \enspace,
\end{align}
with $q^q \leq q! e^q/\sqrt{2\pi q} $.
Then Lemma~\ref{lem.subgaussian.moment.to.exp} provides for every $t>0$
\begin{align*}
 \prob{ \Rh_{p,n} - \E\croch{\Rh_{p,n}} >t }  \vee  \prob{ \E\croch{\Rh_{p,n}} - \Rh_{p,n} >t } 
&\leq \exp\paren{- (n-p+1) \frac{t^2}{ \Delta^2  k^2 } } .
\end{align*}

\paragraph{If $p\geq n/2+1$:}
~\\
This part of the proof relies on Proposition~\ref{prop.moment.exponential} which provides an exponential concentration inequality from upper bounds on the moments of a random variable.

\medskip

Let us now use \eqref{ineq.variance.generalized.efron.stein} and \eqref{ineq.p.large.generalized.efron.stein} combined with \eqref{ineq.moment.exponential.deviation}, where $C=\floor{\frac{n}{n-p+1}}$, $q_0=2$, and $\min_j \alpha_j = 1/2$.
This provides for every $t>0$
\begin{align*}
 & \P\croch{ \abs{\Rh_{p,n} - \E\croch{\Rh_{p,n}}} > t } \leq  \floor{\frac{n}{n-p+1}} e \times \\
 &\exp\gc{- \frac{1}{2e} \min\acc{  \paren{n-p+1}\floor{\frac{n}{n-p+1}}\frac{t^2 }{4\Gamma^2 k \sqrt{k}}  , \  \gp{ (n-p+1) \floor{\frac{n}{n-p+1}}^2\frac{ t^2}{4\Gamma^2k^2} }^{1/3} }  }  , \hfill
\end{align*}
where $\Gamma$ arises from Eq.~\eqref{ineq.p.large.generalized.efron.stein}.

\subsubsection{Proof of Proposition~\ref{prop.exponential.inequality.p.large.explicit.deviations}}

As in the previous proof, the derivation of the deviation terms results from Proposition~\ref{prop.moment.exponential}.

\medskip

With the same notation and reasoning as in the previous proof, let us combine  \eqref{ineq.variance.generalized.efron.stein} and \eqref{ineq.p.large.generalized.efron.stein}.
From \eqref{ineq.moment.exponential.error} of  Proposition~\ref{prop.moment.exponential} where $C=\floor{\frac{n}{n-p+1}}$, $q_0=2$, and $\min_j \alpha_j = 1/2$, it results for every $t>0$
\begin{align*}
 \P\croch{ \abs{\Rh_{p,n} - \E\croch{\Rh_{p,n}}} > \Gamma \sqrt{\frac{2e }{(n-p+1)}} \paren{ \sqrt{\frac{k^{3/2}}{\floor{\frac{n}{n-p+1} } } t }  +   2e \frac{k}{ \floor{\frac{n}{n-p+1} } } t^{3/2} } } \leq \floor{\frac{n}{n-p+1}} e \cdot e^{-t} , \hfill
\end{align*}
where $\Gamma>0$ is given by Eq.~\eqref{ineq.p.large.generalized.efron.stein}.


\newpage

\section{Proofs of deviation upper bounds}

\subsection{Proof of Ineq.~\eqref{ineq.upper.bound.squared.deviation.moment}  in Theorem~\ref{Prop : Consistency}}\label{appendix.proof.consistency}

The proof follows the same strategy as that of Theorem~2.1 in \cite{RogersWagner78}.

Along the proof, we will repeatedly use some notation that we briefly introduce here.
First, let us define $Z_0=(X_0,Y_0)$ and $Z_{n+1}=(X_{n+1},Y_{n+1})$ that are independent copies of $Z_1$.
Second to ease the reading of the proof, we also use several shortcuts: $\f{X_0} = \A_k^{\Dn}\paren{X_0}$, and $\fe{X_0} = \A_k^{\D^e}\paren{X_0}$ for every set of indices $e\in\Enp$ (with cardinality $n-p$).

Finally along the proof, $e,e^\prime \in \Enp$ denote two \emph{random variables} which are sets of distinct indices \emph{with discrete uniform distribution over $\Enp$}.
The notation $\P_e$ (resp. $\P_{e,e^\prime}$) means the integration is made with respect to the sample $\D$ and also the random variable $e$ (resp. $\D$ and also the random variables $e,e^\prime$). $\E_e\croch{\cdot}$ and $\E_{e,e^\prime}\croch{\cdot}$ are teh corresponding expectations.
Note that the sample $\D$ and the random variables $e,e^\prime$ are independent from each other, so that computing for instance $\P_e\paren{ i \not\in e}$ amounts to integrating with respect to the random variable $e$ only.

\subsubsection{Main part of the proof}\label{subsubsec.main.proof.combinatorial}

With the notation $L_n = L(\A_k^{\Dn})$, let us start from
\begin{align*}
& \esp{(\Rh_{p,n}-L_n)^2} = \esp{\Rh_p^2(\A_k^{\Dn})} + \esp{L_n^2} - 2\esp{\Rh_{p,n} L_n},
\end{align*}
let us notice that
\begin{align*}
  \esp{L_n^2} =  \P\paren{\f{X_0}\neq Y_0,\f{X_{n+1}}\neq Y_{n+1}} ,
\end{align*}
and
\begin{align*}
\esp{\Rh_{p,n} L_n} =  \P_e\paren{\f{X_0}\neq Y_0,\fe{X_{i}}\neq Y_{i}| \ i \notin e}  \P_e\paren{ i \not\in e} .
\end{align*}
It immediately comes
\begin{align}
& \esp{(\Rh_{p,n}-L_n)^2} \nonumber\\
&=   \esp{\Rh_p^2(\A_k^{\Dn})} -  \P_e\paren{\f{X_0}\neq Y_0,\fe{X_{i}}\neq Y_{i}\mid \ i \notin e}  \P_e\paren{ i \not\in e} \label{Equ : RiskQuadraticPartI}\\
& \ + \croch{ \P\paren{\f{X_0}\neq Y_0,\f{X_{n+1}}\neq Y_{n+1}} -  \P_e\paren{\f{X_0}\neq Y_0,\fe{X_{i}}\neq Y_{i}| \ i \notin e}  \P_e\paren{ i \notin e} } \label{Equ : RiskQuadraticPartII}.
\end{align}
The proof then consists in successively upper bounding the two terms~\eqref{Equ : RiskQuadraticPartI} and~\eqref{Equ : RiskQuadraticPartII} of the last equality.

\medskip

\paragraph{Upper bound of \eqref{Equ : RiskQuadraticPartI}}

First, we have
\begin{eqnarray*}
p^2 \esp{\Rh_p^2(\A_k^{\Dn})} &=& \sum_{i,j} \E_{e,e^\prime} \croch{\ind{\fe{X_i}\neq Y_i}\ind{i \notin e}\ind{\fep{X_j}\neq Y_j}\ind{j \notin e'}} \\
&=& \sum_{i}\E_{e,e^\prime} \croch{\ind{\fe{X_i}\neq Y_i}\ind{i \notin e}\ind{\fep{X_i}\neq Y_i}\ind{i \notin e'}}\\
& & +\sum_{i\neq j}\E_{e,e^\prime} \croch{\ind{\fe{X_i}\neq Y_i}\ind{i \notin e}\ind{\fep{X_j}\neq Y_j}\ind{j \notin e'}} .
\end{eqnarray*}

Let us now introduce the five following events where we emphasize $e$ and $e^\prime$ are random variables with the discrete uniform distribution over $\Enp$:
\begin{eqnarray*}
S^{0}_{i}&=&\{ i \notin e,\ i \notin e' \},\\
S^{1}_{i,j} &=& \{ i \notin e,\ j \notin e',\ i \notin e',\ j \notin e \},\qquad
S^{2}_{i,j} = \{ i \notin e,\ j \notin e',\ i \notin e',\ j \in e \},\\
S^{3}_{i,j} &=& \{i \notin e,\ j \notin e',\ i \in e',\ j \notin e \},\qquad
S^{4}_{i,j} = \{ i \notin e,\ j \notin e',\ i \in e',\ j \in e \} .
\end{eqnarray*}
Then,
\begin{align*}
p^2\esp{\Rh_p^2(\A_k^{\Dn})} 
&= \sum_{i}\P_{e,e^\prime} \paren{\fe{X_i}\neq Y_i,\ \fep{X_i}\neq Y_i | S^{0}_{i}} \probi{e,e'}{S^{0}_{i}}\\
&+\sum_{i\neq j}\sum_{\ell=1}^{4} \P_{e,e^\prime} \paren{\fe{X_i}\neq Y_i,\ \fep{X_i}\neq Y_i | S^{\ell}_{i,j}} \probi{e,e'}{S^{\ell}_{i,j}}\\
&= n\P_{e,e^\prime} \paren{\fe{X_1}\neq Y_1,\ \fep{X_1}\neq Y_1 | S^{0}_{1}} \probi{e,e'}{S^{0}_{1}}\\
&  +n(n-1)\sum_{\ell=1}^{4} \P_{e,e^\prime} \paren{\fe{X_1}\neq Y_1,\ \fep{X_2}\neq Y_2 \mid S^{\ell}_{1,2}} \probi{e,e'}{S^{\ell}_{1,2}} .
\end{align*}
Furthermore since
\begin{eqnarray*}
\frac{1}{p^2}\left[n\probi{e,e'}{S^{0}_{1}}+
n(n-1)\sum_{\ell=1}^{4} \probi{e,e'}{S^{\ell}_{1,2}}\right] = \frac{1}{p^2}\sum_{i,j}\probi{e,e'}{i\notin e,\ j\notin e'} = 1 ,
\end{eqnarray*}
it comes	
\begin{eqnarray}\label{Equ : DecompositionPartI}
\esp{\Rh_p^2(\A_k^{\Dn})} - \P_{e,e^\prime} \paren{\f{X_0}\neq Y_0,\fe{X_{1}}\neq Y_{1}} &=& \frac{n}{p^2}A + \frac{n(n-1)}{p^2}B ,
\end{eqnarray}
where
\begin{align*}
A & = \left[\P_{e,e^\prime} \paren{\fe{X_1}\neq Y_1,\ \fep{X_1}\neq Y_1 \mid S^{0}_{1}}  - \P_{e,e^\prime} \paren{\f{X_0}\neq Y_0,\fe{X_{1}}\neq Y_{1} \mid S^{0}_{1}}\right] \\
& \times  \probi{e,e'}{S^{0}_{1}} ,\\
\mbox{and}\quad B & = \sum_{\ell=1}^{4} \left[ \P_{e,e^\prime} \paren{\fe{X_1}\neq Y_1,\ \fep{X_2}\neq Y_2 \mid S^{\ell}_{1,2}}  - \P_{e,e^\prime} \paren{\f{X_0}\neq Y_0,\fe{X_{1}}\neq Y_{1} \mid S^{\ell}_{1,2}} \right] \\ 
& \times  \probi{e,e'}{S^{\ell}_{1,2}} .
\end{align*}


\noindent$\bullet$ \textbf{Upper bound for $A$:}\\
To upper bound $A$, simply notice that:
\begin{eqnarray*}
A &\leq&\probi{e,e'}{S^{0}_{i}} \leq \probi{e,e'}{i \notin e,\ i \notin e'} \leq\left(\frac{p}{n}\right)^2 .
\end{eqnarray*}

\noindent$\bullet$ \textbf{Upper bound for $B$:}\\
To obtain an upper bound for $B$, one needs to upper bound
\begin{eqnarray}\label{Equ : Difference_absent}
\P_{e,e^\prime} \paren{\fe{X_1}\neq Y_1,\ \fep{X_2}\neq Y_2 \mid S^{\ell}_{1,2}} - \P_{e,e^\prime} \paren{\f{X_0}\neq Y_0,\fe{X_{1}}\neq Y_{1} \mid S^{\ell}_{1,2}} ,
\end{eqnarray}
which depends on $\ell$, i.e. on the fact that index 2 belongs or not to the training set $e$.
\begin{itemize}
    \item If $2 \not\in e$ (i.e. $\ell = 1$ or 3):
    Then, Lemma~\ref{lem.2.notin.e} proves
    \begin{align*}
    \eqref{Equ : Difference_absent} \leq \frac{4p\sqrt{k}}{\sqrt{2\pi}n} \enspace\cdot
    \end{align*}

    \item If $2 \in e$ (i.e. $\ell = 2$ or 4):
    Then, Lemma~\ref{lem.2.in.e} settles
    \begin{eqnarray*}
    \eqref{Equ : Difference_absent} \leq \frac{8\sqrt{k}}{\sqrt{2\pi} (n-p)} + \frac{4p\sqrt{k}}{\sqrt{2\pi} n} \enspace\cdot
    \end{eqnarray*}
\end{itemize}
Combining the previous bounds and Lemma~\ref{lem.combinatoire.prob.reechantillons} leads to
\begin{align*}
B &\leq \gp{\frac{4p\sqrt{k}}{\sqrt{2\pi} n}}
\croch{ \probi{e,e'}{S^{1}_{1,2}}+\probi{e,e'}{S^{3}_{1,2}} }  + \gp{\frac{8\sqrt{k}}{\sqrt{2\pi} (n-p)} + \frac{4p\sqrt{k}}{\sqrt{2\pi} n}} \croch{ \probi{e,e'}{S^{2}_{1,2}}+ \probi{e,e'}{S^{4}_{1,2}} }  \\
&\leq \frac{2\sqrt{2}}{\sqrt{\pi}}\sqrt{k}\left[ \frac{p}{n} \croch{ \probi{e,e'}{S^{1}_{1,2}}+\probi{e,e'}{S^{3}_{1,2}} }
 + \gp{\frac{2}{n-p} + \frac{p}{n}} \croch{ \probi{e,e'}{S^{2}_{1,2}}+ \probi{e,e'}{S^{4}_{1,2}} } \right] \\
&\leq  \frac{2\sqrt{2}}{\sqrt{\pi}}\sqrt{k}\left[\frac{p}{n}\probi{e,e'}{i\notin e,\ j\notin e'} + \frac{2}{n-p}\left(\probi{e,e'}{S^{2}_{1,2}}+ \probi{e,e'}{S^{4}_{1,2}}\right)\right] \\
&\leq \frac{2\sqrt{2}}{\sqrt{\pi}}\sqrt{k}\left[\frac{p}{n}\gp{\frac{p}{n}}^2 + \frac{2}{n-p}\left(
\frac{(n-p)p^2(p-1)}{n^2(n-1)^2} + \frac{(n-p)^2p^2}{n^2(n-1)^2}\right)\right] \\
&\leq \frac{2\sqrt{2}}{\sqrt{\pi}}\sqrt{k}\gp{\frac{p}{n}}^2\left[\frac{p}{n} + \frac{2}{n-1}\right] .
\end{align*}
Back to Eq.~\eqref{Equ : DecompositionPartI}, one deduces
\begin{align*}
\esp{\Rh_p^2(\A_k^{\Dn})} - \P_{e,e^\prime} \paren{\f{X_0}\neq Y_0,\fe{X_{1}}\neq Y_{1}}
&= \frac{n}{p^2}A + \frac{n(n-1)}{p^2}B \\
& \leq \frac{1}{n} + \frac{2\sqrt{2}}{\sqrt{\pi}}\frac{(p+2)\sqrt{k}}{n}\enspace\cdot
\end{align*}

\paragraph{Upper bound of \eqref{Equ : RiskQuadraticPartII}}
First observe that
\begin{eqnarray*}
\P_{e,e^\prime} \paren{\f{X_0}\neq Y_0,\fe{X_{i}}\neq Y_{i}\mid i\notin e} = \P_{e,e^\prime} \paren{\widehat{f}_k^{(-1)}(X_0)\neq Y_0,\fe{X_{n+1}}\neq Y_{n+1}}
\end{eqnarray*}
where $\widehat{f_k}^{(-1)}$ is built on sample $(X_2,Y_2),...,(X_{n+1},Y_{n+1})$. One has
\begin{eqnarray*}
&& \P\paren{\f{X_0}\neq Y_0,\f{X_{n+1}}\neq Y_{n+1}} - \P_{e,e^\prime} \paren{\f{X_0}\neq Y_0,\fe{X_{i}}\neq Y_{i} \mid i\notin e} \\
 &=& \P\paren{\f{X_0}\neq Y_0,\f{X_{n+1}}\neq Y_{n+1}} - \P_{e,e^\prime} \paren{\widehat{f_k}^{(-1)}(X_0)\neq Y_0,\fe{X_{n+1}}\neq Y_{n+1}}\\
&\leq& \P\paren{\f{X_0}\neq \widehat{f_k}^{(-1)}(X_0)} + \P_{e,e^\prime} \paren{\fe{X_{n+1}}\neq \f{X_{n+1}}}\\
&\leq& \frac{4\sqrt{k}}{\sqrt{2\pi} n} + \frac{4p\sqrt{k}}{\sqrt{2\pi} n}\enspace,
\end{eqnarray*}
where we used Lemma~\ref{Lemma : HOresult} again to obtain the last inequality.\\

\noindent\textbf{Conclusion:}

The conclusion simply results from combining bonds \eqref{Equ : RiskQuadraticPartI} and \eqref{Equ : RiskQuadraticPartII}, which leads to
\begin{eqnarray*}
\esp{\left(\Rh_{p,n}-L_n\right)^2} \leq \frac{2\sqrt{2}}{\sqrt{\pi}}\frac{(2p+3)\sqrt{k}}{n} + \frac{1}{n} \enspace\cdot
\end{eqnarray*}

\subsubsection{Combinatorial lemmas}

All the lemmas of the present section are proved with the notation introduced at the beginning of Section~\ref{appendix.proof.consistency}.

\medskip

\begin{lem}\label{lem.combinatoire.prob.reechantillons}
For any $1\leq i\neq j\leq n$,
\begin{align*}
\probi{e,e'}{S^{1}_{i,j}}
&= \frac{\cb{n-2}{n-p}}{\cb{n}{n-p}}\times\frac{\cb{n-2}{n-p}}{\cb{n}{n-p}} ,\qquad
\probi{e,e'}{S^{2}_{i,j}}
= \frac{\cb{n-p-1}{n-2}}{\cb{n}{n-p}}\times\frac{\cb{n-p}{n-2}}{\cb{n}{n-p}} \enspace,\\
\probi{e,e'}{S^{3}_{i,j}}
&= \frac{\cb{n-p}{n-2}}{\cb{n}{n-p}}\frac{\cb{n-p-1}{n-2}}{\cb{n}{n-p}} , \qquad\ 
\probi{e,e'}{S^{4}_{i,j}}
= \frac{\cb{n-p-1}{n-2}}{\cb{n}{n-p}}\times\frac{\cb{n-p-1}{n-2}}{\cb{n}{n-p}} \enspace\cdot
\end{align*}
\end{lem}

\begin{proof}[Proof of Lemma~\ref{lem.combinatoire.prob.reechantillons}]
	Along the proof, we repeatedly exploit the independence of the random variables $e$ and $e^\prime$, which are set of $n-p$ distinct indices with the discrete uniform distribution over $\Enp$.
	
	Note also that an important ingredient is that the probability of each one of the following events does not depend on the particular choice of the indices $(i, j)$, but only on the fact that $i\neq j$.
\begin{eqnarray*}
\probi{e,e'}{S^{1}_{i,j}} &=& \probi{e,e'}{i \notin e,\ j \notin e',\ i \notin e',\ j \notin e} \\
&=& \probi{e}{i \notin e,\ j \notin e}\probi{e'}{j \notin e',\ i \notin e'} = \frac{\cb{n-2}{n-p}}{\cb{n}{n-p}}\times\frac{\cb{n-2}{n-p}}{\cb{n}{n-p}} \enspace\cdot \\
\probi{e,e'}{S^{2}_{i,j}} &=& \probi{e,e'}{i \notin e,\ j \notin e',\ i \notin e',\ j \in e} \\
&=& \probi{e}{i \notin e,\ j \in e} \probi{e'}{j \notin e',\ i \notin e'} = \frac{\cb{n-p-1}{n-2}}{\cb{n}{n-p}}\times\frac{\cb{n-p}{n-2}}{\cb{n}{n-p}}\enspace\cdot \\
\probi{e,e'}{S^{3}_{i,j}} &=& \probi{e,e'}{i \notin e,\ j \notin e',\ i \in e',\ j \notin e} \\
&=& \probi{e}{i \notin e,\ j \notin e} \probi{e'}{j \notin e',\ i \in e'} = \frac{\cb{n-p}{n-2}}{\cb{n}{n-p}}\frac{\cb{n-p-1}{n-2}}{\cb{n}{n-p}}\enspace\cdot \\
\probi{e,e'}{S^{4}_{i,j}} &=& \probi{e,e'}{i \notin e,\ j \notin e',\ i \in e',\ j \in e} \\
&=& \probi{e}{i \notin e,\ j \in e} \probi{e'}{j \notin e',\ i \in e'} = \frac{\cb{n-p-1}{n-2}}{\cb{n}{n-p}}\times\frac{\cb{n-p-1}{n-2}}{\cb{n}{n-p}}\enspace\cdot 
\end{eqnarray*}

\end{proof}

\medskip

\begin{lem}\label{lem.2.notin.e}
With the above notation, for $\ell \in \acc{1,3}$, it comes
{\small
\begin{align*} 
	\P_e\paren{\fe{X_1}\neq Y_1,\ \fep{X_2}\neq Y_2 \mid S^{\ell}_{1,2}} - \P_e\paren{\f{X_0}\neq Y_0,\fe{X_{1}}\neq Y_{1} \mid S^{\ell}_{1,2}} \leq \frac{4p\sqrt{k}}{\sqrt{2\pi}n} \enspace\cdot 
\end{align*}}

\end{lem}

\begin{proof}[Proof of Lemma~\ref{lem.2.notin.e}]
First remind that as a test sample element $Z_0$ cannot belong to either $e$ or $e'$. Consequently, an exhaustive formulation of
$$\P_e\paren{\f{X_0}\neq Y_0,\fe{X_{1}}\neq Y_{1} \mid S^{\ell}_{1,2}} = \P_e\paren{\f{X_0}\neq Y_0,\fe{X_{1}}\neq Y_{1} \mid S^{\ell}_{1,2} } \ .$$
Then it results
\begin{eqnarray*}
\P_e\paren{\f{X_0}\neq Y_0,\fe{X_{1}}\neq Y_{1} \mid S^{\ell}_{1,2}} 
&=&\P_e\paren{\widehat{f_k}^{(2)}(X_2)\neq Y_2,\ \fe{X_1}\neq Y_1 \mid S^{\ell}_{1,2}},
\end{eqnarray*}
where $\widehat{f_k}^{(2)}$ is built on sample $(X_0,Y_0),(X_1,Y_1),(X_3,Y_3),...,(X_n,Y_n)$.

Hence Lemma~\ref{Lemma : HOresult} implies
\begin{align*}
&\P_{e,e^\prime} \paren{\fe{X_1}\neq Y_1,\ \fep{X_2}\neq Y_2 \mid S^{\ell}_{1,2}} - \P_{e,e^\prime} \paren{\f{X_0}\neq Y_0,\fe{X_{1}}\neq Y_{1} \mid S^{\ell}_{1,2}} \\
&= \P_{e,e^\prime} \paren{\fe{X_1}\neq Y_1,\ \fep{X_2}\neq Y_2 \mid S^{\ell}_{1,2}} -\P_{e,e^\prime} \paren{\widehat{f_k}^{(2)}(X_2)\neq Y_2,\ \fe{X_1}\neq Y_1 \mid S^{\ell}_{1,2} }\\
& \leq \P_{e,e^\prime} \paren{\left\{\fe{X_1}\neq Y_1 \right\}\triangle \left\{\fe{X_1}\neq Y_1\right\} \mid S^{\ell}_{1,2} }+ \P_{e,e^\prime} \paren{\left\{\widehat{f_k}^{(2)}(X_2)\neq Y_2 \right\} \triangle \left\{ \fep{X_2}\neq Y_2\right\}\mid S^{\ell}_{1,2} }\\
&= \P_{e,e^\prime} \paren{\widehat{f_k}^{(2)}(X_2)\neq \fep{X_2}\mid S^{\ell}_{1,2}} \leq  \frac{4p\sqrt{k}}{\sqrt{2\pi} n} \enspace\cdot
\end{align*}

\end{proof}

\medskip

\begin{lem}\label{lem.2.in.e}
With the above notation, for $\ell \in \acc{2,4}$, it comes
\begin{align*}
& \P_{e,e^\prime} \paren{\fe{X_1}\neq Y_1,\ \fep{X_2}\neq Y_2 \mid S^{\ell}_{1,2}} - \P_{e,e^\prime} \paren{\f{X_0}\neq Y_0,\fe{X_{1}}\neq Y_{1} \mid S^{\ell}_{1,2}} \\
& \leq \frac{8\sqrt{k}}{\sqrt{2\pi} (n-p)} + \frac{4p\sqrt{k}}{\sqrt{2\pi} n} \enspace\cdot
\end{align*}

\end{lem}

\begin{proof}[Proof of Lemma~\ref{lem.2.in.e}]
As for the previous lemma, first notice that
\begin{eqnarray*}
\P_{e,e^\prime} \paren{\f{X_0}\neq Y_0,\fe{X_{1}}\neq Y_{1} \mid S^{\ell}_{1,2}} =\P_{e,e^\prime} \paren{\widehat{f_k}^{(2)}(X_2)\neq Y_2,\ \widehat{f_k}^{e_0}(X_1)\neq Y_1 \mid S^{\ell}_{1,2}} ,
\end{eqnarray*}
where $\widehat{f_k}^{e_0}$ is built on sample $e$ with observation $(X_2,Y_2)$ replaced with $(X_0,Y_0)$. Then
{\small
\begin{align*}
&\P_{e,e^\prime} \paren{\fe{X_1}\neq Y_1,\ \fep{X_2}\neq Y_2 \mid S^{\ell}_{1,2}}- \P_{e,e^\prime} \paren{\f{X_0}\neq Y_0,\fe{X_{1}}\neq Y_{1} \mid S^{\ell}_{1,2}} \\
&=\small \P_{e,e^\prime} \paren{\fe{X_1}\neq Y_1,\ \fep{X_2}\neq Y_2 \mid S^{\ell}_{1,2}}-\P_{e,e^\prime} \paren{\widehat{f_k}^{(2)}(X_2)\neq Y_2,\ \widehat{f_k}^{e_0}(X_1)\neq Y_1 \mid S^{\ell}_{1,2}}  \\
&\leq \P_{e,e^\prime} \paren{\left\{\fe{X_1}\neq Y_1\right\}\triangle \left\{\widehat{f_k}^{e_0}(X_1)\neq Y_1\right\} \mid S^{\ell}_{1,2}} +\P_{e,e^\prime} \paren{\left\{\widehat{f_k}^{(2)}(X_2)\neq Y_2\right\} \triangle\left\{ \fep{X_2}\neq Y_2\right\} \mid S^{\ell}_{1,2}}\\
&= \P_{e,e^\prime} \paren{\fe{X_1}\neq \widehat{f_k}^{e_0}(X_1)\mid S^{\ell}_{1,2}} + \P_{e,e^\prime} \paren{\widehat{f_k}^{(2)}(X_2)\neq  \fep{X_2} \mid S^{\ell}_{1,2}} \leq \frac{8\sqrt{k}}{\sqrt{2\pi} (n-p)} + \frac{4p\sqrt{k}}{\sqrt{2\pi} n} \enspace\cdot
\end{align*}}

\end{proof}

\subsection{Proof of Proposition~\ref{res.counter.example.bias}}
\label{sec.proof.couter.example}
%
%

	The bias of the L1O estimator is equal to
	\begin{align*}
	& \E\croch{ L\paren{\A_k^{\Dn}} - L\paren{\A_k^{\D_{n-1}}}  } \\
	& = -2 \E\croch{ \paren{\eta(X)-1/2} \paren{ \E\croch{ \E\croch{ \A_k^{\Dn}(X) - \A_k^{\D_{n-1}}(X) \mid X_{(k+1)}(X), X} \mid X } } } \\
	& = -2 \E\croch{ \paren{\eta(X)-1/2} \paren{ \E\croch{ \E\croch{ \A_k^{\Dn}(X) - \A_k^{\D_{n-1}}(X) \mid X_{(k+1)}(X), X} \mid X } } } \\
	& = 1/2 \left\lbrace  \E\croch{ \A_k^{\Dn}(0) - \A_k^{\D_{n-1}}(0) \mid X_{(k+1)}(0)=0, X=0}\P\croch{ X_{(k+1)}(0)=0\mid X=0} \right.\\
	&\left.\hspace*{2cm}+ \E\croch{ \A_k^{\Dn}(0) - \A_k^{\D_{n-1}}(0) \mid X_{(k+1)}(0)=1, X=0}\P\croch{ X_{(k+1)}(0)=1\mid X=0}  \right\rbrace \\
	& -   1/2 \left\lbrace  \E\croch{ \A_k^{\Dn}(1) - \A_k^{\D_{n-1}}(1) \mid X_{(k+1)}(1)=0, X=1}\P\croch{ X_{(k+1)}(1)=0\mid X=1} \right.\\
	& \left. \hspace*{2cm}+ \E\croch{ \A_k^{\Dn}(1) - \A_k^{\D_{n-1}}(1) \mid X_{(k+1)}(1)=1, X=1}\P\croch{ X_{(k+1)}(1)=1\mid X=1}  \right\rbrace ,
	\end{align*}
	where $X_{(k+1)}(x)$ denotes the $k+1$-th neighbor of $x$.
	
	Then, a few remarks lead to simplify the above expression.
	\begin{itemize}
		\item On the one hand it is easy to check that
		\begin{align*}
		&\E\croch{ \A_k^{\Dn}(0) - \A_k^{\D_{n-1}}(0) \mid X_{(k+1)}(0)=0, X=0} \\ & =	\E\croch{ \A_k^{\Dn}(1) - \A_k^{\D_{n-1}}(1) \mid X_{(k+1)}(1)=1, X=1}  = 0 ,
		\end{align*}
		since all of the $k+1$ nearest neighbors share the same label.
		
		\item On the other hand, let us notice	
		\begin{align*}
		&	\E\croch{ \A_k^{\Dn}(0) - \A_k^{\D_{n-1}}(0) \mid X_{(k+1)}(0)=1, X=0} \\
		& = \P\croch{ \A_k^{\Dn}(0)=1, \A_k^{\D_{n-1}}(0)=0 \mid   X_{(k+1)}(0)=1, X=0} \\
		& - \P\croch{ \A_k^{\Dn}(0)=0, \A_k^{\D_{n-1}}(0)=1 \mid   X_{(k+1)}(0)=1, X=0} .
		\end{align*}
		Then knowing $X_{(k+1)}(X)$ and $X$ are not equal implies the only way for $\A_k^{\Dn}$ and $\A_k^{\D_{n-1}}$ to differ is that the numbers of $k$ nearest neighbors of each label are almost equal, that is either equal to $(k-1)/2$ or to $(k+1)/2$ ($k$ is odd by assumption).
		
		With $N_0^1$ (respectively $\tilde N_0^1$) denoting the number of 1s among th $k$ nearest neighbors of $X=0$ among $X_1,\ldots,X_n$ (resp. $X_1,\ldots,X_{n-1}$), the proof of Theorem~3 in \cite{Chaudhuri_Dasgupta2014} leads to
		\begin{align*}
		& \P\croch{ \A_k^{\Dn}(0)=1, \A_k^{\D_{n-1}}(0)=0 \mid   X_{(k+1)}(0)=1, X=0} \\
		& = \P\croch{  n \in V_k(0), N_0^1 = (k+1)/2, \tilde N_0^1 = (k-1)/2 \mid   X_{(k+1)}(0)=1, X=0} \\
		& = \frac{k}{n}  \times \P\croch{ \tilde N_0^1 = (k-1)/2 \mid  N_0^1 = (k+1)/2,  X_{(k+1)}(0)=1, X=0} \\
		& \times \P\croch{ N_0^1 = (k+1)/2  \mid   X_{(k+1)}(0)=1, X=0} \\
		& = \frac{k}{n} \times \P\croch{\mathcal{H}\paren{ \frac{k+1}{2},\frac{k-1}{2};1}=1 } \cdot \eta_1 \times {k \choose (k+1)/2} \bar\eta^{(k+1)/2}\paren{1- \bar{\eta}}^{(k-1)/2} \\
		& = \frac{k+1}{2n} \times  \eta_1 \times {k \choose (k+1)/2} \bar\eta^{(k+1)/2}\paren{1- \bar{\eta}}^{(k-1)/2} ,
		\end{align*}
		where $\mathcal{H}(a,b;c)$ denotes a hypergeometric random variable with $a$ successes in a population of cardinality $a+b$, and $c$ draws, and $\bar \eta = \pi_0 \eta_0 + (1-\pi_0) \eta_1 = 1/2$.

		Following the same reasoning for $\P\croch{ \A_k^{\Dn}(0)=0, \A_k^{\D_{n-1}}(0)=1 \mid   X_{(k+1)}(0)=1, X=0} $ and recalling that $\eta_0=0$ and $\eta_1=1$ by assumption, it results
		\begin{align*}
		\E\croch{ \A_k^{\Dn}(0) - \A_k^{\D_{n-1}}(0) \mid X_{(k+1)}(0)=1, X=0}  = - \frac{k+1}{2n}  \times {k \choose (k+1)/2} \paren{1/2}^{k} .
		\end{align*}
		
		\item Similar calculations applied to $X=1$ finally lead to
		\begin{align*}
		\E\croch{ L\paren{\A_k^{\Dn}} - L\paren{\A_k^{\D_{n-1}}}  }
		& = \frac{k+1}{2n}  \times {k \choose (k+1)/2} \paren{1/2}^{k} \times \P\croch{ X_{(k+1)}(0)=1 \mid X=0 } \\
		& = \frac{k+1}{2n}  \times {k \choose (k+1)/2} \paren{1/2}^{k} \times \P\croch{ \mathcal{B}(n,1/2)\leq k } .
		\end{align*}
		
		\item The conclusion then follows from considering $k\geq n/2$ which entails that $\P\croch{ \mathcal{B}(n,1/2)\leq k } \geq 1/2 $, and also by noticing that
		\begin{align*}
		\frac{k+1}{2n}  \times {k \choose (k+1)/2} \paren{1/2}^{k} \geq C \frac{\sqrt{k}}{n} ,
		\end{align*}
		where denotes a numeric constant independent of $n$ and $k$.
		\end{itemize}
\newpage
\section{Technical results}

\subsection{Main inequalities}

%
%
%
%
%

\subsubsection{From moment to exponential inequalities}

\begin{prop}[see also \cite{Arl:2007:phd}, Lemma~8.10]
\label{prop.moment.exponential}
  Let $X$ denote a real valued random variable, and assume there exist $C\geq 1$, $\lambda_1,\ldots,\lambda_N>0$, and $\alpha_1,\ldots,\alpha_N>0$ ($N\in\N^*$) such that for every $q\geq q_0$,
\begin{align*}
  \E\croch{ \abs{X}^q } \leq C \paren{ \sum_{i=1}^N \lambda_i q^{\alpha_i} }^q . \hfill
\end{align*}
Then for every $t>0$,
\begin{eqnarray}\label{ineq.moment.exponential.deviation}
  \P\croch{  \abs{X} >t }
& \leq C e^{q_0 \min_j \alpha_j} e^{ - (\min_i \alpha_i )  e^{-1} \min_j\acc{ \paren{\frac{t}{N \lambda_j}}^{\frac{1}{\alpha_j}} } } ,
\end{eqnarray}
%
Furthermore for every $x>0$, it results
\begin{eqnarray}\label{ineq.moment.exponential.error}
  \P\croch{ \abs{X} > \sum_{i=1}^N \lambda_i \paren{ \frac{e x }{\min_j \alpha_j} }^{\alpha_i} } \leq C e^{q_0 \min_j \alpha_j}  \cdot e^{-x} . \hfill
\end{eqnarray}

\end{prop}

\medskip

\begin{proof}[Proof of Proposition~\ref{prop.moment.exponential}]

By use of Markov's inequality applied to $\abs{X}^q$ ($q>0$), it comes for every $t>0$
\begin{align*}
    \P\croch{  \abs{X} >t } \leq  \1_{q \geq q_0} \frac{ \E\croch{ \abs{X}^q }}{t^q} +  \1_{q < q_0} \leq \1_{q \geq q_0} C \paren{ \frac{ \sum_{i=1}^N \lambda_i q^{\alpha_i} }{t} }^q +  \1_{q < q_0} .
\end{align*}
Now using the upper bound $\sum_{i=1}^N \lambda_i q^{\alpha_i} \leq N \max_i \acc{\lambda_i q^{\alpha_i}}$ and choosing the particular value $\tilde q = \tilde q(t)= e^{-1} \min_{j} \acc{ \paren{\frac{t}{N \lambda_j}}^{\frac{1}{\alpha_j}} }$, one gets
\begin{align*}
    \P\croch{  \abs{X} >t } & \leq \1_{ \tilde q \geq q_0} C \paren{ \frac{ \max_i \acc{N\lambda_i \paren{e^{-\alpha_i} \min_{j} \acc{ \paren{\frac{t}{N \lambda_j}}^{\frac{1}{\alpha_j}} }}^{\alpha_i}} }{t} }^{\tilde q} +  \1_{ \tilde q < q_0} \\
&  \leq \1_{\tilde q \geq q_0} C e^{ -(\min_i\alpha_i)   \croch{ e^{-1} \min_{j} \acc{ \paren{\frac{t}{N \lambda_j}}^{\frac{1}{\alpha_j}} } } } +  \1_{\tilde q < q_0} ,
\end{align*}
which provides \eqref{ineq.moment.exponential.deviation}.

Let us now turn to the proof of \eqref{ineq.moment.exponential.error}.
From $t^* = \sum_{i=1}^N\lambda_i \paren{ \frac{e x }{\min_j \alpha_j}}^{\alpha_i}$ combined with $q^* = \frac{x}{\min_j \alpha_j}$, it arises for every $x>0$
\begin{align*}
\frac{ \sum_{i=1}^N \lambda_i (q^*)^{\alpha_i} }{t^*} = \frac{ \sum_{i=1}^N \lambda_i \paren{ e^{-1}\frac{ex}{\min_j \alpha_j} }^{\alpha_i} }{ \sum_{i=1}^N\lambda_i \paren{ \frac{e x }{\min_j \alpha_j}}^{\alpha_i} } \leq \paren{ \max_k e^{-\alpha_k} }  \frac{ \sum_{i=1}^N \lambda_i \paren{\frac{ex}{\min_j \alpha_j} }^{\alpha_i} }{ \sum_{i=1}^N\lambda_i \paren{ \frac{e x }{\min_j \alpha_j}}^{\alpha_i} } =  e^{-\min_k \alpha_k} .
\end{align*}
Then,
\begin{align*}
C \paren{ \frac{ \sum_{i=1}^N \lambda_i (q^*)^{\alpha_i} }{t^*} }^{q^*} \leq C e^{- \paren{\min_k \alpha_k} \frac{x}{\min_j \alpha_j} } = C e^{-x} .
\end{align*}
Hence,
\begin{align*}
  \P\croch{ \abs{X} > \sum_{i=1}^N \lambda_i \paren{ \frac{e x }{\min_j \alpha_j} }^{\alpha_i} } & \leq C e^{-x} \1_{q^*\geq q_0} + \1_{ q^*< q_0 }  \leq C e^{q_0 \min_j \alpha_j}  \cdot e^{-x} ,
\end{align*}
since $ e^{q_0 \min_j \alpha_j} \geq 1 $ and $ -x + q_0 \min_j \alpha_j \geq 0 $ if $q<q_0$.

\end{proof}

\subsubsection{Sub-Gaussian random variables}

\begin{lem}[Theorem~2.1 in \cite{BouLugMas_2013} first part] \label{lem.subgaussian.exp.to.moment}
Any centered random variable $X$ such that $\prob{ X > t } \vee \prob{ -X > t } \leq e^{- t^2/(2\nu)}$ satisfies
\begin{align*}
  \esp{ X^{2q} } \leq q! \paren{ 4 \nu }^q.
\end{align*}
for all  $q$ in $\mathbb{N}_+$.
\end{lem}

\begin{lem}[Theorem~2.1 in \cite{BouLugMas_2013} second part] \label{lem.subgaussian.moment.to.exp}
Any centered random variable $X$ such that
\begin{align*}
  \esp{ X^{2q} } \leq q! C^q.
\end{align*}
for some $C>0$ and $q$ in $\mathbb{N}_+$ satisfies
 $\prob{ X > t } \vee \prob{ -X > t } \leq e^{- t^2/(2\nu)}$ with $\nu = 4C$.
\end{lem}

\subsubsection{The Efron-Stein inequality}

\begin{thm}[Efron-Stein's inequality \cite{BouLugMas_2013}, Theorem~3.1]
\label{thm.efron.stein}
Let $X_1,\ldots,X_n$ be independent random variables and let $Z = f\paren{X_1,\ldots,X_n}$ be a square-integrable function.
Then
\begin{align*}
  \Var(Z) \leq \sum_{i=1}^n \esp{ \paren{Z - \esp{ Z \mid (X_j)_{j\neq i}} }^2 } = \nu.
\end{align*}
Moreover if $X_1^{\prime},\ldots,X_n^{\prime}$ denote independent copies of $X_1,\ldots,X_n$ and if we define for every $1\leq i\leq n$
\begin{align*}
  Z_i^{\prime} = f\paren{X_1,\ldots,X_i^{\prime},\ldots,X_n},
\end{align*}
then
\begin{align*}
  \nu = \frac{1}{2} \sum_{i=1}^n \esp{ \paren{Z - Z_i^{\prime}}^2 }.
\end{align*}
\end{thm}

\subsubsection{Generalized Efron-Stein's inequality}\label{subsec.generalized.efron.stein}


\begin{thm}[Theorem~15.5 in \cite{BouLugMas_2013}]
Let $X_1,\ldots,X_n$ $n$ independent random variables, $f : \R^n\rightarrow\R$ a measurable function, and define $\zeta=f(X_1,\ldots,X_n)$ and $\zeta_i'=f(X_1,\ldots,X_i',\ldots,X_n)$, with $X'_1,\ldots,X'_n$  independent copies of $X_i$.
Furthermore let $V_+~=~\esp{  \sum_{i} ^n \gc{\paren{\zeta-\zeta_i'}_+}^2 \mid X_1^n}$ and
 $\zeta_-~=~\esp{  \sum_{i} ^n \gc{\paren{Z-Z_i'}_-}^2 \mid X_1^n}$.
Then there exists a constant $\kappa\leq 1,271$ such that for all $q$ in $[2,+\infty[$,
\begin{align*}
	\norm{\paren{\zeta-\E \zeta}_+}_q \leq \sqrt{2\kappa q\norm{V_+}_{q/2}}\enspace,\qquad \mbox{and}\qquad 	\norm{\paren{\zeta-\E \zeta}_-}_q \leq \sqrt{2\kappa q\norm{V_-}_{q/2}}\enspace.
\end{align*}
\end{thm}

\begin{cor}\label{cor.generalized.Efron.Stein.appendix}
With the same notation, it comes
\begin{align}
	\norm{ \zeta-\E\zeta }_q & \leq \sqrt{2\kappa q} \sqrt{\norm{\sum_{i=1}^n \paren{\zeta-\zeta_i^{\prime}}^2}_{q/2}}  \leq \sqrt{4\kappa q } \sqrt{\norm{\sum_{i=1}^n \paren{\zeta-\esp{\zeta \mid (X_j)_{j\neq i}}}^2}_{q/2}} \label{ineg.moment.integrated.obs} \enspace.
\end{align}
Moreover considering $ \zeta^{(j)} = f( X_1, \ldots, X_{j-1}, X_{j+1}, \ldots, X_n) $ for every $1\leq j\leq n$, it results
\begin{eqnarray}
\norm{ \zeta-\E \zeta }_q &  \leq 2\sqrt{2\kappa q } \sqrt{ \norm{\sum_{i=1}^n \paren{\zeta-\zeta^{(j)} }^2}_{q/2}} \label{ineg.moment.removed.obs} \enspace.
\end{eqnarray}
\end{cor}

%

\subsubsection{McDiarmid's inequality}
\label{subsec.mc.diarmid}

\begin{thm}\label{thm.bounded.differences.ineq}
Let $X_1,...,X_n$ be independent random variables taking values in a set $A$, and assume that $f: A^n \rightarrow \mathbb{R}$ satisfies
\begin{eqnarray*}
\underset{x_1,...,x_n,x_i'}{\sup}\left|f(x_1,...,x_i,...,x_n) - f(x_1,...,x_i',...,x_n) \right| \leq c_i, \ 1\leq i \leq n \ .
\end{eqnarray*}
Then for all $\varepsilon>0$, one has
\begin{eqnarray*}
\prob{f(X_1,...,X_n) - E\gc{f(X_1,...,X_n)} \geq \varepsilon} &\leq& e^{-2\varepsilon^2/\sum_{i=1}^{n}c_i^2} \\
\prob{E\gc{f(X_1,...,X_n)} - f(X_1,...,X_n) \geq \varepsilon} &\leq& e^{-2\varepsilon^2/\sum_{i=1}^{n}c_i^2}
\end{eqnarray*}
A proof can be found in \cite{DeGyLu_1996} (see Theorem~9.2).
\end{thm}

\subsubsection{Rosenthal's inequality}

\begin{prop}[Eq.~(20) in \cite{IbragShar2002}]
\label{prop.rosenthal.inequality}
Let $X_1,\ldots,X_n$ denote independent real random variables with symmetric distributions.
Then for every $q>2$ and $\gamma>0$,
\begin{align*}
  E\croch{ \abs{\sum_{i=1}^n X_i}^q } \leq B(q,\gamma) \acc{ \gamma \sum_{i=1}^n E\croch{\abs{X_i}^q} \vee \paren{ \sqrt{ \sum_{i=1}^n E\croch{X_i^2} } }^q  } , \hfill
\end{align*}
where $a\vee b = \max(a,b)$ ($a,b\in \R$), and $B(q,\gamma)$ denotes a positive constant only depending on $q$ and $\gamma$.
Furthermore, the optimal value of $B(q,\gamma)$ is given by
$$\begin{array}{rcll}
  B^*(q,\gamma)  & = &\ 1+ \frac{E\croch{\abs{ N}^q} }{\gamma}&,\ \mbox{if}\quad 2< q \leq 4, \\
& = &\ \gamma^{-q/(q-1)} E\croch{ \abs{ Z - Z^\prime }^q } &,\ \mbox{if}\quad 4< q,
\end{array}
$$
where $N$ denotes a standard Gaussian variable, and $Z,Z^\prime$ are \iid random variables with Poisson distribution $\mathcal{P}\paren{ \frac{\gamma^{1/(q-1)}}{2} }$.

\end{prop}

\medskip

\begin{prop}
\label{prop.rosenthal.inequality.upper.bound}
Let $X_1,\ldots,X_n$ denote independent real random variables with symmetric distributions.
Then for every $q>2$,
\begin{align*}
    E\croch{ \abs{\sum_{i=1}^n X_i}^q } \leq \paren{ 2\sqrt{2e} }^q  \max\acc{ q^{q} \sum_{i=1}^n E\croch{\abs{X_i}^q} , \paren{\sqrt{q} }^q  \paren{ \sqrt{ \sum_{i=1}^n E\croch{X_i^2} } }^q  } .
\end{align*}

\end{prop}

\begin{proof}[Proof of Proposition~\ref{prop.rosenthal.inequality.upper.bound}]

From Lemma~\ref{lem.rosenthal.ineq.optimal.constant.upper.bound}, let us observe
\begin{itemize}
  \item if $2<q\leq 4$, choosing $\gamma=1$ provides
\begin{align*}
  B^*(q,\gamma) & \leq \paren{2 \sqrt{2e} \sqrt{q} }^{q} .
\end{align*}

  \item if $4<q$, $\gamma= q^{(q-1)/2}$ leads to
\begin{align*}
  B^*(q,\gamma) & \leq  q^{-q/2}  \paren{ \sqrt{ 4e q \paren{ q^{1/2} + q} } }^q \leq  q^{-q/2}  \paren{ \sqrt{8e} q }^q = \paren{ 2\sqrt{2e} \sqrt{q} }^q .
\end{align*}
\end{itemize}
Plugging the previous upper bounds in Rosenthal's inequality (Proposition~\ref{prop.rosenthal.inequality}), it results for every $q>2$
\begin{align*}
  E\croch{ \abs{\sum_{i=1}^n X_i}^q } \leq \paren{ 2\sqrt{2e} \sqrt{q} }^q \max\acc{ \paren{\sqrt{q} }^{q} \sum_{i=1}^n E\croch{\abs{X_i}^q} , \paren{ \sqrt{ \sum_{i=1}^n E\croch{X_i^2} } }^q  } .
\end{align*}

\end{proof}

\medskip

\begin{lem}\label{lem.rosenthal.ineq.optimal.constant.upper.bound}
With the same notation as Proposition~\ref{prop.rosenthal.inequality} and for every $\gamma>0$, it comes
\begin{itemize}
  \item for every $2<q\leq 4$,
\begin{align*}
  B^*(q,\gamma) & \leq 1+ \frac{  \paren{ \sqrt{2e} \sqrt{q} }^{q} }{\gamma} ,
\end{align*}

  \item for every $4<q$,
\begin{align*}
  B^*(q,\gamma) & \leq  \gamma^{-q/(q-1)}  \paren{ \sqrt{ 4e q \paren{ \gamma^{1/(q-1)} + q} } }^q .
\end{align*}

\end{itemize}

\end{lem}

\begin{proof}[Proof of Lemma~\ref{lem.rosenthal.ineq.optimal.constant.upper.bound}]

If $2<q\leq 4$,
\begin{align*}
  B^*(q,\gamma)  & =  1+ \frac{E\croch{\abs{ N}^q} }{\gamma} \leq 1+ \frac{ \sqrt{2} e \sqrt{q} \paren{ \frac{q}{e} }^{\frac{q}{2}} }{\gamma} \leq 1+ \frac{ \sqrt{2e}^q \sqrt{e}^q \paren{ \frac{q}{e} }^{\frac{q}{2}} }{\gamma}  = 1+ \frac{  \paren{ \sqrt{2e} \sqrt{q} }^{q} }{\gamma} \enspace ,
\end{align*}
by use of Lemma~\ref{lem.moment.gaussian} and $\sqrt{q}^{1/q} \leq \sqrt{e}$ for every $q>2$.

If $q>4$,
\begin{align*}
  B^*(q,\gamma)  & = \gamma^{-q/(q-1)} E\croch{ \abs{ Z - Z^\prime }^q } \\
& \leq \gamma^{-q/(q-1)} 2^{q/2+1} e \sqrt{q} \croch{\frac{q}{e} \paren{ \gamma^{1/(q-1)} + q} }^{q/2} \\
& \leq \gamma^{-q/(q-1)} 2^{q/2} \sqrt{2e}^q \sqrt{e}^q \croch{\frac{q}{e} \paren{ \gamma^{1/(q-1)} + q} }^{q/2} \\
& \leq \gamma^{-q/(q-1)}  \croch{ 4e q \paren{ \gamma^{1/(q-1)} + q} }^{q/2} =  \gamma^{-q/(q-1)}  \paren{ \sqrt{ 4e q \paren{ \gamma^{1/(q-1)} + q} } }^q ,
\end{align*}
applying Lemma~\ref{lem.diff.poisson} with $\lambda = 1/2 \gamma^{1/(q-1)}$.

\end{proof}

\subsection{Technical lemmas}

%

\subsubsection{Basic computations for resampling applied to the $k$NN algorithm}

\begin{lem}\label{BasicResultForResamplingKnn}
  For every $1\leq i \leq n$ and $1\leq p\leq n$, one has
\begin{align} \label{eq.sum.all.points}
\probi{e}{i\in \bar{e}}  & = \frac{p}{n}  ,\\
 \sum_{j=1}^n \P_e\croch{ i\in \eb,\ j\in V_k^e(X_i) } & = \frac{kp}{n} \enspace , \\
  \sum_{ k< \sigma_i(j)\leq k+p} \P_e\croch{ i\in \eb,\ j\in V_k^e(X_i) } & = \frac{kp}{n} \frac{p-1}{n-1} \enspace\cdot \label{eq.sum.faraway.neighbors.points}
\end{align}

\end{lem}

\begin{proof}[Proof of Lemma~\ref{BasicResultForResamplingKnn}]
The first equality is straightforward.
The second one results from simple calculations as follows.
\begin{align*}
    \sum_{j=1}^n \P_e\croch{ i\in \eb,\ j\in V_k^e(X_i) } & = \sum_{j=1}^n  {n\choose p}^{-1} \sum_{e} \1_{i\in\eb} \1_{j\in V_k^e(X_i)}  =  {n\choose p}^{-1} \sum_{e} \1_{i\in\eb}  \paren{  \sum_{j=1}^n \1_{j\in V_k^e(X_i)} }\\
& = \paren{ {n\choose p}^{-1} \sum_{e} \1_{i\in\eb} } k  = \frac{p}{n}\, k \enspace.
\end{align*}

For the last equality, let us notice every $j \in V_i$ satisfies
\begin{align*}
  \P_e\croch{ i\in \eb,\ j\in V_k^e(X_i) } = \P_e\croch{ j\in V_k^e(X_i) \mid i\in \eb}  \P_e\croch{ i\in \eb} =  \frac{n-1}{n-p} \frac{p}{n} \enspace,
\end{align*}
hence
\begin{align*}
  \sum_{ k< \sigma_i(j)\leq k+p} \P_e\croch{ i\in \eb,\ j\in V_k^e(X_i) } & = \sum_{ j=1}^n \P_e\croch{ i\in \eb,\ j\in V_k^e(X_i) } - \sum_{ \sigma_i(j)\leq k }\P_e\croch{ i\in \eb,\ j\in V_k^e(X_i) } \\
& = k \frac{p}{n} - k \frac{n-1}{n-p} \frac{p}{n} = k \frac{p}{n} \frac{p-1}{n-1} \enspace\cdot
\end{align*}
\end{proof}

\subsubsection{Stone's lemma}
\begin{lem}[\cite{DeGyLu_1996}, Corollary~11.1, p.~171]
\label{Stone}
Given $n$ points $(x_1,...,x_n)$ in $\mathbb{R}^d$, any of these points belongs to the $k$ nearest neighbors of at most $k\gamma_d$ of the other points, where $\gamma_d$ increases on $d$.
\end{lem}

\subsubsection{Stability of the $k$NN classifier when removing $p$ observations}
\begin{lem}[\cite{DeWa79}, Eq.~ (14)]
	\label{Lemma : HOresult}
For every $1\leq k \leq n$, let $\A_k$ denote $k$-NN classification algorithm defined by Eq.~\eqref{def.knn.classifier}, and let $Z_1,\ldots,Z_n$ denote $n$ \iid random variables such that for every $1\leq i\leq n$, $Z_i = (X_i,Y_i) \sim P$.
Then for every $1\leq p \leq n-k$, 
\begin{eqnarray*}
\P\croch{ \A_k(Z_{1,n};X) \neq \A_k(Z_{1,n-p};X) } \leq \frac{4}{\sqrt{2\pi}}\frac{ p \sqrt{k}}{ n} \enspace ,
\end{eqnarray*}
where $Z_{1,i} = \paren{Z_1,\ldots,Z_i}$ for every $1\leq i\leq n$, and $(X,Y) \sim P$ is independent of $Z_{1,n}$.
\end{lem}

\subsubsection{Exponential concentration inequality for the L1O estimator}
\begin{lem}[\cite{DeGyLu_1996}, Theorem~24.4]
	\label{Lemma : UpperBoundL1OEstimator}
For every $1\leq k \leq n$, let $\A_k$ denote $k$-NN classification algorithm defined by Eq.~\eqref{def.knn.classifier}. Let also $\Rh_1(\cdot)$ denote the L1O estimator defined by Eq.~\eqref{def.Lpo.estimator} with $p=1$.
Then for every $\varepsilon>0$,
\begin{eqnarray*}
\prob{\left| \widehat{R}_1(\A_k,Z_{1,n}) - \E\croch{ \widehat{R}_1(\A_k,Z_{1,n})  } \right|> \varepsilon} \leq 2\exp\ga{-n \frac{ \varepsilon^2}{\gamma_d^2 k^2} } .
\end{eqnarray*}
\end{lem}

\subsubsection{Moment upper bounds for the L1O estimator}

\begin{lem}\label{Lemma : MomentUpperBoundL1OEstimator}
	For every $1\leq k \leq n$, let $\A_k$ denote $k$-NN classification algorithm defined by Eq.~\eqref{def.knn.classifier}. Let also $\Rh_1(\cdot)$ denote the L1O estimator defined by Eq.~\eqref{def.Lpo.estimator} with $p=1$.
	Then for every $q\geq 1$, 
\begin{eqnarray}\label{ineq.moment.pair_L1o}
  \esp{ \abs{ \Rh_1\paren{ \A_k,Z_{1,n}} - \E\croch{\Rh_1\paren{ \A_k,Z_{1,n}}}   }^{2q} } \leq q! \paren{ 2 \frac{\paren{k\gamma_d}^2}{n} }^q .
\end{eqnarray}
\end{lem}
The proof is straightforward from the combination of Lemmas~\ref{lem.subgaussian.exp.to.moment} and~\ref{Lemma : UpperBoundL1OEstimator}.

\subsubsection{Upper bound on the optimal constant in the Rosenthal's inequality}

\begin{lem}\label{lem.moment.gaussian}
  Let $N$ denote a real-valued standard Gaussian random variable.
Then for every $q>2$, one has
\begin{align*}
  \E\croch{ \abs{N}^q } & \leq \sqrt{2} e \sqrt{q} \paren{ \frac{q}{e} }^{\frac{q}{2}} .
\end{align*}

\end{lem}

\begin{proof}[Proof of Lemma~\ref{lem.moment.gaussian}]

If $q$ is even ($q=2k>2$), then
\begin{align*}
  \E\croch{ \abs{N}^q } & = 2\int_{0}^{+\infty} x^q  \frac{1}{\sqrt{2\pi}} e^{-\frac{x^2}{2}} dx = \sqrt{\frac{2}{\pi}} (q-1) \int_{0}^{+\infty} x^{q-2}  e^{-\frac{x^2}{2}} dx \\
& = \sqrt{\frac{2}{\pi}} \frac{(q-1)!}{2^{k-1} (k-1)!} = \sqrt{\frac{2}{\pi}} \frac{q!}{2^{q/2} (q/2)!} \enspace\cdot
\end{align*}
Then using for any positive integer $a$
\begin{align*}
 \sqrt{2\pi a} \paren{ \frac{a}{e} }^a <  a ! < \sqrt{2e\pi a} \paren{ \frac{a}{e} }^a ,
\end{align*}
it results
\begin{align*}
  \frac{q!}{2^{q/2} (q/2)!} < \sqrt{2e}\, e^{-q/2} q^{q/2} ,
\end{align*}
which implies
\begin{align*}
  \E\croch{ \abs{N}^q } & \leq 2 \sqrt{\frac{e}{\pi}} \paren{ \frac{q}{e} }^{q/2} < \sqrt{2} e \sqrt{q} \paren{ \frac{q}{e} }^{\frac{q}{2}} \enspace \cdot
\end{align*}

If $q$ is odd ($q=2k+1>2$), then
\begin{align*}
  \E\croch{ \abs{N}^q } & = \sqrt{\frac{2}{\pi}} \int_{0}^{+\infty} x^q   e^{-\frac{x^2}{2}} dx = \sqrt{\frac{2}{\pi}} \int_{0}^{+\infty} \sqrt{2t}^q   e^{-t} \frac{dt}{\sqrt{2t}} ,
\end{align*}
by setting $x=\sqrt{2t}$.
In particular, this implies
\begin{align*}
  \E\croch{ \abs{N}^q } & \leq \sqrt{\frac{2}{\pi}} \int_{0}^{+\infty} \paren{2t}^k   e^{-t}  dt =
\sqrt{\frac{2}{\pi}} 2^k k! =  \sqrt{\frac{2}{\pi}} 2^{\frac{q-1}{2}} \paren{\frac{q-1}{2}}!  < \sqrt{2} e \sqrt{q} \paren{\frac{q}{e}}^{\frac{q}{2}} .
\end{align*}

\end{proof}

\medskip

\begin{lem}
\label{lem.binom.moment.upper.bound}
Let $S$ denote a binomial random variable such that $S \sim \mathcal{B}(k,1/2)$ ($k\in\N^*$).
Then for every $q>3$, it comes
\begin{align*}
  \E\croch{ \abs{ S-\E\croch{S} }^q } & \leq 4\sqrt{ e } \sqrt{q } \sqrt{\frac{qk}{2e}}^{q} \enspace\cdot
\end{align*}

\end{lem}

\begin{proof}[Proof of Lemma~\ref{lem.binom.moment.upper.bound}]

Since $S-\E(S)$ is symmetric, it comes
\begin{align*}
  \E\croch{ \abs{ S-\E\croch{S} }^q } & = 2 \int_{0}^{+\infty} \P\croch{  S < \E\croch{S} - t^{1/q}  } \, dt =  2 q \int_{0}^{+\infty} \P\croch{  S < \E\croch{S} - u  } u^{q-1}\, du .
\end{align*}
Using Chernoff's inequality and setting $u = \sqrt{k/2} v$, it results
\begin{align*}
  \E\croch{ \abs{ S-\E\croch{S} }^q } & \leq  2 q \int_{0}^{+\infty}  u^{q-1} e^{-\frac{u^2}{k}} du = 2 q \sqrt{\frac{k}{2}}^{q} \int_{0}^{+\infty}  v^{q-1} e^{-\frac{v^2}{2}} dv .
\end{align*}

If $q$ is even, then $q-1>2$ is odd and the same calculations as in the proof of Lemma~\ref{lem.moment.gaussian} apply, which leads to
\begin{align*}
  \E\croch{ \abs{ S-\E\croch{S} }^q } & \leq  2 \sqrt{\frac{k}{2}}^{q} 2^{q/2} \paren{\frac{q}{2}}! \leq 2 \sqrt{\frac{k}{2}}^{q} 2^{q/2} \sqrt{\pi e q } \paren{\frac{q}{2e}}^{q/2} =    2\sqrt{\pi e } \sqrt{q } \sqrt{\frac{qk}{2e}}^{q}  <  4\sqrt{ e } \sqrt{q } \sqrt{\frac{qk}{2e}}^{q} \enspace\cdot
\end{align*}

If $q$ is odd, then $q-1>2$ is even and another use of the calculations in the proof of Lemma~\ref{lem.moment.gaussian}
provides
\begin{align*}
\E\croch{ \abs{ S-\E\croch{S} }^q } & \leq  2 q \sqrt{\frac{k}{2}}^{q}  \frac{(q-1)!}{ 2^{(q-1)/2}\frac{q-1}{2}!} =  2  \sqrt{\frac{k}{2}}^{q}  \frac{q!}{ 2^{(q-1)/2}\frac{q-1}{2}!} .
\end{align*}
Let us notice
\begin{align*}
  \frac{q!}{ 2^{(q-1)/2}\frac{q-1}{2}!} & \leq \frac{ \sqrt{2\pi e q} \paren{\frac{q}{e}}^{q} }{ 2^{(q-1)/2} \sqrt{\pi (q-1)} \paren{\frac{q-1}{2e}}^{(q-1)/2}  } = \sqrt{2e}\sqrt{\frac{q}{q-1}} \frac{ \paren{\frac{q}{e}}^{q} }{ \paren{\frac{q-1}{e}}^{(q-1)/2}  } \\
& = \sqrt{2e} \sqrt{\frac{q}{q-1}} \paren{\frac{q}{e}}^{(q+1)/2}  \paren{ \frac{ q }{q-1} }^{(q-1)/2}
\end{align*}
and also that
\begin{align*}
\sqrt{\frac{q}{q-1}}   \paren{ \frac{ q }{q-1} }^{(q-1)/2} \leq \sqrt{2e} .
\end{align*}
This implies
\begin{align*}
  \frac{q!}{ 2^{(q-1)/2}\frac{q-1}{2}!} & \leq
2e \paren{\frac{q}{e}}^{(q+1)/2} = 2\sqrt{e} \sqrt{q} \paren{\frac{q}{e}}^{q/2} ,
\end{align*}
hence
\begin{align*}
\E\croch{ \abs{ S-\E\croch{S} }^q } & \leq  2  \sqrt{\frac{k}{2}}^{q}   2\sqrt{e} \sqrt{q} \paren{\frac{q}{e}}^{q/2} =
4 \sqrt{e} \sqrt{q} \sqrt{\frac{qk}{2e}}^{q}   \enspace \cdot
\end{align*}

\end{proof}

\medskip

\begin{lem}
  \label{lem.diff.poisson}
Let $X,Y$ be two \iid random variables with Poisson distribution $\mathcal{P}(\lambda)$ ($\lambda>0$).
Then for every $q>3$, it comes
\begin{align*}
  \E\croch{ \abs{ X - Y}^q } & \leq  2^{q/2+1} e \sqrt{q} \croch{\frac{q}{e} \paren{2\lambda + q} }^{q/2} .
\end{align*}

\end{lem}

\begin{proof}[Proof of Lemma~\ref{lem.diff.poisson}]
Let us first remark that
\begin{align*}
  \E\croch{ \abs{ X - Y}^q } & = \E_N\croch{ \E\croch{ \abs{ X - Y}^q \mid N} }  = 2^q \E_N\croch{ \E\croch{ \abs{ X - N/2 }^q \mid N} } ,
\end{align*}
 where $N=X+Y$.
Furthermore, the conditional distribution of $X$ given $N=X+Y$ is a binomial distribution $\mathcal{B}(N,1/2)$.
Then Lemma~\ref{lem.binom.moment.upper.bound} provides that
\begin{align*}
  \E\croch{ \abs{ X - N/2 }^q \mid N} & \leq 4\sqrt{ e } \sqrt{q } \sqrt{\frac{qN}{2e}}^{q} \qquad a.s.\,,
\end{align*}
which entails that
\begin{align*}
  \E\croch{ \abs{ X - Y}^q } & \leq 2^q \E_N\croch{4\sqrt{ e } \sqrt{q } \sqrt{\frac{qN}{2e}}^{q}} = 2^{q/2+2} \sqrt{ e } \sqrt{q} \sqrt{\frac{q}{e}}^{q} \E_N\croch{ N^{q/2} } .
\end{align*}
It only remains to upper bound the last expectation where $N$ is a Poisson random variable $\mathcal{P}(2\lambda)$ (since $X,Y$ are \iid):
\begin{align*}
  \E_N\croch{ N^{q/2} } & \leq \sqrt{  \E_N\croch{ N^{q} } }
\end{align*}
by Jensen's inequality.
Further introducing Touchard polynomials and using a classical upper bound, it comes
\begin{align*}
  \E_N\croch{ N^{q/2} } & \leq \sqrt{ \sum_{i=1}^q (2\lambda)^i  \frac{1}{2} {q\choose i } i^{q-i} } \leq \sqrt{ \sum_{i=0}^q (2\lambda)^i  \frac{1}{2} {q\choose i } q^{q-i} } \\
& = \sqrt{ \frac{1}{2} \sum_{i=0}^q {q\choose i } (2\lambda)^i    q^{q-i} } =  \sqrt{ \frac{1}{2} \paren{2\lambda + q }^q }  =   2^{\frac{-1}{2}} \paren{2\lambda + q }^{q/2} .
\end{align*}
Finally, one concludes
\begin{align*}
\E\croch{ \abs{ X - Y}^q } & \leq
 2^{q/2+2} \sqrt{ e } \sqrt{q} \sqrt{\frac{q}{e}}^{q}     2^{\frac{-1}{2}} \paren{2\lambda + q }^{q/2} <
 2^{q/2+1} e \sqrt{q} \croch{\frac{q}{e} \paren{2\lambda + q} }^{q/2} .
\end{align*}

\end{proof}

\end{document}